\documentclass{article}
\usepackage[utf8]{inputenc}
\usepackage[normalem]{ulem}
\usepackage[title]{appendix}\usepackage{authblk}
\usepackage{hyperref}

\usepackage{amssymb} 

\usepackage{enumitem}

\usepackage{amsmath}
\usepackage{amsfonts}
\usepackage{amssymb}
\usepackage[english]{babel}
\usepackage{graphicx}
\usepackage{blindtext}
\usepackage{textcomp}
\usepackage{pgfplots}
\usepackage{mathtools}
\usepackage{graphicx}
\usepackage{geometry}
\usepackage{hyperref}
\usepackage{mathrsfs}
\usepackage{centernot}
\usepackage{color}
\usepackage{multicol}
\usepackage{xcolor}
\usepackage{float}

\usepackage[nottoc,numbib]{tocbibind}

\pgfplotsset{width=10cm, compat=1.9}

\usepackage{fancyhdr}
\fancypagestyle{plain}{ %
\fancyhf{} 
}

\usepackage{amsthm, amssymb}
\newcommand{\vertiii}[1]{{\left\vert\kern-0.25ex\left\vert\kern-0.25ex\left\vert #1 
    \right\vert\kern-0.25ex\right\vert\kern-0.25ex\right\vert}}

\numberwithin{equation}{section}

\newtheorem{theorem}{Theorem}
\newtheorem{lemma}[theorem]{Lemma}
\newtheorem{proposition}[theorem]{Proposition}
\newtheorem{corollary}[theorem]{Corollary}

\theoremstyle{definition}
\newtheorem{definition}[theorem]{Definition}

\theoremstyle{definition}

\newtheorem{remark}[theorem]{Remark}
\numberwithin{theorem}{section}

\def\a{\alpha}

\def\R{\mathbb R}
\def\N{{\mathbb N}}

\def\Z{\mathbb Z}
\def\T{\mathbb T}

\def\A{\mathcal A}
\def\C{\mathbb C}

\def\({\biggl(}
\def\){\biggr)}
\def\<{\mathbf{\langle}}
\def\>{\mathbf{\rangle}}

\newcommand{\Meng}[2]{\left\{#1\mathrel{}\middle|\mathrel{}#2\right\}}
\allowdisplaybreaks

\makeatletter
\def\blfootnote{\xdef\@thefnmark{}\@footnotetext}
\makeatother

\title{Weak mixing behavior for the projectivized derivative extension\blfootnote{Keywords: Smooth ergodic theory, approximation-by-conjugation method, almost isometries, weakly mixing diffeomorphisms, analytic approximation, projectivization of tangent bundle.}
\blfootnote{2020 Mathematics Subject Classification:	Primary 37C40; Secondary 37A05, 53C99, 57R50.}}
\author{Shilpak Banerjee\footnote{Department of Mathematics and Statistics, Indian Institute of Technology Tirupati, Yerpedu, Andhra Pradesh 517619, India, E-mail: banerjee.shilpak@gmail.com},\; Divya Khurana\footnote{Department of Mathematics, Indraprastha Institute of Information Technology Delhi (IIIT-Delhi), Okhla Industrial Estate Phase 3, New Delhi 110020, E-mail: divyak@iiitd.ac.in. The author is supported by the University Grants Commission (UGC-JRF), India.},\; Philipp Kunde\footnote{Oregon State University, Department of Mathematics, Corvallis, OR 97331, USA, E-mail: pkunde.math@gmail.com. The research of P.K. is part of the project No. 2021/43/P/ST1/02885 co-funded by the National Science Centre
		and the European Union's Horizon 2020 research and innovation programme under the Marie Sklodowska-Curie
		grant agreement no. 945339.}}
        
\graphicspath{{pic/}}

\begin{document}
\maketitle
\lhead{}
\rhead{\thepage}

\begin{abstract}
In both smooth and analytic categories, we construct examples of diffeomorphisms of topological entropy zero with intricate ergodic properties. On any smooth compact connected manifold of dimension $2$ admitting a nontrivial circle action, we construct a smooth diffeomorphism whose differential is weakly mixing with respect to a smooth measure in the projectivization of the tangent bundle. In case of the $2$-torus, we also obtain the analytic counterpart of such a diffeomorphism. The constructions are based on a quantitative version of the ``Approximation by Conjugation'' method, which involves explicitly defined conjugation maps, partial partitions, and the adaptation of a specific analytic approximation technique.
\end{abstract}

\section{Introduction}
The primary objective of this paper is to provide examples of smooth and analytic volume-preserving diffeomorphisms on a manifold $M$ which demonstrate intricate dynamics within the projectivization $\mathbb{P} \mathrm{T}M$ of the tangent bundle of $M$. Our approach relies on the `approximation by conjugation method', also known as the `Anosov-Katok method', which was initially introduced by D.~Anosov and A.~Katok in 1970 \cite{AK}. 
This method allows the construction of transformations with predetermined ergodic or topological properties on manifolds admitting a nontrivial circle action. For instance, \cite{AK} resolved a long-standing question concerning the existence of an ergodic area-preserving smooth diffeomorphism on the disc $\mathbb{D}$.

This approach is still one of the few methods available for constructing diffeomorphisms within the smooth or real-analytic category that satisfy a prescribed specific set of ergodic or topological properties. The produced diffeomorphisms can exhibit an exotic combination of ergodic and topological properties. For instance, the existence of minimal but not uniquely ergodic diffeomorphisms in the smooth category and non-ergodic generic measures for smooth dynamical systems have been explored in \cite{Win} and \cite{DK23}, respectively. The Anosov-Katok method is also the most powerful tool for non-standard smooth realization of measure-preserving transformations \cite{AK}, \cite{Be13}, and \cite{FSW}. Recently, this technique played a significant role in the proof of anti-classification theorems for smooth measure preserving diffeomorphisms  \cite{FW19}, \cite{FW1}, \cite{FW22}. Here, one of the crucial steps in the proof of smooth anti-classification theorems is the smooth realization of a class of symbolic systems (the so-called \emph{circular systems}) obtained by the Anosov-Katok method.

In their foundational paper \cite{AK}, Anosov and Katok provided, among other things, examples of weakly mixing diffeomorphisms in the space $$\A(M) = \overline{\{h\circ S_t\circ h^{-1}: t\in \T^1, h \in \textup{Diff}^{\infty}(M,\mu)\}}^{C^{\infty}},$$ on any manifold admitting a non-trivial $\T^1$ action $\{S_t\}_{t \in \T^1}$. Later Fayad and Saprykina in \cite{FS} constructed smooth weakly mixing diffeomorphism in the restricted space \begin{align}\label{eqn:1.1}
\A_{\alpha}(M) = \overline{\{h\circ S_{\alpha}\circ h^{-1}: h \in \textup{Diff}^{\infty}(M,\mu)\}}^{C^{\infty}},
\end{align}
for any Liouville number $\alpha$. Here, we recall that an irrational number $\alpha$ is Liouville if for each $n$, there exist integers $p$ and $q>1$ such that $0<|\alpha-\frac{p}{q}|<\frac{1}{q^n}$. Both constructions utilize the approximation by conjugation method: Each diffeomorphism is obtained as the limit of a sequence $f_n= H_n\circ S_{\alpha_{n+1}}\circ H_n^{-1}$ with $\alpha_{n+1}\in \mathbb{Q}$ and $H_n=h_1\circ\ldots \circ h_n$, where $h_n$ is a volume-preserving diffeomorphism satisfying $S_{{\alpha}_n}\circ h_n= h_n\circ S_{{\alpha}_n}$. To guarantee convergence of the diffeomorphisms $f_n$ in the space $\A_{\alpha}(M)$ for a prescribed $\alpha$, one needs the construction of more explicit conjugation maps $h_n$ and more precise norm estimates in comparison to convergence proofs in the space $\A(M)$. 

In this article, our objective is to explore ergodic properties of the diffeomorphism along with its differential map with respect to a smooth measure in the projectivization of the tangent bundle. The derivative extension of a smooth map, $f: M\rightarrow M$, on a $m$-dimensional manifold $M$ is a transformation defined on the tangent bundle $\mathrm{T}M$ denoted as $(f,\mathrm{d}f)$. The projective tangent bundle $\mathbb{P} \mathrm{T}M$ is obtained by identifying the tangent space of a point $p\in M$, $T_pM$, with $\R^m$, and considering its projective space $\mathbb{P} \R^m$. It is important to note that a measurable Riemannian metric for a diffeomorphism induces an invariant measure $\bar{\mu}$ for the projectivized derivative extension, which is absolutely continuous on the fibers (see Section~\ref{subsec:projDer}). Such a diffeomorphism, which preserves a measurable Riemannian metric and an absolutely continuous probability measure, is called an IM diffeomorphism; see \cite[section 3]{GKa00} for a description of IM diffeomorphisms and IM group actions.

In {\cite{GKa00}}, R.~Gunesch and A.~Katok constructed a weakly mixing volume-preserving smooth diffeomorphism that preserves a measurable Riemannian metric. For every Liouville number $\alpha$, such diffeomorphisms are dense in the restricted space $\mathcal{A}_{\alpha}(M)$ in the $C^{\infty}$ topology {\cite{GKu15}}. In the analytic category, the third author constructed a real analytic weakly mixing diffeomorphism $f\in\textup{Diff}_{\rho}^{\omega}(\T^m,\mu)$ preserving a measurable Rienmannian metric {\cite{Ku17}}.

It is a natural question to investigate the ergodic properties of the projectivized derivative extension with respect to the measure induced by the measurable invariant Riemannian metric. In fact, all the constructions described in the previous paragraph are as non-ergodic
as possible: Their projectivized derivative extensions are isomorphic to the direct product of the diffeomorphism in the base with the trivial action in the fibers, resulting in each ergodic component intersecting almost every fiber at a single point. In \cite{Ku20}, the third author was able to realize the other
extreme possibility by constructing smooth diffeomorphisms whose differential is ergodic with respect to
such a smooth measure in the projectivization of the tangent bundle. 

In this paper we study further ergodic
properties of the projectivized derivative extension and we focus on weak mixing behavior.  We recall that a measure-preserving system $(X,\mathcal{B},\mu, T)$ is said to be
	weakly mixing if there is no nonconstant function $h\in L^2(X,\mu)$ such that $h(Tx)=\lambda \cdot h(x)$ for some $\lambda \in \C$. Equivalently, $(X,\mathcal{B},\mu, T)$ is weakly mixing iff there exists an increasing sequence $\{m_n\}_{n\in \mathbb{N}}$
such that for any pair $A,B \in \mathcal{B}:$
 $ \mu(B\cap f^{-m_n}(A)) \longrightarrow \mu(B)\mu(A)$. We explore this aspect in both the smooth and analytic settings. Hereby, we solve Problems~6.2 and~10.7 in~\cite{Ku_book24}. Our approach bases upon a new weak mixing criterion specifically tailored for the projectivized derivative extension. It requires the construction of appropriate smooth maps on the projectivized tangent bundle that yield a suitable distribution of partition elements. We refer to Section~\ref{subsec:outline} for an overview of our construction. Here, we present the smooth version of our main result.

\begin{theorem}\label{thm:1.1}
Let $M$ be a smooth compact and connected  manifold of dimension $2$ with a non-trivial circle action $\mathcal{S}= \{S_t\}_{t \in \R}, S_{t+1}= S_t,$ preserving a smooth volume $\mu.$ If $\alpha\in \R$ is Liouville, there exists a measure-preserving weakly mixing diffeomorphism $f \in \A_{\alpha}(M),$ whose projectivized derivative extension $(f, \mathrm{d}f)$ is weakly mixing with respect to the measure $\bar{\mu}$ in the projectivized tangent bundle $\mathbb{P} \mathrm{T}M$ which is absolutely continuous in the fibers.  
\end{theorem}
The Anosov-Katok constructions are renowned for their flexibilty. In this article, we have discussed several instances exploring intricate dynamics within the smooth category. However, challenges emerge in the real-analytic category, as emphasized in \cite[Section 7.1]{FK} and \cite[Section 6.3]{BaKu}: Maps with notably large derivatives in the real domain or their inverses tend to display singularities in a small complex neighborhood. For a real analytic family $S_t, 0 \leq t \leq  t_0,$ with $S_0 = id,$ the family $h \circ S_t \circ h^{-1}$ is expected to display singularities very proximate to the real domain for any $t > 0.$ Consequently, the domain of analyticity for maps in the form of $f_n = H_n\circ S_t\circ H_n^{-1}$ diminishes progressively during each construction step, leading to a limiting diffeomorphism that is not analytic. Hence, there is a need to identify specific conjugation maps that can be explicitly inverted, ensuring both the map and its inverse maintain analyticity across a broad complex domain. 

B.~Fayad and M.~Saprykina successfully constructed examples of volume-preserving weakly mixing real-analytic diffeomorphisms on $\T^2$ using highly explicit analytic conjugation maps \cite{FS}. Additionally, B.~Fayad and A.~Katok devised uniquely ergodic examples on any odd-dimensional sphere in \cite{FaKa13}. Recently, G.~Farre in \cite{Farre23} combined both approaches to find weakly mixing real-analytic diffeomorphisms on odd-dimensional spheres. On the other hand, P.~Kunde and S.~Banerjee  developed a more general approximation technique utilizing block-slide type maps \cite[Theorem 5]{BaKu} on the finite-dimensional torus. This approach provided sufficient flexibility to replicate the construction process and obtain numerous analytic counterparts of Anosov-Katok constructions, as seen in \cite{BaKu}. The tool of block-slide type maps is also used in the aforementioned construction of real-analytic weakly mixing diffeomorphisms preserving a measurable Riemannian metric \cite{Ku17}. However, the derivatives of such block-slide type maps are the identity on large parts of the space. Hence, this analytic approximation technique is unable to yield a specific distribution for the elements in the space and tangent direction, thus hindering the achievement of ergodicity for the projectivized derivative extension in the analytic category.

In this article, we have achieved an even stronger property than ergodicity- the weak mixing property- for the projectivized derivative extension in the analytic category, specifically for the torus case. We modified an analytic approximation scheme, as discussed in \cite[Theorem 1.8]{Be22}, where P.~Berger introduced a specific technique for approximating a smooth transformation within the infinite annulus $[0,1]\times \R$ by closely approximating it with an analytic transformation under certain conditions. This technique is built on the generators of groups of Hamiltonian maps and their flows, employing specific conditions. This distinctive approach plays a crucial role in our analytic construction in Section~\ref{sec:4:4.a} which involves adapting the approximation technique for the torus case. This requires exercising sufficient control over the proximity between smooth and analytic conjugation maps, and establishing the convergence of the Anosov-Katok method while maintaining the desired dynamical properties, as outlined in Section 3, within our setup.

\begin{theorem}\label{thm:1.2}
    There exists a real analytic diffeomorphism $\hat{f} \in \textup{Diff}_{\infty}^{\omega}(\mathbb{T}^2,\mu),$ whose projectivized derivative extension $(\hat{f}, \mathrm{d}\hat{f})$ is weakly mixing with respect to the measure $\bar{\mu}$ in the projectivized tangent bundle $\mathbb{P}T\mathbb{T}^2$. Additionally the complexification of this $\hat{f}$ is entire. 
\end{theorem}
These constructions represent the only known instances of volume-preserving diffeomorphisms whose differentials are weakly mixing with respect to a smooth measure in the projectivization of the tangent bundle, applicable to both smooth and analytic categories.

 Regarding the successful implementation of analytic AbC schemes, we mention the recent breakthrough work \cite{Be24} by P.~Berger. Based on a so-called \emph{Anosov-Katok principle} it enables the realization, by analytic symplectomorphisms on annulus, sphere or disc, of
dynamical properties which are $C^0$-realizable by the AbC method. However, the current results only demonstrate that an Anosov-Katok principle holds for transitivity. Therefore, this approach does not directly allow the realization of weakly mixing analytic diffeomorphisms, let alone properties for the projectivized derivative extension.

\textbf{Acknowledgement}: D.~Khurana thanks Jagiellonian University, Krakow, for their hospitality and the NBHM for supporting international travel, where the initial idea of the paper was developed.

\section{Preliminaries}
\subsection{A primer on symplectic dynamics}

A \emph{symplectic form} $\omega$ on a smooth manifold $M$ is a closed non-degenerate diferential $2$-form. A smooth manifold equipped with a symplectic form is called a \emph{symplectic manifold}. 

\definition[Symplectomorphism]
A symplectomorphism $f:(M,\omega)\to(M,\omega)$ is a smooth diffeomorphism satisfying $f^*\omega=\omega$, where $f^*$ is the pullback of $f$. We denote the set of all symplectic diffeomorphsims of $M$ by $\text{Symp}(M,\omega).$ 

\definition[Hamlitonian Diffeomorphism]
A symplectomorphism $\phi \in \text{Symp}(M, \omega)$ is said to be a Hamiltonian diffeomorphism if there exists a Hamiltonian isotopy $h_t$ such that $\phi = h_1$. We say that $h_t$ is a Hamiltonian isotopy if there exists a smooth family of functions $H_t : M \rightarrow \mathbb{R}$ such that
$i_{(h_t)}\omega = \mathrm{d} H_t$. In other words, if there exist a smooth isotopy  $H:[0,1]\times M\rightarrow M$ which can be associated to the time dependent vector field  $X_{H_t}$ defined by $i_{(X_{H_t})}(\omega) = \mathrm{d} H_t$. Let  $\psi_t$ be the flow associated with the time dependent vector field  $X_{H_t}$.  A diffeomorphism $\phi$ is said to be
Hamiltonian if it is a time-1 map of a time-dependent Hamiltonian system i.e. $\phi = \psi_1$. 

Let $\text{Ham}^{\infty}(M,\omega)$ be the space of all smooth Hamiltonian diffeomorphism of $M$ and $\text{Ham}^{\omega}(M,\omega)$ be space of all entire Hamiltonian diffeomorphism of $M$.
For smooth volume form $\mu$,  $\text{Ham}^{\infty}(M,\mu)\subseteq \text{Diff}^{\infty}(M,\mu)$. 

\begin{lemma}\label{lem: 2.2.3a}
  Any smooth area-preserving diffeomorphism $f:[0,1]^2 \rightarrow [0,1]^2$ that acts as identity in the neighborhood of the boundary of $[0,1]^2$ is a  Hamiltonian diffeomorphisms on the torus $\T^2$.
\end{lemma}
\begin{proof}
Since the diffeomorphism $f: [0,1]^2 \to [0,1]^2$ acts as the identity close to the boundary of $[0,1]^2$, it can be extended to a diffeomorphism $\tilde{f}$ on $\mathbb{T}^2$ by identifying opposite edges of $[0,1]^2$. Note that any closed non-degenerate 2-form on a simply connected domain of $\R^n$, or a contractible symplectic manifold, is exact due to Poincaré's lemma \cite{Silva14}. Additionally, the exactness of the symplectic form on a contractible manifold implies the existence of a Hamiltonian vector field associated with any closed $2$-form. Furthermore, Moser's theorem guarantees a continuous family of area-preserving diffeomorphisms $F_t: \mathbb{T}^2 \to \mathbb{T}^2$ such that $F_1 = \tilde{f}$ and $F_0$ represents the identity map on $\mathbb{T}^2$, implying that $\tilde{f}$ is isotopic to the identity through area-preserving diffeomorphisms on the torus.

Since $[0,1]^2$ is a contractible domain, the diffeomorphism $f$ within $[0,1]^2$, when extended to the torus as $\tilde{f}$, is indeed a Hamiltonian diffeomorphism on the torus, following from Poincaré's lemma and Moser's theorem.
\end{proof}
\remark The same argument extends the result to $n$-dimensional torus, as any smooth area-preserving map $f:[0,1]^n\rightarrow [0,1]^n$ that acts as the identity near the boundary is a Hamiltonian diffeomorphism on $\T^n.$
\subsection{Projectivized derivative Extension} \label{subsec:projDer}
We refer to \cite[section 2.1]{GKu15} for useful definitions and notations. Subsequently, we introduce the concept of the invariant measure for the projectivized derivative extension. 
We consider the derivative extension of a smooth diffeomorphism, $f: M \rightarrow M$, as a transformation defined on the tangent bundle $TM,$ and denoted as $(f, \mathrm{d}f)$. For any point $p \in M$, we identify the tangent space $T_pM$ with $\mathbb{R}^2$ in the case of a 2-dimensional manifold. Subsequently, we consider its projective space $\mathbb{P}\mathbb{R}^2$, which is diffeomorphic to the circle $\T^{1}$ and introduce the notation $[a, b] \subset \mathbb{P}\mathbb{R}^2$ to describe the allowed values for the spherical coordinate $\phi \in \mathbb{R} / \pi \mathbb{Z}$. This results in the formation of the projectivized tangent bundle denoted as $\mathbb{P}\mathrm{T}M$. We consider the projectivized derivative extension of a diffeomorphism $f : M \rightarrow M$ on the projectivized tangent bundle $\mathbb{P}\mathrm{T}M,$ denoted as $(f,\mathrm{d}f)$ again by abusing the notation. 
We employ the notation $c \times [a, b] \subset \mathbb{P}\mathrm{T}M,$ with $c \subset M,$ to represent the subset in $\mathbb{P}\mathrm{T}M$ comprising base points $x \in c$ and corresponding spherical coordinates $\phi \in [a, b]$.
\subsubsection{Existence of invariant measure on the \texorpdfstring{$\mathbb{P}\mathrm{T}M$:}{lg}}\label{sec:2.2.1}
The existence of an invariant measure is explored following the framework detailed in \cite[chapter~5.1]{Ch97}. Initially, we consider the cotangent bundle $TM^*$ along with the projection maps $\pi_1: TM \rightarrow M$ and $\pi_2: TM^* \rightarrow M$. Subsequently, we define the canonical 1-form $\omega$ on $TM^*$ by $\omega|_\tau = \pi_2^{*}\tau$, where $\omega|_\tau$ denotes 1 form $\omega$ 
evaluated at $\tau \in TM^*$. Additionally, the canonical 2-form $\Omega$ on $TM^*$ is defined as $\Omega = \mathrm{d}\omega$, exhibiting symplectic properties.

Let $M$ be a Riemannian manifold, and $V: M \rightarrow \mathbb{R}$ be a function. We examine the Lagrangian $L: TM \rightarrow \mathbb{R}$ given by $L(\zeta) = |\zeta|^2 - V \circ \pi_1(\zeta)$, where $|\zeta|$ is calculated using the Riemannian metric. Associated with this Lagrangian, we define a bundle map $F_L: TM \rightarrow TM^*$ defined by $F_L(\zeta)(\eta) = \frac{d}{dt}(L(\zeta + t\eta))|_{t=0}$ for $p \in M,$ and $\zeta, \eta \in T_pM$. Subsequently, we define $\Theta = F_L^*\Omega$ and $\nu = F_L^*\omega$.

In \cite[Chapter 5.1]{Ch97}, they examine the differential form $\nu \wedge \Theta$ on the unit tangent bundle $SM$, which is proven to be locally a product, up to a constant multiple, of the Riemannian volume on $M$ with the Lebesgue 1-form on the unit tangent spheres of $M$ with respect to the Riemannian metric. Particularly, for any $\nu \wedge \Theta$-integrable function $g$ on $SM$, the ``integration over the fibers"
$$\int_{SM} g \nu \wedge \Theta = c \int_M \mathrm{d}\text{Vol}(p) \int_{S_pM} g|_{S_pM} \mathrm{d}\mu_p
$$
Here, $\text{Vol}$ represents the volume form induced by the Riemannian metric, and $\mu_p$ is the standard Borel measure on the tangent sphere $S_pM$ concerning the Riemannian metric.
Applying the same approach, we deduce a similar formula for the constructed invariant measurable Riemannian metric $\omega_{\infty}$ and for any integrable function on $\mathbb{P}\mathrm{T}M$. The corresponding measure will be denoted by $\bar{\mu}$. Furthermore, the measure induced by the measurable Riemannian metric $\omega_{\infty}$ in our explicit construction aligns with the measure $\mu$ on $M$. As $\omega_{\infty}$ is $f$-invariant, we infer that $\bar{\mu}$ is $(f, \mathrm{d}f)$-invariant.
\subsection{Analytic topology}
Real analytic diffeomorphisms of $\T^2$ homotopic to the identity have a lift of type 
$$F(x_1,x_2) = (x_1 + f_1(x_1,x_2), x_2 + f_2(x_1,x_2)),$$ where the function $f_i: \R^2 \rightarrow \R$ are real analytic and $\Z^2$ periodic for $i=1,2.$ Any real-analytic $\Z^2$-periodic function defined on $\R^2$ can be extended as a holomorphic (complex analytic) function from some complex neighbourhood of $\R^2$ in $\C^2$. 
For any fixed $\rho > 0$, we define the complex domain $\Omega_{\rho}$ as follows: 
\begin{align}\label{eqn:2:2.3a}
\Omega_{\rho}= \{(z_1,z_2) \in \C^2 \ :\ |\Im{z_i}| < \rho , \ i =1, 2\}.
\end{align}
Now, for a function $f$ defined on this set, we define the norm as $|f|_{\rho} = \sup_{z\in \Omega_{\rho}}|f(z)|.$

Denote the space $C_{\rho}^{\omega}(\mathbb{T}^2)$ as the space of all $\mathbb{Z}^2$-periodic real-analytic functions on $\mathbb{R}^2$ that extend to holomorphic functions on $\Omega_{\rho}$, with $|f|_{\rho}<\infty$. Let $\textup{Diff}_{\rho}^{\omega}(\mathbb{T}^2, \mu)$ be the set of all measure-preserving real-analytic diffeomorphisms of $\mathbb{T}^2$ homotopic to the identity, whose lift $F$ to $\mathbb{R}^2$ extends to a holomorphic functions on $\Omega_{\rho}$. 

\definition For two functions $f, g\in \textup{Diff}_{\rho}^{\omega}(\T^2, \mu)$, we define the norm $ | f |_{\rho}  $ and metric $d_{\rho}^{\omega}(f,g)${\footnote{ Here, we denote the metric $d_r$ for $C^r$ and $d_\rho^{\omega}$ for $C^{\omega}_{\rho}$ topology.}}  as follows:
\begin{align*}
    |f|_{\rho} = \max_{i=1,2}|f_i|_{\rho} \quad  \text{and} \quad
 d_{\rho}^{\omega}(f,g)= \max_{i=1,2}\{\inf_{p\in \Z} |f_i-g_i +p|_{\rho}\}.
 \end{align*}
 Moreover, for a diffeomorphism $T$ with lift $\widetilde{T}(x_1,x_2)= (T_1(x_1,x_2),T_2(x_1,x_2))$ we define 
$$|DT|_{\rho} = \max \left\{ \left| \frac{\partial T_i}{ \partial x_i} \right|_{\rho} \ :  \ i,j = 1,2 \right\}.$$
Additionally, define $\text{Ham}_{\rho}^{\omega}(\T^2, \mu)$ as a subset of $\textup{Diff}_{\rho}^{\omega}(\T^2, \mu)$, containing all the real analytic Hamiltonian diffeomorphisms of $\T^2.$  Next, we denote $\textup{Diff}^{\omega}_{\infty}(\T^2,\mu) \subseteq \textup{Diff}_{\rho}^{\omega}(\T^2,\mu)$ for any $\rho>0$, be the space containing all the real-analytic measure-preserving diffeomorphism of $\T^2$ whose lift $F$ to $\R^2$ extends to an entire function on $\mathbb{C}^2$. Analogously, define $\text{Ham}_{\infty}^{\omega}(\T^2, \mu)$ as a subset of $\textup{Diff}_{\infty}^{\omega}(\T^2, \mu)$, containing all the real analytic Hamiltonian diffeomorphisms of $\T^2.$ 
\subsection{Smooth topology}
Smooth diffeomorphisms of $\T^2$ homotopic to the identity have a lift of type
$$\widetilde{F}(x_1, x_2) = (x_1 + f_1(x_1, x_2), x_2 + f_2(x_1, x_2) ),$$
where $f_i: \mathbb{R}^2 \longrightarrow \mathbb{R}$ are $\mathbb{Z}^2$ periodic for $i= 1,2$. We can use the smooth topology of $\mathbb{R}^2$ defined by the norms as follows:
For any continuous function $f:(0,1)^2 \longrightarrow \mathbb{R},$ the norm is given by
$$||f||_0 := \sup_{z\in (0,1)^2}|f(z)|.$$
The partial derivative of a function is denoted as follows: for $(a_1,a_2)\in \mathbb{N}^2,$ where $|a|= a_1+a_2,$
$$\mathrm{D}_a := \frac{\partial^a}{\partial_{r}^{a_2}\partial_{\theta}^{a_1}}.$$

Consider $\textup{Diff}^k(\mathbb{T}^2)$ be the space of $k$-differentiable diffeomorphisms of the torus. For any $F,G \in \textup{Diff}^k(\mathbb{T}^2)$, denote their lifts by $\widetilde{F}$ and $\widetilde{G}$. For mappings $F : \mathbb{R}^2\longrightarrow \mathbb{R}^2$, let $F_i$ represent the i-th coordinate function, and denote 
$$\vertiii{F}_k: = \max \{ \|\mathrm{D}_aF_i\|_0, \|\mathrm{D}_a(F)^{-1}_i\|_0 \ | \ i=1,2, \ 0\leq |a| \leq k \}.$$
\definition Define the distances between two diffeomorphisms $F, G \in \textup{Diff}^k(\mathbb{T}^2):$ 
$$ \tilde{d}_0(F,G)=\max_{i=1,2}\{ \inf_{p\in \mathbb{Z}} ||(\widetilde{F}_i - \widetilde{G}_i) +p||_0\},$$
$$\tilde{d}_k(F,G)=\max \{\tilde{d}_0(F,G),|| \mathrm{D}_a(\widetilde{F}_i - \widetilde{G}_i)||_0 | \ i=1,2, \  1 \leq |a| \leq k\}.$$
We will employ a metric that measures the distance between diffeomorphisms and their inverses:
$$ d_k(F,G)=\max \{ \tilde{d}_k(F,G), \tilde{d}_k (F^{-1} ,G^{-1}) \}. $$
In the smooth topology on $M$, a sequence of diffeomorphisms in $\textup{Diff}^{\infty}(M)$
 is considered convergent if it converges in 
$\textup{Diff}^{k}(M)$ for all k.  The space $\textup{Diff}^{\infty}(M)$ is equipped with the metric
$$ d_{\infty}(F,G) = \sum_{k=1}^{\infty} 2^{-k} \frac{d_k (F,G)}{1 +d_k (F,G)}. $$
It is a compact metric space, and the Baire theorem holds for any of its closed subspaces.

\subsection{Premiliary Lemmas}
The following fact will be useful for our norm estimates of conjugation maps.
\begin{lemma}\label{le:2a}
Let $F,G \in\textup{Diff}^{\infty}(\mathbb{T}^2)$. For $k \in \mathbb{N},$  the norm estimates of the  composition $F \circ G$ satisfy
\begin{align}
    \vertiii{F\circ G}_k \leq C \vertiii{F}_k^k\cdot\vertiii{G}_k^k,
\end{align}
where $C$ is constant only depending on $k$.
\end{lemma}
\begin{proof}The above can be deduced using the Faa di Bruno formula; a similar proof is presented in \cite[Lemma 4.1]{Ku15}.
\end{proof}
In Section~\ref{sec:crit_weak_mixing} we will state a criterion for weak mixing expressed on suitable partial partitions.
\begin{definition}[Partial Partition] A collection $\zeta_n$ of disjoint sets on $\mathbb{T}^2$ is called a partial partition of $\mathbb{T}^2$. A sequence of partial partitions $\zeta_n$ converges to the decomposition into points (notation: $\zeta_n \rightarrow \varepsilon$) if, for any measurable set $A$, there exists a measurable set $A_n$ for any $n$, which is a union of elements of $\zeta_n$, such that $\lim_{n\rightarrow \infty} \mu(A \triangle A_n) = 0$ (here, $\triangle$ denotes the symmetric difference).
\end{definition}
The next lemma will allow us to verify the weak mixing property for a limit diffeomorphism from approximating maps.
 \begin{lemma}\label{lem:8:8.2}Let $\varepsilon>0$ and $f,g\in \textup{Diff}^{\infty}(\T^2,\mu)$ with $d_0(f,g)< \varepsilon$. Moreover, let $\eta$ be a partial partition of $\T^2$. Consider two subsets $A, B \subset \T^2$ such that $A\subset B$, and $\textit{dist}(\partial A, \partial B) > \varepsilon$, where $\text{dist}(A,B)=\inf_{x\in A,y \in B}\inf_{k\in \Z^2}\|x-y+k\|,$ where $\|\cdot\|$ denotes the Euclidean norm, and $\partial A $ denotes the boundary of $A$. Then the following relations hold:
        \begin{enumerate}
        \item If $f(x) \in A $, then $g(x)\in B,$ and thus $\mu(f(I)\cap A) \leq \mu(g(I)\cap B)$ for any $I\in \eta.$
      \item If $g(x) \in A$, then $f(x) \in B$, and thus $\mu(g(I) \cap A) \leq \mu(f(I) \cap B)$ for any $I\in \eta$.
      \end{enumerate}
      \end{lemma}
\begin{proof} The proof can be derived directly from the definitions.
\end{proof}
\subsection{Approximation by conjugation method}
Here, we outline the scheme of the approximation by conjugation method introduced in \cite{AK}, which allows for the construction of smooth area-preserving diffeomorphisms with specific ergodic properties. We present the following method specifically for the case of $\mathbb{T}^2$, but it is applicable to any compact connected smooth manifold $M$ with a non-trivial circle action. Let us denote by $R_t$ the measure-preserving smooth circle action $\mathbb{T}^1$ on the torus $\mathbb{T}^2= \mathbb{R}/\mathbb{Z}\times \mathbb{R}/\mathbb{Z}$ defined by translation $t$ in the first coordinate: $R_{t}(\theta,r)= (\theta+t,r).$
The required map $f$ is constructed as the limit of a sequence of periodic measure preserving diffeomorphism $f_n$ in the smooth topology. The sequence of $f_n$ is defined by induction as
     \begin{align}\label{eq:1d}
     f_n = H_n\circ R_{\alpha_{n+1}}\circ H_n^{-1},
     \end{align}
     where $\alpha_{n+1}= \frac{p_{n+1}}{q_{n+1}}\in \mathbb{Q}/\mathbb{Z}$ and $H_n\in \textup{Diff}^{\infty}(\mathbb{T}^2,\mu)$.
   The diffeomorphism $H_n$ 
    is constructed successively as $H_n = h_1\circ \ldots \circ h_n,$ where $h_n$ is an area preserving diffeomorphism of $\mathbb{T}^2$ that satisfies 
    \begin{align}\label{eq:2d}
    h_n\circ R_{\alpha_{n}}= R_{\alpha_{n}}\circ h_n.
    \end{align}
    The rationals $\alpha_{n+1}$ are chosen close enough to $\alpha_{n}$ to ensure closeness between $f_n$ and $f_{n-1}$ in the $C^{\infty}$ or $C^{\omega}$ topology. Given $\alpha_{n+1}, H_n$, at the $n+1$ stage of this iterative process, we construct $h_{n+1}$ such that $f_{n+1}$ satisfy an approximative finitary version of the specific property we eventually need to achieve for the limiting diffeomorphism. The explicit construction of $h_{n+1}$ for our purpose is exhibited in Section~\ref{sec:constr}.
    Then we construct $\alpha_{n+2}$ close enough to $\alpha_{n+1}$ by choosing some parameters $k_{n+1}\in \mathbb{N}$ and $l_{n+1}\in \mathbb{N}$ sufficiently large such that they satisfy certain conditions \ref{item:P1}-\ref{item:P2}
    and guarantee the convergence of iterative sequence $(f_n)_{n\in \N}$ in the smooth or real analytic topologies, respectively. 
The limit obtained from this iterative sequence is the required smooth or real analytic diffeomorphism with the specific ergodic properties.

Additionaly, we employ the quantitative version of this method to construct the diffeomorphism with a prescribed Liouvillean rotation number $\alpha$. This involves ensuring that the sequence $\alpha_n$ converges to $\alpha$ and the limiting transformation $f\in \A_{\alpha}(\T^2)$, as discussed in \cite{FS}.
\subsection{Outline of the proofs} \label{subsec:outline}
Similar to the constructions in \cite{GKu15}, we introduce two sequences of partial partitions, $\eta_n$ and $\zeta_n$, of the torus, both converging to decompositions into points. The partition $\eta_n$ will satisfy the requirements to demonstrate the weak mixing property, and the other partition $\zeta_n$ consists of even more refined partition elements. The conjugation map $h_n$ acts as an isometry on the partition elements of $\zeta_n$, allowing the construction of an invariant measurable Riemannian metric. Subsequently, the conjugation diffeomorphism $h_n = g_n \circ \phi_n$ is introduced. Compared to \cite{GKu15} and \cite{Ku20}, modifications are made to $g_n$ and $\phi_n$ to establish the weak-mixing property of the projectivized derivative extension. Specifically, $g_n$ introduces shear in the $\theta$ direction, while still acting as an isometry in the $\phi_n$-image of any partition element $I \in \zeta_n$. Meanwhile, $\phi_n$ acts as an isometry on the elements $I \in \zeta_n$ and simultaneously arranges the elements of $\eta_n$ to meet weak mixing criteria. The map $\phi_n$ is defined as $i_n \circ \tilde{\phi}_n$ in half of the fundamental domain of the torus and as an identity on the other half. Here, $\tilde{\phi}_n$ acts as an isometry and distributes the $\eta_n$ partition elements as required by weak mixing criteria on the manifold. The map $i_n$ is a composition of translation and different rotations in different sections of the domains in $\zeta_n$ to achieve good distribution behavior in the tangent space, as described in Section \ref{sec:constr}. 

Compared to \cite{GKu15}, we introduce a weak mixing criterion for the projectivized derivative extension, with a sequence of partial partitions $\hat{\eta}_n$ for the space $\mathbb{P} \mathrm{T}M$ converging to point decomposition. This criterion is based on the notion of $(\gamma,\delta,\varepsilon_1, \varepsilon_2)$ distribution of the map $(\Phi_n,\mathrm{d} \Phi_n) = (\phi_n,\mathrm{d}\phi_n)\circ (R_{\alpha_{n+1}},\mathrm{d}R_{\alpha_{n+1}})^{m_n}\circ (\phi_n,\mathrm{d}\phi_n)^{-1}$ with a specific choice of sequence $(m_n)_n\in \mathbb{N}$. This criterion is based on \cite{GKu15}, but adapted for the derivative extension $(\Phi_n,\mathrm{d} \Phi_n)$ in Section \ref{sec:crit_weak_mixing}.

Furthermore, in the context of analytic constructions, we can closely approximate an analytic diffeomorphism to a smooth diffeomorphism by employing the approximation technique outlined in Sections \ref{sec:4:4.a}. In contrast to \cite{Ku17}, this approximation technique is specific to Hamiltonian diffeomorphisms and enables the decomposition of the smooth Hamiltonian diffeomorphism into a finite composition of one-dimensional linear flows, with each flow being approximated by analytic ones. These approaches ensure a precise level of approximation.
Hereby, we get the conjugation map $\hat{h}_n = \hat{g}_n \circ \hat{\phi}_n$, where $\hat{g}_n$ and $\hat{\phi}_n$ denote the analytic diffeomorphisms obtained by applying approximations to $g_n$ and $\phi_n$, respectively. This map $\hat{h}_n$ acts as an ``almost isometry" on the elements $I\in \zeta_n$, thereby establishing the existence of an invariant measurable Riemannian metric on $\mathbb{P} \mathrm{T}M$, as discussed in Section \ref{sec:metric}, following a similar strategy as in \cite{Ku17}. Subsequently, modifications are made to adapt the weak mixing criterion for the analytic projectivized derivative extension, as executed in Section \ref{sec:crit_weak_mixing}.

Finally, convergence of the approximation by conjugation scheme and the application of the weak mixing criterion are carried out in Sections  \ref{sec:conv} and \ref{sec:app_crit}, for both smooth and analytic constructions.

\section{Approximation by Conjugation schemes} \label{sec:constr}
\subsection{Smooth AbC scheme}\label{sec:3.3.1}
In our approximation by conjugation scheme in the smooth category we will inductively construct maps
\begin{equation}
    f_n=H_n\circ R_{\a_{n+1}}\circ H_n^{-1}, \ \text{where} \ H_n=h_1\circ\ldots \circ h_n \label{eq:3a}
\end{equation} 
with conjugation maps of the form
\begin{equation}
    h_n= g_n \circ \phi_n, \label{eq:3:3b}
\end{equation}
where the area-preserving diffeomorphisms $g_n$ and $\phi_n$ commute with $R_{\alpha_n}$ and will be explicitly constructed in Sections \ref{sec:conj_g_n} and \ref{sec:phi}, respectively. We will choose the sequence of rationals $\alpha_n = \frac{p_n}{q_n}$ converging to a prescribed Liouville number $\alpha$ in such a way that $\alpha_{n+1}$ is sufficiently close to $\alpha_n$ in order to guarantee convergence of the sequence $(f_n)_{n\in \N}$ to a limit diffeomorphism. In our explicit inductive construction process at step $n$, we have $H_{n-1}\in \textup{Diff}^{\infty}(\T^2,\mu)$ and $\alpha_1, \alpha_2,\ldots, \alpha_n\in \T^1$. Moreover, we are given some parameters $k_n\in \N$ and $l_n\in \N$ such that the following properties hold:
\begin{enumerate}[label={\bf(P\arabic*)}]
    \item\label{item:P1}
    The parameter $k_n\geq n^5$ satisfies that for every subset $ \hat{S} \subset \mathbb{P}\mathrm{T}M$ with $\text{diam}(\hat{S}) < \frac{3}{k_n}$, we have
     \begin{align}  \label{eqn:3.1.1a}
     \text{diam}\left((H_{n-1}, \mathrm{d}H_{n-1})(\hat{S})\right) \leq \frac{1}{n^2}.
    \end{align}
   \item\label{item:P0} The sequence $(k_{n})_{n \in \N}$ increases rapidly enough to guarantee, $
 \sum_{u=n}^{\infty} \frac{1}{k_{u}^5} \leq  \frac{1}{4k_n^4}$ for every $n \in \N. $
   \item\label{item:P3}  The parameter  $q_n$  is chosen to be large enough  such that $q_n^{0.25}>2 k_n$. Note that this is possible since the parameter $k_n$ is independent of $q_n$ and depends only on the conjugation map $H_{n-1}$.
    \item\label{item:P2} Growth conditions on $ l_n \geq 2 \cdot k_n^{10}\cdot q_n^2\cdot \|\mathrm{d} H_{n-1}\|_0,$ and $q_{n+1} > 2 k_n^{12} q_n^2.$ 
\end{enumerate}

\subsubsection{The conjugation map \texorpdfstring{$g_n$}{Lg}} \label{sec:conj_g_n} 
We introduce the smooth map $g_n$ as an approximation by suitably chosen step functions to the shear function $\tilde{g}_n(\theta,r)= (\theta+ br,r),$ where $b=[nq_n^{\sigma}]$ and $0.25 < \sigma < 0.5$. The purpose of the map $g_n$ is to introduce the appropriate kind of shear in the $\theta$ direction to achieve the weak mixing property on $M$. Additionally, the map $g_n$ must act as an isometry on the image $\phi_n(I_n)$ of all partition elements $I_n\in \zeta_n$, where $\zeta_n$ is a specific partial partition of $M$ defined in Section~\ref{subsubsec:zeta} below. This property will allow us to show that $f$ admits an invariant measurable Riemannian metric on the space, inducing an invariant measure on $\mathbb{P} \mathrm{T}M$ in our setup. This approach follows~\cite{GKu15}.
\begin{proposition}\label{lemma:3a} 
 Let $a_n={k_n^5}, \varepsilon_n = \frac{1}{2n^5k_n^{10}},$ and $b_n = [nq_n^{\sigma}]$ for $0.25<\sigma<0.5.$ Then there exists $g_n\in \text{Ham}^{\infty}(\T^2,\mu)$ such that
\begin{enumerate}
    \item $g_n\circ R_{\frac{1}{q_n}} = R_{\frac{1}{q_n}}\circ g_n$.
    \item For any $ j\in \{0,\ldots, a_n-1\},$ $g_n$ acts on $\T^1 \times [\frac{j+2\varepsilon_n}{a_n}, \frac{j+1- 2\varepsilon_n}{a_n}]$ as translation in the $\theta$-direction by $b_n\cdot\frac{j}{a_n}$.  
    \item $\vertiii{g_n}_r\leq c_{n,k_n,r}\cdot q_n^r,$ where the constant $c_{n,k_n,r}$ is independent of $q_n.$
\end{enumerate}
\end{proposition}
\begin{proof}
Let $a,b\in \mathbb{N}$ and $\varepsilon>0$ satisfying $\frac{1}{2\varepsilon}\in \mathbb{N}.$ Let $\rho: \mathbb{R}\rightarrow \mathbb{R}$ be a smooth increasing function that is equal to $0$ for $x\leq -1$ and takes the value $1$ for $x\geq 1$. Consider the map $\tilde{\psi}_{a,b,\varepsilon}:[0,1]\rightarrow \mathbb{R}$ as 
\begin{align}
   \tilde{\psi}_{a,b,\varepsilon}(x) = \frac{b}{a}\cdot \sum_{i=1}^{a-1}\rho\left(\frac{ax-i}{2\varepsilon}\right).
\end{align}
Observe that for any $0\leq j \leq a-1$, we have $\tilde{\psi}_{a,b,\varepsilon}|_{[\frac{j+2\varepsilon}{a}, \frac{j+1- 2\varepsilon}{a}]}  = b\cdot\frac{j}{a} \mod 1.$ Additionally, we can estimate $\|D^r\tilde{\psi}_{a,b,\varepsilon}\|_0\leq \frac{b\cdot a^{r-1}}{\varepsilon^r}\cdot \|D^r\rho\|_0.$ \\
We consider an antiderivative map  $\varrho_{a,b,\varepsilon}: \mathbb{T} \rightarrow \mathbb{R}$  where $\frac{\mathrm{d}}{\mathrm{d}r}\varrho_{a,b,\varepsilon}(r) = \tilde{\psi}_{a,b,\varepsilon}(r)$. Furthermore, we can introduce a symplectic vector field $$X_{H} = (\partial_{r}(H), -\partial_{\theta}(H)) = (\tilde{\psi}_{a,b,\varepsilon}, 0)$$
defined by the Hamiltonian $H \in C^{\infty}(\mathbb{T}^2, \mathbb{R})$ given by $H(\theta, r) = \varrho_{a,b,\varepsilon}(r)$. Finally, we obtain a Hamiltonian diffeomorphism $g_{a,b,\varepsilon}: \mathbb{T}^2 \rightarrow \mathbb{T}^2$  given by
\begin{align}
    g_{a,b,\varepsilon}(\theta, r) = 
    (\theta + \tilde{\psi}_{a,b,\varepsilon}(r)\text{ mod } 1, r).
\end{align}
as the time-1 map of the Hamiltonian $H(\theta, r) = \varrho_{a,b,\varepsilon}(r)$.

In our explicit construction, we will use $g_n = g_{a_n,b_n,\varepsilon_n},$ where $a_n= k_n^{5}$, $\varepsilon_n = \frac{1}{2n^5k_n^{10}}$ and $b_n = [nq_n^{\sigma}]$ for $0.25< \sigma < 0.5.$
 Observe that $g_n\circ R_{\frac{1}{q_n}}= R_{\frac{1}{q_n}} \circ g_n $  and  
$$\vertiii{g_n}_r \leq \frac{b_n\cdot a_n^{r-1}}{\varepsilon_n^r} \cdot \|D^r\rho\|_0 \leq 2[nq_n^{\sigma}]\cdot k_n^{3(r-1)+2r}\cdot q_n^{r-1}\cdot \|D^r\rho\|_0 \leq c_{n,k_n,r}\cdot q_n^r,$$
where $c_{n,k_n,r}$ is a constant that depends only on $n,k_n$ but 
is independent of $q_n.$
\end{proof} 
 We define the ``good domain" of $g_n$ as 
\begin{align}\label{eqn:gooddomain_g_n}
\mathcal{G}_n = \bigcup_{j\in \{0,\ldots, a_n-1\}} \T^1 \times \left[\frac{j+2\varepsilon_n}{a_n}, \frac{j+1- 2\varepsilon_n}{a_n}\right].
\end{align}
By the second part of Proposition~\ref{lemma:3a} the map $g_n$ acts as a translation on $\mathcal{G}_n$ and, hence, its differential map $\mathrm{d}_pg_n= \text{id}$ for any base point $p \in \mathcal{G}_n$.

\subsection{Analytic AbC scheme} \label{sec:4:4.a}
On the torus $\T^2$ we execute an approximation by conjugation method in the analytic category, where at step $n$ in the inductive construction process we consider
\begin{equation} 
    \hat{f}_n=\hat{H}_n\circ R_{\a_{n+1}}\circ \hat{H}_n^{-1}, \ \text{where } \ \hat{H}_n=\hat{h}_1\circ\ldots \circ \hat{h}_n, \label{eqn:4:4.1a}
    \end{equation}
    with conjugation maps
    \begin{equation}
    \hat{h}_n= \hat{g}_n \circ \hat{\phi}_n.\label{eqn:4:4.1b}
\end{equation}
This time, we choose the sequence $\alpha_{n+1} = \frac{p_{n+1}}{q_{n+1}}= \alpha_n +\frac{1}{k_n\cdot l_n \cdot q_n^2}$ with parameters $k_n, l_n$ satisfying the conditions $\ref{item:P1}- \ref{item:P2}$ as defined in Section \ref{sec:3.3.1}. 

Our conjugation maps $\hat{\phi}_n$ and $\hat{g}_n$ are real analytic diffeomorphisms that commute with $R_{\alpha_n}$ and admit entire complexifications of their lifts of the torus. These maps closely resemble the combinatorial behavior of the conjugation maps $\phi_n$ and $g_n$ in the smooth case. To establish the existence of such analytic conjugation maps $\hat{\phi}_n$ and $\hat{g}_n$, we can use the following approximation result in Theorem \ref{thm:4:4.1}.

\begin{theorem}[Analytic Approximation Result]\label{thm:4:4.1} 
For any $\epsilon> 0$ and $h\in\text{Ham}^{\infty}(\T^2, \mu),$ there exists a real analytic map $\tilde{h} \in \text{Ham}_{\infty}^{\omega}(\T^2, \mu)$ such that $d_r(h, \tilde{h}\restriction_{\T^2}) < \epsilon$ for all $r\in \N$,  and the map $\tilde{h}^{-1}\in \text{Ham}_{\infty}^{\omega}(\T^2, \mu).$ 
\end{theorem}

We skip the proof of the above as it is almost identical to the proof for the annulus in \cite[Theorem 1.8]{Be22}.

\begin{corollary}\label{coro:4:4.1}
Let $\epsilon> 0 $ and $h\in \text{Ham}^{\infty}(\T^2, \mu)$ be $(\frac{1}{q},0)$-periodic. Then there exists $(\frac{1}{q},0)$-periodic $\hat{h}\in \text{Ham}_{\infty}^{\omega}(\T^2, \mu)$ such that $\hat{h}^{-1}\in \text{Ham}_{\infty}^{\omega}(\T^2, \mu)$  and the following properties hold: 
\begin{enumerate}
    \item $d_{r}(h, \hat{h}|_{\T^2}) < \epsilon $ for $r=1,2$; 
    \item $\sup_{x\in \T^2} d((h,\mathrm{d} h )(x), (\hat{h}\restriction_{\T^2},\mathrm{d} \hat{h})(x)) < \epsilon.$
\end{enumerate}  
\end{corollary}
\begin{proof}[Proof of Corollary \ref{coro:4:4.1}]
Since $h \in \text{Ham}^{\infty}(\T^2, \mu)$ is $(\frac{1}{q},0)$-periodic, there exists a $\Z^2$-periodic map $\bar{h}=(\bar{h}_1,\bar{h}_2) \in \text{Ham}^{\infty}(\T^2, \mu)$ such that 
$$h\left(x+\frac{a}{q},y\right)=\left(\frac{1}{q}\bar{h}_1(qx,y)+\frac{a}{q},\bar{h}_2(qx,y)\right), \ \text{ where } x\in \left[0,\frac{1}{q}\right).$$ 
We apply Theorem~\ref{thm:4:4.1} to get $\tilde{h}=(\tilde{h}_1,\tilde{h}_2)\in \text{Ham}_{\infty}^{\omega}(\T^2, \mu)$ approximating $\bar{h}$ with precision $\epsilon/q^2$. Then we introduce the map $\hat{h}\in \text{Ham}_{\infty}^{\omega}(\T^2, \mu)$ as
$$\hat{h}\left(x+\frac{a}{q},y\right)=\left(\frac{1}{q}\tilde{h}_1(qx,y)+\frac{a}{q},\tilde{h}_2(qx,y)\right), \ \text{ where } x \in \left[0,\frac{1}{q}\right).$$ 
Clearly, $\hat{h}$ is $(\frac{1}{q},0)$-periodic and satisfies the required properties. 
\end{proof}
For our explicit construction, we pick the proximity parameters as follows:  
\begin{align} \label{eqn:4:4.2}
    \epsilon_n = \frac{\mathfrak{d}_n}{2k_n^{8}q_n^2 \cdot \|DH_{n-1}\|_{0}^2 \cdot (2\vertiii{\phi_n}_{2}+1)} \ \text{ with } \ \mathfrak{d}_n= \frac{1}{2^{n^2+1}\cdot n^2 \cdot l_n}.
    \end{align}
    Here, we will have to choose $\mathfrak{d}_n$ sufficiently small to meet specific conditions, depending upon $H_{n-1}$, stated in Sections~\ref{sec:conv} and~\ref{sec:metric}. This can be achieved by choosing a sufficiently large $l_n$.
    
By applying Corollary~\ref{coro:4:4.1} we obtain the maps $\hat{\phi}_n$ and $\hat{g}_n$ in $\textup{Diff}_{\infty}^{\omega}(\T^2, \mu)$ approximating $\phi_n$ and $g_n$, respectively, and satisfying the following conditions: 
 \begin{itemize}
        \item $d_r(\hat{\phi}_n\restriction_{\T^2}, \phi_n) < \epsilon_n \  \text{and} \  d_r(\hat{g}_n\restriction_{\T^2}, g_n) < \epsilon_n \  \text{for all} \ r \in \N,$
         \item $\sup_{x\in \T^2} d((\phi_n,\mathrm{d}\phi_n)(x), (\hat{\phi}_n\restriction_{\T^2},\mathrm{d} \hat{\phi}_n)(x)) < \epsilon_n, \ \text{and} $
        \item $\sup_{x\in \T^2} d((g_n,\mathrm{d}g_n)(x), (\hat{g}_n\restriction_{\T^2},\mathrm{d} \hat{g}_n)(x)) < \epsilon_n.$
 \end{itemize}
 
\section{Criterion of an invariant measurable Riemannian metric}

We follow the criteria for the existence of an $f$-invariant measurable Riemannian metric, as deduced in {\cite[ section 4.8]{GKa00} and \cite[section 7]{Ku17}}. Let $\omega_0$ be the standard Riemannian metric for $\T^2$. The following definition will help to express closeness to being a local isometry for conjugation map $h_n$ on the elements of some partial partition ${\zeta}_n.$

\begin{definition}\label{def:8:8.1}
We can define the \emph{deviation from isometry} for a diffeomorphism $f$ defined on a compact subset $U$ of a smooth Riemannian manifold by
$$\text{dev}_{U}(f):= \max_{v\in \mathrm{T}U, \|v\|=1} |\log \|\mathrm{d}f(v)\||.$$  
\end{definition}

\begin{remark} \label{rem:dev}
Observe that $\text{dev}_{U}(f) =0$ if and only if $f$ is a smooth isometry of $U$. Furthermore, we have $\text{dev}_U(f) = \text{dev}_{f\left(U\right)} \left(f^{-1}\right).$
\end{remark}

\begin{proposition}[Criterion for the existence of a $f$-invariant measurable Riemannian metric]\label{prop:metric} Let $f_n=H_n\circ R_{\alpha_{n+1}}\circ H_n^{-1},$ defined by (\ref{eq:3a}) and (\ref{eq:3:3b}), such that $(f_n)_{n\in \N}$ converges to a limit diffeomorphism $f$ in the $\textup{Diff}^{\infty}$-topology. Let $\left(\zeta_n\right)_{n \in \mathbb{N}}$ be a sequence of partial partitions whose elements cover a set of measure at least $1-\frac{1}{n^2}$ for every $n \in \mathbb{N}$. There exists a function $\kappa: \textup{Diff}^{\infty}(M,\mu)\to \left(0,1\right)$ such that for every decreasing sequence $(\mathfrak{d}_n)_{n\in \N}$ of positive reals satisfying $\sum_{k\geq n}\mathfrak{d}_k < \kappa(H_{n-1})$ for every $n \in \N$ the following holds: If for every $n \in \mathbb{N}$ and every partition element $I_n\in \zeta_n$ the conjugation map $h_n$ satisfies $\text{dev}_{I_n}(h_n)\leq \frac{2\mathfrak{d}_n}{\|{DH_{n-1}}\|_0^2}$, then the limit diffeomorphism $f$ admits an invariant measurable Riemannian metric.
\end{proposition}
\begin{proof}
   The proof follows along the lines of \cite[section 7]{Ku17}. First, we put $\omega_n := (H_n^{-1})^{*}\omega_0,$ and each $\omega_n$ is a smooth Riemannian metric as it is defined as the pullback of a smooth metric via a diffeomorphism. Since $R_{\alpha_{n+1}}^{*}\omega_0 = \omega_0$ the metric $\omega_n$ is $f_n$-invariant:
$$f_n^{*}\omega_n = (H_n\circ R_{\alpha_{n+1}\circ H_n^{-1}})^{*}(H_n^{-1})^{*} \omega_0 = (H_n^{-1})^{*} R_{\alpha_{n+1}}^{*}H_n^{*} (H_n^{-1})^{*}\omega_0 =(H_n^{-1})^{*} R_{\alpha_{n+1}}^{*}\omega_0 = \omega_n. $$
With our assumptions that $\text{dev}_{I_n}(h_n) \leq \frac{\mathfrak{d}_n}{\|D H_{n-1}\|_0^2}$ for every $I_n\in {\zeta}_n$  and ${\zeta}_n$ covering a set of measure at least $1-\frac{1}{n^2}$, we can compute as in \cite[Lemma~7.2]{Ku17} that $\omega_{\infty}:= \lim_{n\rightarrow \infty} \omega_n$ exists almost everywhere with respect to Lebesgue measure $\mu$. To show that $\omega_{\infty}$ is a measurable Riemannian metric, we follow the proof of \cite[Lemma~7.3]{Ku17}. This proof requires sufficient closeness of $\omega_{\infty}$ to $\omega_{n-1}=(H_{n-1}^{-1})^{*}\omega_0$ on a set of large measure. We express this requirement via the function $\kappa$. Finally, we apply \cite[Lemma~7.5]{Ku17} to conclude that $\omega_{\infty}$ is $f$-invariant.
\end{proof}

\section{Criterion for weak mixing of the derivative extension}\label{sec:crit_weak_mixing}
Building on criteria for weak mixing for AbC diffeomorphisms on $m$-dimensional manifolds in \cite{FS} and \cite{GKu15}, here we introduce such a criterion for weak mixing of the projectivized derivative extension. 

Throughout this section we assume that the limiting smooth diffeomorphism $(f, \mathrm{d}f)$ (see \ref{eq:3a}) and the limiting analytic diffeomorphism $(\hat{f},\mathrm{d}\hat{f})$ (see \ref{eqn:4:4.1a}), admit an invariant measure $\bar{\mu}$. Strictly speaking we will be working with two different metrics for $f$ and $\hat{f}$ and, hence, different measures for the smooth and the analytic  case, as the metrics will depend on the limiting diffeomorphims. To avoid unnecessary notational complexity we will use only one notation, namely $\bar{\mu}$.

\subsection{\texorpdfstring{$(\gamma, \delta, \varepsilon_1, \varepsilon_2)$}{Lg} distribution}
A key concept in our weak mixing criterion for the projectivized derivative extension is the notion of $(\gamma, \delta, \varepsilon_1, \varepsilon_2)$ distribution that we introduce in this subsection. It generalizes the notion of $(\gamma, \delta, \varepsilon)$ distribution from  \cite{FS} that plays a central role in criteria for weak mixing for AbC diffeomorphisms in \cite{FS} and \cite{GKu15}.

To express the $(\gamma, \delta, \varepsilon_1, \varepsilon_2)$ distribution property, we consider a partial partition, denoted as $\hat{\eta}_n,$ of the space $\mathbb{P} \mathrm{T}M$ of a particular form. We let $k_n$ be the number defined by condition \ref{item:P1} and we let $\tilde{\eta}_n$ be a partial decomposition of $M$ satisfying the following properties:
\begin{enumerate}[label={\bf(D\arabic*)}]
\item\label{item:D1} The conjugation map $h_n = g_n\circ \phi_n$ acts as a composition of translations and rotations, thus acting as an isometry, on the elements of the partial partition $\tilde{\eta}_n$.
\item\label{item:D2} Each partition element $\tilde{I}_n\in\tilde{\eta}_n$ is of the form $\tilde{I}_n =\bigcup_{l=0}^{k_n -1} \tilde{I}_{n,l}$, where each $\tilde{I}_{n,l}$ is a union of squares with a side length smaller than $\frac{1}{2k_n^5q_n}$, lying in the good domain $\mathcal{G}_n$ of the map $g_n$ (see \eqref{eqn:gooddomain_g_n}).
    \item\label{item:p1} Each set $\tilde{I}_{n,l}$ has the same measure of at least $\frac{1}{2k_n^6q_n}\left(1 - \frac{16}{k_n^5}\right)$. The union of elements from $\tilde{\eta}_n$ covers a set with a measure of at least $1 - \frac{25}{k_n^5}$ in $M.$
    \item \label{item:p2}
      For any $\tilde{I}_{n,l}$, we have $\mu(\tilde{I}_{n,l} \triangle [\cup_{\tilde{I}_{n,s}^i \in \bar{\Lambda}_{\tilde{I}_{n,l}}} h_n(\tilde{I}_{n,s}^i)]) \leq  \frac{41}{k_n^5} \mu(\tilde{I}_{n,l})$, where 
    $$\bar{\Lambda}_{\tilde{I}_{n,l}} = \{\tilde{I}_{n,s}^i\ : \  h_{n}(\tilde{I}_{n,s}^i) \cap \tilde{I}_{n,l} 
 \neq \emptyset, \,   \tilde{I}_{n}^i =  \cup_{s =0}^{k_{n}-1}\tilde{I}_{n,s}^{i}\in \tilde{\eta}_{n}\}.$$
    \item\label{item:p4} 
     Each $\tilde{I}_{n,l}$ is covered by elements of $\tilde{\eta}_{n+1}$ by measure of at least $\left(1-\frac{25}{k_{n+1}^5}\right)\mu(\tilde{I}_{n,l})$ and, in particular,   
    $$\mu\left(\tilde{I}_{n,l}\triangle [ \cup_{\tilde{I}_{n+1,l'}^i\in \Lambda_{\tilde{I}_{n,l}}}\tilde{I}_{n+1,l'}^i]\right)\leq \frac{25}{k_{n+1}^5}\mu(\tilde{I}_{n,l}),$$ where $$\Lambda_{\tilde{I}_{n,l}} = \{\tilde{I}_{n+1,l'}^i  \ : \    \tilde{I}_{n+1,l'}^i \cap \tilde{I}_{n,l} \neq \emptyset, \, \tilde{I}_{n+1}^i =  \cup_{l' =0}^{k_{n+1}-1}\tilde{I}_{n+1,l'}^{i}\in \tilde{\eta}_{n+1}\}.$$
\end{enumerate}
Then we let $\hat{\eta}_n$ be a partial partition of the space $\mathbb{P} \mathrm{T}M$ of the following form:
\begin{align}\label{eq:6:1a}
    \hat{\eta}_n= \left\{ \tilde{I}_n \times T_j \ \ : \ \ \tilde{I}_n\in \tilde{\eta}_n, \ j=0,\dots , k_n -1 \right\},
\end{align}    
where $T_j=\left[\frac{j}{k_n},\frac{j+1}{k_n}\right]$, that is, elements $\hat{I}_{n,j} \in \hat{\eta}_n$ are of the form
\begin{align*}
    \hat{I}_{n,j}= \tilde{I}_n \times T_j = \bigcup_{l=0}^{k_n-1}\tilde{I}_{n,l,j}, \ \text{ with } \ \tilde{I}_{n,l,j}\coloneqq \tilde{I}_{n,l}\times T_j, 
\end{align*}
where $\tilde{I}_n =\bigcup_{l=0}^{k_n -1}\tilde{I}_{n,l}\in \tilde{\eta}_n$ and $j\in \{0,\dots , k_n -1\}$.

\remark For our explicit constructions in Section~\ref{sec:exp_setup}, properties \ref{item:D1}, \ref{item:D2}, \ref{item:p1}, \ref{item:p2}, and \ref{item:p4} are verified in Remark~\ref{rem:8:3:1}, Remark~\ref{rem:assumption_intesection}, Remark~\ref{rem:assumption_intersection2}, and Remark~\ref{item:verify:D5}, respectively.

\definition\label{def:6.2a} A diffeomorphism $(\Phi_n,\mathrm{d}\Phi_n): \mathbb{P} \mathrm{T}M \longrightarrow \mathbb{P} \mathrm{T}M $ is $(\gamma, \delta,\varepsilon_1,\varepsilon_2)$-distributing  $\hat{I}_{n,j} = \tilde{I}_{n}\times T_j \in \hat{\eta}_n$, where $\tilde{I}_n = \cup_{l=0}^{k_n-1}\tilde{I}_{n,l}\in \tilde{\eta}_n,$ if the following properties hold
\begin{enumerate}
    \item for any $l\in \{0,\ldots,k_n-1\},$ we have $J_l \subseteq \pi_{r}(\Phi_n(\tilde{I}_{n,l}))$  with  $1-\delta \leq \lambda(J_l) \leq 1$, 
    where $\pi_{r}$ is the projection onto the $r$-axis and $\lambda$ is the Lebesgue measure of $\mathbb{R}$;
    \item for any $l\in \{0,\ldots,k_n-1\},$ we have $\Phi_n(\tilde{I}_{n,l}) \subseteq K_{c_l,\gamma} \coloneqq {[c_l,c_l+\gamma]\times [0,1]}$ for some constant $c_l\in \T^1$;   
    \item $\Phi_n:M\rightarrow M$ is $\varepsilon_1$-distributing $\tilde{I}_{n}\in \tilde{\eta}_n,$ that is,  for  $l\in \{0,\ldots,k_n-1\}$ and for any $\tilde{J}\subseteq J_l,$
    $${\left\lvert \frac{\mu(\tilde{I}_{n,l}\cap \Phi_n^{-1}(\mathbb{T}\times \tilde{J}))}{\mu(\tilde{I}_{n,l})} - \frac{\lambda(\tilde{J})}{\lambda(J_l)}\right\rvert} \leq \varepsilon_1 \frac{\lambda(\tilde{J})}{\lambda(J_l)};$$    
    \item $(\Phi_n,\mathrm{d}\Phi_n)$ is $\varepsilon_2$-distributing $\hat{I}_{n,j} \in \hat{\eta}_n$ in the tangent direction in the following sense: for all $j,k\in \{0,1,\ldots, k_n-1\}$, there exists a unique $ l \in \{0,1,\ldots, k_n-1\} $ such that $ k \equiv l+j \mod{k_n},$ and for $\tilde{J}\subseteq J_l $ and $T_k=\left[\frac{k}{k_n},\frac{k+1}{k_n}\right]$ we have
    $${\left\lvert \frac{\mu\left(\pi_M\left(\tilde{I}_{n,l}\times T_j \cap (\Phi_n,\mathrm{d}\Phi_n)^{-1}(\mathbb{T}\times \tilde{J}\times T_k)\right)\right)}{\mu(\tilde{I}_{n,l})} - \frac{\lambda(\tilde{J})}{\lambda(J_l)}\right\rvert} \leq \varepsilon_2 \frac{\lambda(\tilde{J})}{\lambda(J_l)},$$
    where $\pi_{M}$ denotes the projection onto the manifold $M$.
    \end{enumerate}    
      We will write the fourth condition as 
     $${\left\lvert \mu\left(\pi_M\left(\tilde{I}_{n,l,j}\cap (\Phi_n,\mathrm{d}\Phi_n)^{-1}(\mathbb{T}\times \tilde{J}\times T_k)\right)\right){\lambda(J_l)} - \mu(\tilde{I}_{n,l}){\lambda(\tilde{J})}\right\rvert} \leq \varepsilon_2 \mu(\tilde{I}_{n,l}){\lambda(\tilde{J})}.$$
     
 In particular, the distribution in the tangent direction implies that the product set of the $l$th segment $\tilde{I}_{n,l}$ of a partition element $\tilde{I}_{n} \in \tilde{\eta}_n$ on the manifold with the $j$th tangent element $T_j$ is mapped under $(\Phi_n,\mathrm{d}\Phi_n)$ to the $k$th tangent element $T_k,$ where $k\equiv l+ j\mod{k_n}$. 

\paragraph{Key Idea of proof:} For $\tilde{I}_n = \cup_{l=0}^{k_n-1}\tilde{I}_{n,l}\in \tilde{\eta}_n$ and $l,j \in \{0,\ldots,k_n-1\}$, we consider  $\Gamma_{n,l,j} = (H_{n-1},\mathrm{d}H_{n-1})(g_n, d g_n) (\tilde{I}_{n,l}\times T_j)$. We approximate its measure by using the ``good domain" of $H_{n-1},$ defined as $\widetilde{G
}_n$ in (\ref{eqn:future_domain}), where the measure $\bar{\mu}$ is computed with the length in the tangent direction and it follows
$$ |\bar{\mu}(\Gamma_{n,l,j})- \frac{1}{k_n}\mu(\tilde{I}_{n,l})| \to 0\ \text{ as } \ n \to \infty.$$ 
In fact, it subsequently follows in Lemma~\ref{lemm~measurepreservation} for $\Gamma_{n,j} = \cup_{l=0}^{k_n-1} \Gamma_{n,l,j}$ that
\begin{align}\label{eqn:Key_idea1}
|\bar{\mu}(\Gamma_{n,j})- \mu(\tilde{I}_{n,l})| \to 0 \ \text{ as } \ n \to \infty.
\end{align}

Analogously, we obtain in Lemma~\ref{lemm~measurepreservation}
for any $l,j \in \{0,\ldots, k_n-1\}$ that
\begin{align*}
|\bar{\mu}(\Gamma_{n,l,j}\cap (f_n,\mathrm{d} f_n)^{-m_n}(C_{n,k}))- \frac{1}{k_n}\mu(\pi_M(\tilde{I}_{n,l}\times T_j\cap (\Phi_n,\mathrm{d}\Phi_n)^{-1}\circ(g_n, \mathrm{d} g_n)^{-1}(\tilde{S}_{n,k}))| \to 0,
\end{align*}
where $C_{n,k} = (H_{n-1},\mathrm{d} H_{n-1})(\tilde{S}_{n,k})$ and $\{\tilde{S}_{n,k}\}$ is a collection of small cubes inside $\mathbb{P}\mathrm{T}M$ that covers almost the entire space.

Together with the distribution property of $(\Phi_n,\mathrm{d} \Phi_n)$ and $(g_n, \mathrm{d} g_n)$ proved in Lemma~\ref{lem: 6.6}, this yields the estimate
\begin{align}\label{eqn:Key_idea2}
|\bar{\mu}(\Gamma_{n,l,j}\cap (f_n,\mathrm{d} f_n)^{-m_n}(C_{n,k})) - \mu(\tilde{I}_{n,l})\bar{\mu}(C_{n,k})| \to 0.
\end{align}
Finally using equations \eqref{eqn:Key_idea1} and \eqref{eqn:Key_idea2}, we obtain that for all $ j,k\in \{0,\ldots,k_n-1\},$ the following estimate holds: 
\begin{align}
 |\bar{\mu}(\Gamma_{n,j}\cap (f_n,\mathrm{d} f_n)^{-m_n}(C_{n,k})) - \bar{\mu}(\Gamma_{n,j})\bar{\mu}(C_{n,k})| \to 0,
\end{align}
which is the final step required for the criterion of weak mixing in Proposition~\ref{prop:6.6.3a}.  

\subsection{Preliminary lemmas for weak mixing}
Using the partial partition $\tilde{\eta}_n$, we denote $G_{n} = \cup_{\tilde{I}_{n}\in \tilde{\eta}_n}h_n(\tilde{I}_{n})$, and define
\begin{align}\label{eqn:future_domain}
\widetilde{G}_n = G_{n} \cap \bigcap_{j=1}^{\infty} h_n\circ h_{n+1}\circ \ldots \circ h_{n+j-1}(G_{n+j}).
\end{align}
We observe that the composition of conjugation maps $h^{-1}_{n+s}\circ \ldots \circ h^{-1}_{n+1} \circ h^{-1}_n, \ s\geq 0,$  acts as an isometry on $\widetilde{G}_n$. This will be useful in approximating the $\bar{\mu}$ measure of specific sets (see Lemma~\ref{lemm~measurepreservation}).

For each $\tilde{I}_n = \cup_{l=0}^{k_n-1}\tilde{I}_{n,l}\in \tilde{\eta}_n$, we define sets $\Breve{I}_{n,l}$  as
$\Breve{I}_{n,l} = \tilde{I}_{n,l}\cap  \widetilde{G}_n.$ Note that each $\tilde{I}_{n,l}$  is almost covered by elements $\tilde{I}_{n+1}\in \tilde{\eta}_{n+1}$ by \ref{item:p4}. Specifically, we have $\mu\left(\tilde{I}_{n,l} \triangle \bigcup_{\tilde{I}_{n+1,l'}^i\in \Lambda_{\tilde{I}_{n,l}}}\tilde{I}_{n+1,l'}^i\right) \leq  \frac{25}{k_{n+1}^5}\mu(\tilde{I}_{n,l}).$ Moreover, \ref{item:p2} states that for each $\tilde{I}_{n+1,l'}^i,$  we have  $$\mu(\tilde{I}_{n+1,l'}^i \triangle  [\cup_{\tilde{I}_{n+1,s}^k\in \bar{\Lambda}_{\tilde{I}_{n+1,l'}^i}} h_{n+1} (\tilde{I}_{n+1,s}^k)])  \leq \frac{41}{k_{n+1}^5}\mu(\tilde{I}_{n+1,l'}^i).$$ 
We denote $\mathcal{C}_{\tilde{I}_{n,l}} = \cup_{\tilde{I}_{n+1,l'}^i\in \Lambda_{\tilde{I}_{n,l}}} \cup_{\tilde{I}_{n+1,s}^k\in \bar{\Lambda}_{\tilde{I}_{n+1,l'}^i}}\tilde{I}_{n+1,s}^k.$ Altogether, for each $\tilde{I}_{n,l},$ we have 
\begin{align}\label{cover_estimate_1}
&\mu\bigg(h_n(\tilde{I}_{n,l})\triangle [\cup_{\tilde{I}_{n+1,s}^k\in \mathcal{C}_{\tilde{I}_{n,l}}} h_n\circ h_{n+1}(\tilde{I}_{n+1,s}^k)]\bigg) \nonumber \\
 \leq & \mu\left(\cup_{\tilde{I}_{n+1,l'}^i\in \Lambda_{\tilde{I}_{n,l}}}\tilde{I}_{n+1,l'}^i  \triangle [\cup_{\tilde{I}_{n+1,s}^k\in \mathcal{C}_{\tilde{I}_{n,l}}} h_{n+1}(\tilde{I}_{n+1,s}^k)]\right) + \mu\left(\cup_{\tilde{I}_{n+1,l'}^i\in \Lambda_{\tilde{I}_{n,l}}}\tilde{I}_{n+1,l'}^i  \triangle \tilde{I}_{n,l}\right) \nonumber\\
 \leq & \sum_{\tilde{I}_{n+1,l'}^i\in \Lambda_{\tilde{I}_{n,l}}}\mu\left(\tilde{I}_{n+1,l'}^i  \triangle [\cup_{\tilde{I}_{n+1,s}^k\in \bar{\Lambda}_{\tilde{I}_{n+1,l'}^i}} h_{n+1}(\tilde{I}_{n+1,s}^k)]\right) + \frac{25}{k_{n+1}^5}\mu(\tilde{I}_{n,l})\nonumber\\
\leq & \sum_{\tilde{I}_{n+1,l'}^i\in \Lambda_{\tilde{I}_{n,l}}} \frac{41}{k_{n+1}^5}\mu(\tilde{I}_{n+1,l'}^i) + \frac{25}{k_{n+1}^5}\mu(\tilde{I}_{n,l}) = \frac{41}{k_{n+1}^5} \mu(\cup_{\tilde{I}_{n+1,l'}^i\in \Lambda_{\tilde{I}_{n,l}}}\tilde{I}_{n+1,l'}^{i}) + \frac{25}{k_{n+1}^5}\mu(\tilde{I}_{n,l}) \nonumber\\
 \leq & \frac{41}{k_{n+1}^5}\left(\mu(\tilde{I}_{n,l}) +\frac{25}{k_{n+1}^5}\mu(\tilde{I}_{n,l})\right) + \frac{25}{k_{n+1}^5}\mu(\tilde{I}_{n,l}) \leq \frac{70}{k_{n+1}^5}\mu(\tilde{I}_{n,l}).
\end{align}
Similarly, using properties \ref{item:p2} and \ref{item:p4}, we find that each $\tilde{I}_{n+1,l'}^i$ is covered by elements of $h_{n+2}(\tilde{I}_{n+2,s'}^k),$ where $\tilde{I}_{n+2,s'}^k \in \mathcal{C}_{\tilde{I}_{n+1,l'}^i}$ and we have the estimate:
\begin{align}\label{eqn:h_{n+1}covering_0}
\mu\left(h_n\circ h_{n+1}(\tilde{I}_{{n+1},l'}^i)\triangle \cup_{\tilde{I}_{n+2,s'}^k\in \mathcal{C}_{\tilde{I}_{n+1,l'}^i}}h_n\circ  h_{n+1}\circ h_{n+2}(\tilde{I}_{n+2,s'}^k)\right)\leq  \frac{70}{k_{n+2}^5}\mu( h_{n+1}(\tilde{I}_{n+1,l'}^i)).
\end{align}
Given any set $\tilde{I}_{n,l}$, we consider the following family of sets: 
\begin{align*}
   & A_{n_1} = \bigcup_{\tilde{I}_{n+1,l'}^i\in  \mathcal{C}_{\tilde{I}_{n,l}}}\bigcup_{\tilde{I}_{n+2,s}^k\in \mathcal{C}_{\tilde{I}_{n+1,l'}^i}}\ldots \bigcup_{\tilde{I}_{n+{n_1},s''}^{\ell}\in \mathcal{C}_{\tilde{I}_{n+ n_1-1,s'}^{i'}}} h_{n}\circ h_{n+1}\circ \ldots \circ h_{n+n_1}(\tilde{I}^{\ell}_{n+n_1,s''}). 
    \end{align*}
Analogous to equations \eqref{cover_estimate_1} and \eqref{eqn:h_{n+1}covering_0}, we have estimates for $n_1=2,3,4,\ldots$,
\begin{align}\label{eqn:symmetric_differenc1}
    \mu(A_{n_1-1}\triangle A_{n_1}) \leq \frac{70}{k_{n+n_1-1}^5}\mu(A_{n_1-1}).
\end{align}
Using the triangle inequality on the aforementioned equation, we conclude that

\begin{align*}
   \mu(A_1\triangle A_3)&\leq \mu(A_1\triangle A_2)+ \mu(A_2\triangle A_3) 
    \leq \mu(A_1\triangle A_2) + \frac{70}{k_{n+2}^5}\mu(A_2)\nonumber \\
   &\leq \mu(A_1\triangle A_2) + \frac{70}{k_{n+2}^5}(\mu(A_1) +\mu(A_1\triangle A_2))\leq \left(1 + \frac{70}{k_{n+2}^5}\right)\mu(A_1\triangle A_2) + \frac{70}{k_{n+2}^5}\mu(A_1)\nonumber\\
   &\leq \left(\left(1 + \frac{70}{k_{n+2}^5}\right)\frac{70}{k_{n+1}^5} + \frac{70}{k_{n+2}^5}\right)\mu(A_1).
\end{align*}
Likewise, by applying recursion with equation \ref{eqn:symmetric_differenc1}, we obtain for each $n_1\geq 3$,
\begin{align}\label{eqn:h_{n+1}covering_2}
\mu(A_1\triangle A_{n_1}) &\leq  \mu(A_1\triangle A_{n_1-1}) + \mu(A_{n_1-1}\triangle A_{n_1}) \nonumber\\
&\leq \left(1+\frac{70}{k_{n+n_1-1}^5}\right)\mu(A_1\triangle A_{n_1-1}) + \frac{70}{k_{n+n_1-1}^5}\mu(A_{1})\nonumber\\
& \leq \left(\sum_{u=1}^{n_1-2}\prod_{j=u+1}^{n_1-1}\left(1+\frac{70}{k_{n+j}^5}\right)\frac{70}{k_{n+u}^5}+ \frac{70}{k_{n+n_1-1}^5}\right)\mu(A_1).
\end{align}
Following this, we apply mathematical induction to derive the approximation for the infinite intersection of these sets $\{A_n\}_{n\in \N}$. Using the inequality $ \mu(A_1\cap A_2\setminus A_3) \leq \mu(A_2\setminus A_3)\leq \mu(A_2\triangle A_3) \leq \frac{70}{k_{n+2}^5}\mu(A_2),$ we obtain for $n=3$:
\begin{align}
    \mu(A_1\cap A_2\cap A_3)&= \mu(A_1\cap A_2) - \mu(A_1\cap A_2 \setminus A_3) 
    \geq \mu(A_1)- \mu(A_1 \triangle A_2) - \frac{70}{k_{n+2}^5}\mu(A_2)\nonumber\\
   &  \geq \mu(A_1)- \mu(A_1 \triangle A_2) - \frac{70}{k_{n+2}^5}(\mu(A_1)+ \mu(A_1\triangle A_2))\nonumber\\
   &\geq \left(1-\left(1+\frac{70}{k_{n+2}^5}\right)\frac{70}{k_{n+1}^5}-\frac{70}{k_{n+2}^5} \right)\mu(A_1).
    \end{align}

By employing mathematical induction and equation \ref{eqn:h_{n+1}covering_2}, one can show that for any $n_1 \geq 3$ we have
\begin{align}
    \mu(A_1\cap A_2\ldots \cap A_{n_1}) 
    &\geq \left(1- \sum_{u=1}^{n_1-2}\prod_{j=u+1}^{n_1-1}\left(1+\frac{70}{k_{n+j}^5}\right)\frac{70}{k_{n+u}^5}- \frac{70}{k_{n+n_1-1}^5}\right)\mu(A_1).
\end{align}
     
With the aforementioned estimates leads to the following conclusion for $n\geq 3$:
\begin{align}\label{eqn:5:5.16}
\mu(\Breve{I}_{n,l})= \mu\left(\tilde{I}_{n,l}\cap \widetilde{G}_n\right) &\geq \left(1- \sum_{u=0}^{\infty}\prod_{j=u+1}^{\infty}\left(1+\frac{70}{k_{n+j}^5}\right)\frac{70}{k_{n+u}^5} -  \frac{70}{k_{n+1}^5}\right)\mu(\tilde{I}_{n,l})\nonumber\\
& \geq \left(1- 2\cdot\sum_{u=0}^{\infty}\frac{70}{k_{n+u}^5}-\frac{70}{k_{n+1}^5}\right)\mu(\tilde{I}_{n,l})\nonumber \\
&\geq \left(1- \frac{36}{k_n^4}\right)\mu(\tilde{I}_{n,l}).
\end{align}
We employed the fact that $\prod_{j=1}^{\infty}\left(1+\frac{70}{k_{n+j}^5}\right)\leq 2,$ and under condition \ref{item:P0}, the sequence $(k_{j})_{j\in \N}$ increases rapidly enough to guarantee that $\sum_{u=0}^{\infty} \frac{140}{k_{n+u}^5} \leq \frac{140}{4k_n^4}$ and $k_{n+1}^5 \geq 2 k_n^4$. 
\remark\label{eqn_domains_measure_preserving}  Moreover, for any subset $S_n'\subseteq \T^2$, we define $\tilde{c}_{n,l} = \tilde{I}_{n,l}\cap \Phi_n^{-1}\circ g_n^{-1}({S}_n')$ and $\Breve{c}_{n,l} = \Breve{I}_{n,l}\cap \Phi_n^{-1}\circ g_n^{-1}({S}_n')$. Thus, we obtain
$|\mu(\tilde{c}_{n,l}) - \mu(\Breve{c}_{n,l})|\leq \mu(\tilde{I}_{n,l}\backslash \Breve{I}_{n,l}) \leq \frac{36}{k_n^4}\mu(\tilde{I}_{n,l}).$
\begin{lemma}\label{lemm~measurepreservation}
   Denote $\mathcal{A}_{n,l,j}= (H_{n-1},\mathrm{d} H_{n-1})\circ (g_n,\mathrm{d}g_n)(\tilde{c}_{n,l,j}),$ where   
   $\tilde{c}_{n,l,j} = \tilde{c}_{n,l} \times T_j$. Then we have 
   \begin{align}\label{eqn:measure_estimate}
   \left\lvert\bar{\mu}(\mathcal{A}_{n,l,j})  - \frac{1}{k_n}{\mu}(\pi_M(\tilde{c}_{n,l,j}))\right\rvert \leq \frac{36(k_n-1)}{k_n^5}\mu(\tilde{I}_{n,l}).
    \end{align}
 Moreover, for any $S_n  \times T_k \subset \mathbb{P}\mathrm{T} M$ and $C_{n,k} = (H_{n-1},\mathrm{d} H_{n-1})(\tilde{S}_n\times T_k )$, where $\tilde{S}_n = S_n \cap \widetilde{G}_n,$ we have
  $ \bar{\mu}(C_{n,k}) = \frac{1}{k_n}\mu(\tilde{S}_n). $
    \end{lemma}
        \begin{proof}
Note that $\Breve{c}_{n,l} = \tilde{c}_{n,l} \cap \widetilde{G}_n$. For any $x\in \Breve{c}_{n,l}$ let $y = H_{n-1}(x).$ Then let $(y,v), (y,w)$ be defined as $(y, v) = (H_{n-1} (x), \mathrm{d}_{x} H_{n-1}(\bar{v}))$ and $(y, w) = (H_{n-1}(x), 
 \mathrm{d}_{x} H_{n-1}(\bar{w})),$ with $x\in \Breve{c}_{n,l} \subset \widetilde{G}_n, \ \bar{v},\bar{w}\in T_j.$ 
Using the definition of the $f-$invariant Riemannian metric $\omega_{\infty},$ which induces the measure $\bar{\mu}$ invariant under the map  $(f,\mathrm{d}f),$
we conclude that 
\begin{align*}
\omega_{\infty}\restriction_{y}(v,w) &= \lim_{k\rightarrow\infty} (H_k^{-1})^{*}\omega_{0}\restriction_{y}(v,w) = \lim_{k\rightarrow\infty} \omega_{0}\restriction_{H_k^{-1}(y)}(\mathrm{d}_y H_k^{-1}(v), \mathrm{d}_y H_k^{-1}(w))\\
& = \lim_{k\rightarrow\infty}  \omega_{0}\restriction_{h_k^{-1}\circ \ldots \circ h_n^{-1}(x)}(\mathrm{d}_{x} h_k^{-1}\circ \ldots \circ h_n^{-1}(\bar{v}), \mathrm{d}_{x} h_k^{-1}\circ \ldots \circ h_n^{-1}(\bar{w})).
\end{align*}
Since $h_k^{-1}\circ \ldots \circ h_n^{-1}$ is an isometry with respect to $\omega_{0}$ on $\widetilde{G}_n$ and $\omega_{0}$ is independent from the base point, we conclude that $\omega_{\infty}\restriction_{y}(v,w) = \omega_{0}\restriction_{x} (\bar{v}, \bar{w}).$ Thus, for images under the map $(H_{n-1},\mathrm{d}H_{n-1})$ of sets with base points $ x\in \Breve{c}_{n,l}\subseteq \widetilde{G}_n,$
 the measure $\bar{\mu}$ acts as a uniform measure, preserving the length in the tangent direction. By utilizing the estimate in Remark~\ref{eqn_domains_measure_preserving} and noting that 
$g_n^{-1}$ acts as translation  on $\widetilde{G}_n,$ we have
\begin{align*}
    & \frac{1}{k_n}\mu(g_n^{-1}(\Breve{c}_{n,l}))\leq \bar{\mu}(\mathcal{A}_{n,l,j})\leq \frac{1}{k_n}\mu(g_n^{-1}(\Breve{c}_{n,l}))+ \frac{36(k_n-1)}{k_n^5}  \mu(g_n^{-1}(\tilde{I}_{n,l})).
\end{align*}
Finally, using that $g_n$ is measure preserving and estimate $\mu(\tilde{c}_{n,l})\geq \mu(\Breve{c}_{n,l})\geq \mu(\tilde{c}_{n,l}) -  \frac{36}{k_n^4}\mu(\tilde{I}_{n,l}) $, we conclude that
$$\left\lvert\bar{\mu}(\mathcal{A}_{n,l,j})  - \frac{1}{k_n}{\mu}(\pi_M(\tilde{c}_{n,l,j}))\right\rvert \leq \frac{36(k_n-1)}{k_n^5}\mu(\tilde{I}_{n,l}).$$
Analogously to the above estimate, we conclude the other claim. 
\end{proof}
\begin{lemma}\label{sec:6:lem:3}
Consider $g_n: \mathbb{T}^2\to\mathbb{T}^2$ as defined in Proposition {\ref{lemma:3a}}. Let $K$ be an interval on $r$-axis of the form $[\frac{i+2\varepsilon_n}{a_n}, \frac{i+1- 2\varepsilon_n}{a_n}]$ where $i \in \{0,...,a_n-1\}.$ Denote $K_{c_l,\gamma}= [c_l,c_l+\gamma]\times K$
for some constant $c_l\in \mathbb{T}^1$ and $\gamma$. Let $L= [l_1,l_2]$ be an interval on the $\theta$ axis. If $b_n\lambda(K)>2,$ then for
$ Q= \pi_{r}(K_{c_l,\gamma}\cap g_n^{-1}(L\times K)),$
it holds 
\begin{align}
    |\lambda(Q)- \lambda(K)\lambda(L)|\leq \frac{2}{b_n}\lambda(L) + \frac{2\gamma}{b_n}+ \gamma\lambda(K)+ \frac{b_n}{a_n}\lambda(K) +\frac{2}{a_n}.
\end{align}
\end{lemma}
\begin{proof}
The proof follows directly from \cite[Lemma 6.4]{GKu15} for dimension $m=2$.
\end{proof}
\begin{lemma}\label{lem: 6.6}
Let $n>4$, $g_n$ be as in Proposition {\ref{lemma:3a}}, and $\hat{I}_{n,j} = \cup_{l=0}^{k_n-1}\tilde{I}_{n,l,j}\in \hat{\eta}_n,$ where $\hat{\eta}_n$ is a partial partition of $\mathbb{P} \mathrm{T}M$ as described by equation (\ref{eq:6:1a}). Suppose the diffeomorphism $(\Phi_n,\mathrm{d}\Phi_n):\mathbb{P} \mathrm{T}M\to \mathbb{P} \mathrm{T}M$ is $(\gamma,\delta,\varepsilon_1,\varepsilon_2)$-distributing $\hat{I}_{n,j}\in\hat{\eta}_n$ with 
$\gamma < \frac{1}{k_n^3q_n}, \delta < \frac{50}{k_n^2}, \varepsilon_1, \varepsilon_2< \frac{1}{n}.$ Let $S_n$ be a square of side length equal to $\frac{1}{k_n}$ lying in $\T^1 \times [\delta,1-\delta]$. Furthermore, let $\tilde{S}_n =S_n \cap \widetilde{G}_n$ using the good domain $\widetilde{G}_n$ from (\ref{eqn:future_domain}). For $k\in \{0,\dots , k_n-1\},$ denote $\tilde{S}_{n,k}= \tilde{S}_n\times T_k$, where $T_k =\left[\frac{k}{k_n},\frac{k+1}{k_n}\right]$. 

Then for every pair $j,k\in \{0,1,\ldots,k_n-1\},$ there exists $l\in  \{0,1,\ldots,k_n-1\}$ such that $k\equiv (l+j)\mod k_n$ and it holds:
$${\left\lvert \mu\left(\pi_M\left(\tilde{I}_{n,l,j}\cap (\Phi_n,\mathrm{d}\Phi_n)^{-1}\circ (g_n,\mathrm{d}g_n)^{-1}(\tilde{S}_{n,k})\right)\right) - {{\mu}}(\tilde{I}_{n,l}){\mu(\tilde{S}_{n})}\right\rvert} \leq \frac{68}{n} {\mu}(\tilde{I}_{n,l})\mu(\tilde{S}_{n}).$$
\end{lemma}
\begin{proof}
Denote $\tilde{S}_{n,k} = \tilde{S}_n\times T_k$ with $\tilde{S}_n =S_n \cap \widetilde{G}_n,$ where $S_n$ is a square of side length equal to $\frac{1}{k_n}$ lying in $\T^1 \times [\delta,1-\delta].$ Let $\tilde{S}_{\theta}= \pi_{\theta}(\tilde{S}_n)\subseteq \pi_{\theta}({S}_n) = [s_1,s_2], \ \tilde{S}_r= \pi_r(\tilde{S}_n) \subseteq \pi_{r}({S}_n)$,
where $\pi_{\theta}$ and $\pi_{r}$ are projections on $\theta$ and $r$ axis, respectively. For $\tilde{I}_n = \cup_{l=0}^{k_n-1}\tilde{I}_{n,l} \in \tilde{\eta}_n$ and $l\in \{0,\ldots,k_n-1\},$
denote $J_l = \pi_r(\Phi_n(\tilde{I}_{n,l})).$ 

Since $g_n^{-1}$ acts as a translation on $\tilde{S}_n$, we have $\mathrm{d}_pg_n^{-1} \equiv \text{id}$ for all base points $p \in \tilde{S}_n$. 
Thus, the following holds using $(\Phi_n,\mathrm{d}\Phi_n)$ being $(\gamma,\delta,\varepsilon_1,\varepsilon_2)$-distributing: 
 \begin{align}
     \Phi_n(\tilde{I}_{n,l})\cap g_n^{-1}(\tilde{S}_n)\subseteq [c_l,c_l+\gamma]\times \tilde{S}_r \subseteq K_{c_l,\gamma}, \label{lem:cl:1} \\
    (g_n,\mathrm{d}g_n)^{-1}(\tilde{S}_n \times T_k) = g_n^{-1}(\tilde{S}_n)\times T_k. \label{lem:cl:2}
\end{align}
Consider the following collection of sets, for
$\tilde{S}_{n,\gamma} \subset \tilde{S}_n$ where  $\tilde{S}_{n,\gamma}\subseteq [s_1+\gamma, s_2-\gamma]\times \tilde{S}_r$, 
\begin{align}
    Q= \pi_{r}(K_{c_l,\gamma}\cap g_n^{-1}(\tilde{S}_{n})) \ \ \ \text{and} \ \
    Q_1= \pi_{r}(K_{c_l,\gamma}\cap g_n^{-1}(\tilde{S}_{n,\gamma})).\nonumber
\end{align}
Then we have the following containment relation using (\ref{lem:cl:1}):
\begin{align} \label{eqn:6.6.2}
\Phi_n(\tilde{I}_{n,l})\cap \T^1 \times Q_1 \subseteq \Phi_n(\tilde{I}_{n,l})\cap g_n^{-1}(\tilde{S}_n)\subseteq \Phi_n(\tilde{I}_{n,l})\cap \T^1 \times Q.
\end{align}
Subsequently, we use equation (\ref{lem:cl:2}) to express the following: for any $ k,j\in \{0,\ldots,k_n-1\},$ there exist a $l\in \{0,\ldots,k_n-1\}$  such that $k \equiv (l+j)\mod k_n,$
\begin{align} \label{eqn:6.6.4}
(\Phi_n,\mathrm{d} \Phi_n)(\tilde{I}_{n,l,j})\cap (g_n,\mathrm{d} g_n)^{-1}(\tilde{S}_n \times T_k) = (\Phi_n(\tilde{I}_{n,l})\cap g_n^{-1}(\tilde{S}_n)) \times  (\pi_v((\Phi_n,\mathrm{d}\Phi_n)(\tilde{I}_{n,l,j}))\cap T_k),
\end{align}
where $\pi_v$ represents the projection to the projective tangent vector, i.e. projection in the third coordinate. Thus, altogether with  (\ref{eqn:6.6.2}) and (\ref{eqn:6.6.4}), we have the following inclusions 
$$ (\Phi_n,\mathrm{d}\Phi_n)(\tilde{I}_{n,l,j})\cap (g_n,\mathrm{d}g_n)^{-1}(\tilde{S}_n\times T_k)\subseteq (\Phi_n,\mathrm{d}\Phi_n)(\tilde{I}_{n,l,j})\cap (\T^1\times Q\times T_k),$$
$$(\Phi_n,\mathrm{d}\Phi_n)(\tilde{I}_{n,l,j})\cap (\T^1\times Q_1\times T_k)\subseteq (\Phi_n,\mathrm{d}\Phi_n)(\tilde{I}_{n,l,j})\cap (g_n,\mathrm{d}g_n)^{-1}(\tilde{S}_n\times T_k).$$
We conclude that
\begin{align}
   &{\left\lvert \mu\left(\pi_M\left(\tilde{I}_{n,l,j} \cap (\Phi_n,\mathrm{d}\Phi_n)^{-1}\circ (g_n,\mathrm{d}g_n)^{-1}(\tilde{S}_{n,k})\right)\right)\lambda(J_l) - {\mu}(\tilde{I}_{n,l}){\mu}(\tilde{S}_{n})\right\rvert}\nonumber\\
   &\leq \max \Big( \big\lvert \mu\big(\pi_M\big(\tilde{I}_{n,l,j}\cap (\Phi_n,\mathrm{d}\Phi_n)^{-1}(\T^1\times Q_1\times T_k)\big)\big)\lambda(J_l) - {\mu}(\tilde{I}_{n,l}){\mu}(\tilde{S}_{n})\big\rvert, \nonumber\\
   & \quad \quad \big|\mu\big(\pi_M\big(\tilde{I}_{n,l,j}\cap (\Phi_n,\mathrm{d}\Phi_n)^{-1}(\T^1\times Q \times T_k)\big)\big)\lambda(J_l) - {\mu}(\tilde{I}_{n,l}){\mu}(\tilde{S}_{n})\big|\Big).\label{lem:eq:3a}
\end{align}
Next we apply Lemma \ref{sec:6:lem:3} for $S_n' = S_n\cap \mathcal{G}_n,$ where $\mathcal{G}_n$ is defined by (\ref{eqn:gooddomain_g_n}). In particular, $|\mu(S_n)- \mu(S_n')|\leq 4 \varepsilon_n \mu(S_n)$.  By our choices $\gamma < \frac{1}{k_n^3q_n}, \delta < \frac{50}{k_n^2}, \varepsilon_1 = \varepsilon_2 < \frac{1}{n}, K = \pi_{r}(S_n'), L= \pi_{\theta}(S_n'), b_n= [nq_n^{\sigma}], a_n= k_n^5, \varepsilon_n \leq \frac{1}{2k_n^{10}}$ we have $\lambda(L)\leq \frac{1}{k_n}, \lambda(K)\leq \frac{1}{k_n},$ and for $n>4$ using the condition \ref{item:P1} and \ref{item:P3}, we find that $\frac{1}{q_n^{\sigma}}<\frac{1}{2k_n},$ $k_n> n^2,$  $\frac{b_n}{a_n} \leq \frac{1}{n k_n^2},$ and $b_n\lambda(K) >2$. Altogether, we obtain
\begin{align}
    |\lambda(Q)- \mu({S}_n')| \leq \left(\frac{2}{[nq_n^{\sigma}]}\lambda(S_{\theta}) + \frac{2}{[nq_n^{\sigma}] \cdot k_n^3q_n} +\frac{1}{k_n^{3}q_n}\lambda({S}_{r}) + \frac{[nq_n^{\sigma}]}{2k_n^9q_n}\lambda({S}_r) + \frac{2}{2k_n^9q_n}\right)\leq \frac{13}{n}\mu({S}_n). \nonumber
\end{align}
Using the estimate $|\mu(S_n)- \mu(\tilde{S}_n)|\leq \frac{36}{k_n^4}\mu(S_n),$ we have $|\lambda(Q)- \mu(\tilde{S}_n)|\leq |\lambda(Q)- \mu({S}_n')| +|\mu(S_n')- \mu({S}_n)|+ |\mu(S_n)- \mu(\tilde{S}_n)| \leq \frac{30}{n}\mu(\tilde{S}_n). $
In particular, $\lambda({Q})\leq (\frac{30}{n}+ 1)\mu(\tilde{S}_n) \leq 2\mu(\tilde{S}_n).$ Similarly we estimate $\lambda(Q_1)\leq 2\mu(\tilde{S}_n)$ as well as $|\lambda(Q_1)- \mu(\tilde{S}_{n,\gamma})|\leq \frac{30}{n}\mu(\tilde{S}_n).$

Since $Q$ and $Q_1$ are finite unions of disjoint intervals contained within $J_l$, we utilize the $(\Phi_n,\mathrm{d}\Phi_n)$-$(\gamma,\delta,\varepsilon_1,\varepsilon_2)$ distribution property with $\tilde{J}= Q \subseteq J_l,$ and we have
\begin{align}
&\left\lvert \mu\left(\pi_M\left(\tilde{I}_{n,l,j}\cap (\Phi_n,\mathrm{d}\Phi_n)^{-1}(\T^1\times Q\times T_k)\right) \right) \lambda(J_l) - {\mu}(\tilde{I}_{n,l})\mu(\tilde{S}_n)\right\rvert \nonumber\\
&\leq \left\lvert \mu\left(\pi_M\left(\tilde{I}_{n,l,j}\cap (\Phi_n,\mathrm{d}\Phi_n)^{-1}(\T^1\times Q\times T_k)\right)\right)\lambda(J_l) - {\mu}(\tilde{I}_{n,l})\lambda(Q)\right\rvert + {\mu}(\tilde{I}_{n,l})| \lambda(Q)- \mu(\tilde{S}_n)| \nonumber\\
&\leq \left(\varepsilon_2 + \frac{30}{n}\right){\mu}(\tilde{I}_{n,l})\mu(\tilde{S}_n). \label{lem:eq:1a}
\end{align}
Note that $\mu(\tilde{S}_{n,\gamma})= \mu(\tilde{S}_n)- 2\gamma$, and $|\mu(\tilde{S}_{n,\gamma})- \mu(\tilde{S}_n)|\leq \frac{2}{2k_n^3q_n}\leq \frac{2}{n}\mu(\tilde{S}_n).$ Analogously we obtain,
\begin{align}
   \left\lvert \mu\left(\pi_M\left(\tilde{I}_{n,l,j}\cap (\Phi_n,\mathrm{d}\Phi_n)^{-1}(\T^1\times Q_1\times T_k)\right)\right)\lambda(J_l) - {\mu}(\tilde{I}_{n,l})\mu(\tilde{S}_n)\right \rvert  \leq \frac{32}{n}{\mu}(\tilde{I}_{n,l})\mu(\tilde{S}_n).\label{lem:eq:2a}
\end{align}
Using estimates from (\ref{lem:eq:1a}) and (\ref{lem:eq:2a}) in (\ref{lem:eq:3a}), we have
\begin{align*}
    {\left\lvert \mu\left(\pi_M\left(\tilde{I}_{n,l,j} \cap (\Phi_n,\mathrm{d}\Phi_n)^{-1}\circ (g_n,\mathrm{d}g_n)^{-1}(\tilde{S}_{n,k})\right)\right){\lambda(J_l)} - {\mu}(\tilde{I}_{n,l})\mu(\tilde{S}_n)\right\rvert} \leq  \frac{32}{n}{\mu}(\tilde{I}_{n,l})\mu(\tilde{S}_n).
\end{align*}
Finally, using the triangle inequality with $\lambda(J_l)\geq 1-\frac{50}{k_n^2} \geq \frac{1}{2},$ and $\frac{1-\lambda(J_l)}{\lambda(J_l)}\leq \frac{100}{k_n^2}\leq \frac{4}{n}$ we conclude the claim as
\begin{align*}
    &{\bigg\lvert \mu\left(\pi_M\left(\tilde{I}_{n,l,j} \cap (\Phi_n,\mathrm{d}\Phi_n)^{-1}\circ (g_n,\mathrm{d}g_n)^{-1}(\tilde{S}_{n,k})\right)\right) - {\mu}(\tilde{I}_{n,l})\mu(\tilde{S}_n)\bigg\rvert} \nonumber \\
   &  \leq  \frac{1}{\lambda(J_l)} 
   \left\lvert \mu\left(\pi_M\left(\tilde{I}_{n,l,j} \cap (\Phi_n,\mathrm{d}\Phi_n)^{-1}\circ (g_n,\mathrm{d}g_n)^{-1}(\tilde{S}_{n,k})\right)\right)\lambda(J_l) - {\mu}(\tilde{I}_{n,l})\mu(\tilde{S}_n)\right\rvert \nonumber \\
   & \qquad + \frac{1-\lambda(J_l)}{\lambda(J_l)} {\mu}(\tilde{I}_{n,l}) \mu(\tilde{S}_n)\\
   & \leq   \frac{64}{n}{\mu}(\tilde{I}_{n,l})\mu(\tilde{S}_n) + \frac{4}{n}{\mu}(\tilde{I}_{n,l}) \mu(\tilde{S}_n) \leq \frac{68}{n}{\mu}(\tilde{I}_{n,l})\mu(\tilde{S}_n).\nonumber\qedhere
\end{align*}
\end{proof}
\remark\label{lem:corolarly} The statement of the above lemma still holds for $\tilde{S}_n = S_n\cap \widetilde{G}_n,$ where $S_n$ is a square with a side length in the range $\frac{1}{4k_n}$ to $\frac{2}{k_n}.$
\begin{lemma}\label{thm:6.12}
Let $(f,\mathrm{d}f)= \lim_{n\to\infty} (f_n,\mathrm{d}f_n)$ and $(m_n)_{n\in \N}$ be a sequence of natural numbers satisfying  $d_1(f_n^{m_n}, f^{m_n})< \frac{1}{2^n}$
for all $n\in \N$. Furthermore, let $\hat{\nu}_n$ be a sequence of partial partitions satisfying $\hat{\nu}_n \to \epsilon$ and the following property: For every $3$-dimensional cube $A\subseteq \mathbb{P} \mathrm{T}M$ and for every $\varepsilon\in (0,1],$ there exists $N\in \N$ such that for every $n\geq N,$ and every $\Gamma_n\in \hat{\nu}_n,$ we have 
\begin{align}\label{eqn:6.2a}
    |\bar{\mu}(\Gamma_n\cap (f_n,\mathrm{d}f_n)^{-m_n}(A)) - \bar{\mu}(\Gamma_n)\bar{\mu}(A)|\leq \varepsilon\bar{\mu}(\Gamma_n)\bar{\mu}(A).
\end{align}
Then $(f,\mathrm{d}f)$ is weakly mixing on $\mathbb{P} \mathrm{T}M.$
\end{lemma}
\begin{proof}
The proof follows directly from \cite[Lemma 6.1, 6.2]{GKu15} for dimension $m=3$.
\end{proof}
\subsection{Criterion for smooth weakly mixing diffeomorphisms}
Here, we prove the criterion for weak mixing on $\mathbb{P} \mathrm{T}M.$ 
 \begin{proposition}{(Criterion for weak mixing in smooth setting)} \label{prop:6.6.3a}
  Let $f_n=H_n\circ R_{\alpha_{n+1}}\circ H_n^{-1},$ defined by (\ref{eq:3a}) and (\ref{eq:3:3b}), such that $(f_n,\mathrm{d}f_n)$ converges to $(f,\mathrm{d}f)$ and satisfies { $d_1(f_n^{m_n}, f^{m_n})< \frac{1}{2^n}$}  for every $n\in \N$. Additionally we assume that the sequence $(H_n)_{n\in \N}$ satisfies condition \ref{item:P1}, and the map $g_n$ is as in Proposition \ref{lemma:3a}.
Consider a sequence of partial partitions $(\hat{\eta}_n)_{n\in \mathbb{N}}$, defined as in (\ref{eq:6:1a}), such that $\hat{\eta}_n \to \epsilon$. Additionally, let the partial partition $\hat{\nu}_n$ be defined as 
\begin{align}
\hat{\nu}_n= \left\{\Gamma_{n,j}= (H_{n-1},\mathrm{d} H_{n-1})\circ (g_n,\mathrm{d}g_n)(\hat{I}_{n,j})\ : \ \ \ \hat{I}_{n,j}\in \hat{\eta}_n\right\},
\end{align}
and suppose that $\hat{\nu}_n \to \epsilon.$ Suppose for a sequence $(m_n)_{n\in \N}$ and diffeomorphisms $\Phi_n \coloneqq \phi_n\circ R_{\alpha_{n+1}}^{m_n}\circ \phi_n^{-1}$ that $(\Phi_n,\mathrm{d}\Phi_n):\mathbb{P} \mathrm{T}M\to \mathbb{P} \mathrm{T}M$ is $(\gamma,\delta,\varepsilon_1,\varepsilon_2)$-distributing all partition elements of $\hat{\eta}_n$ with 
$\gamma < \frac{1}{k_n^3q_n}, \delta < \frac{50}{k_n^2}, \varepsilon_1, \varepsilon_2<\frac{1}{n}$. Then  $(f,\mathrm{d}f)= \lim_{n\to\infty}(f_n,\mathrm{d}f_n)$ is weakly mixing with respect to the measure $\overline{\mu}$ on $\mathbb{P} \mathrm{T}M$.
\end{proposition}
\begin{proof}
Fix $\epsilon>0$ and $\hat{I}_{n,j} = \cup_{l=0}^{k_n-1}\tilde{I}_{n,l,j}\in \hat{\eta}_n.$ Let $S_n$ be a square of side length equal to $\frac{1}{k_n}$ contained in $\T^1 \times [\delta,1-\delta]$. Using the good domain $\widetilde{G}_n$ from (\ref{eqn:future_domain}), we define $\tilde{S}_n =S_n \cap \widetilde{G}_n$. 
For $k\in \{0,\dots , k_n-1\}$, let $\tilde{S}_{n,k}= \tilde{S}_n\times T_k$, where $T_k =\left[\frac{k}{k_n},\frac{k+1}{k_n}\right]$. Consider $C_{n,k}= (H_{n-1},\mathrm{d} H_{n-1})(\tilde{S}_{n,k})$ and $\Gamma_{n,l,j} = (H_{n-1},\mathrm{d} H_{n-1})\circ (g_n,\mathrm{d}g_n)(\tilde{I}_{n,l,j})$. We apply  Lemma~\ref{lemm~measurepreservation} and Lemma~\ref{lem: 6.6} to conclude that for any $ j, k\in \{0,\ldots,k_n-1\},$ there exists $l\in \{0,1,\dots,k_n-1\}$ satisfying $k\equiv l+j\mod k_n,$ and we have  
\begin{align}
    &\bar{\mu}(\Gamma_{n,l,j} \cap (f_n,\mathrm{d}f_n)^{-m_n}(C_{n,k})) \nonumber \\
    \leq & \frac{1}{k_n} {\mu}(\pi_M(\tilde{I}_{n,l,j}\cap (\Phi_n,\mathrm{d}\Phi_n)^{-1}\circ(g_n,\mathrm{d}g_n)^{-1}(\tilde{S}_{n,k}))) +  \frac{36(k_n-1)}{k_n^5}\mu(\tilde{I}_{n,l})\nonumber\\
    \leq & \left(1+\frac{68}{n}\right)\frac{1}{k_n} {\mu}(\tilde{I}_{n,l})\mu(\tilde{S}_n) +  \frac{36(k_n-1)}{k_n^5}\frac{1}{\bar{\mu}(C_{n,k})} \mu(\tilde{I}_{n,l})\bar{\mu}(C_{n,k})\nonumber\\
    \leq& \left(1+\frac{68}{n}\right){\mu}(\tilde{I}_{n,l})\bar{\mu}(C_{n,k}) + \frac{2}{n}\mu(\tilde{I}_{n,l})\bar{\mu}(C_{n,k}) \leq \left(1+\frac{70}{n}\right){\mu}(\tilde{I}_{n,l})\bar{\mu}(C_{n,k})\nonumber.
\end{align}
Here, we used the relation: $\mu(C_{n,k}) = \frac{\mu(\tilde{S}_n)}{k_n},$ and $\mu(\tilde{S}_n)\geq (1- \frac{36}{k_n^4})\mu(S_n),$ which implies for $n>4,$ $\frac{1}{\mu(\tilde{S}_n)}\leq (1-\frac{36}{k_n^4})^{-1}\frac{1}{\mu(S_n)} \leq \frac{2}{\mu(S_n)} \leq 2 k_n^{2}$. Similarly, we conclude that $\bar{\mu}(\Gamma_{n,l,j} \cap (f_n,\mathrm{d}f_n)^{-m_n}(C_{n,k}))$ $ \geq \left(1-\frac{70}{n}\right){\mu}(\tilde{I}_{n,l})\bar{\mu}(C_{n,k}).$ Altogether, we get 
\begin{align}\label{eqn:5:5:28}
    |\bar{\mu}(\Gamma_{n,l,j} &\cap (f_n,\mathrm{d}f_n)^{-m_n}(C_{n,k}))- \mu(\tilde{I}_{n,l})\bar{\mu}(C_{n,k})| \leq \frac{70}{n}{\mu}(\tilde{I}_{n,l})\bar{\mu}(C_{n,k}). 
\end{align}
Furthermore, consider $\Gamma_{n,j}= \bigcup_{l=0}^{k_n-1}\Gamma_{n,l,j}.$ With $\bar\mu(\Gamma_{n,j})=\sum_{l=0}^{k_n-1} \bar{\mu}(\Gamma_{n,l,j}),$ and considering $\tilde{c}_{n,l}$ with  ${S}_n' = \T^2$ in Lemma~\ref{lemm~measurepreservation},  we conclude that 
\begin{align}\label{eqn:5:5:29}
|\bar{\mu}(\Gamma_{n,j})- \mu(\tilde{I}_{n,l})|\leq \frac{36}{k_n^4}\mu(\tilde{I}_{n,l}).
\end{align}
Additionally, using the fact that for any $ k,j\in \{0,\ldots,k_n-1\}$ there exists a unique $l\in \{0,1,\dots,k_n-1\}$ satisfying $k\equiv l+j\mod k_n,$ we have
\begin{align}\label{eqn:Distri_Gamma}
    \Gamma_{n,j}\cap (f_n,\mathrm{d}f_n)^{-m_n}(C_{n,k}) = \Gamma_{n,l,j}\cap (f_n,\mathrm{d}f_n)^{-m_n}(C_{n,k}). 
    \end{align} 
Since $C_{n,k}\subset \mathbb{P} \mathrm{T}M$, with $C_{n,k} = (H_{n-1},\mathrm{d} H_{n-1})(\tilde{S}_{n,k})$ of $\text{diam}(\tilde{S}_{n,k})\leq \frac{3}{k_n},$ we can use condition \ref{item:P1} and conclude that
$\text{diam}(C_{n,k})\leq \text{diam}((H_{n-1},\mathrm{d} H_{n-1})(\tilde{S}_{n,k}))
< \frac{1}{n^2},$
that is, $\text{diam}(C_{n,k})\rightarrow 0$ as $n\rightarrow \infty.$

Let $\Meng{S^i_n}{i\in \Lambda_n}$ with a finite index set $\Lambda_n$ be a collection of squares $S^i_n$ of side length $\frac{1}{k_n}$ contained in $\T^1 \times [\delta,1-\delta]$ such that $\mu(\cup_{i \in \Lambda_n} S_n^i )\geq 1-4 \delta$. Let $\tilde{S}_{n}^{i} = S_{n}^{i}\cap \widetilde{G}_n$ and $\tilde{S}_{n,k}^{i} = \tilde{S}_n \times T_k$ with $k\in \{0,\ldots,k_n-1\}$. Considering sets
$C_{n,k}^{i} = (H_{n-1},\mathrm{d} H_{n-1})(\tilde{S}_{n,k}^{i})$ and index set $\tilde{\Lambda}_n = \Meng{(i,k)}{i\in \Lambda_n, k\in \{0,\ldots,k_n-1\}}$, we can conclude that
\begin{align*}
\bar{\mu}&\left(\bigcup_{(i,k)\in \tilde{\Lambda}_n}C_{n,k}^{i}\right)= \sum_{(i,k)\in \tilde{\Lambda}_n}\bar{\mu}(C_{n,k}^{i}) = \sum_{i\in \Lambda_n} \sum_{k=0}^{k_n-1} \bar{\mu}(C_{n,k}^{i}) =  \sum_{i\in \Lambda_n} \sum_{k=0}^{k_n-1}\frac{1}{k_n}\mu(\tilde{S}_n^{i}) \\
&\geq  \sum_{i\in \Lambda_n} \left(1-\frac{36}{k_n^4}\right)\mu(S_n^{i}) \geq \left(1-\frac{36}{k_n^4}\right)(1-4\delta) \geq \left(1-\frac{36}{k_n^4}\right)\left(1-\frac{200}{k_n^2}\right) \geq \left(1-\frac{236}{n^4}\right).
\end{align*}
Thus, $\bar{\mu}\left(\cup_{(i,k)\in \tilde{\Lambda}_n}C_{n,k}^{i}\right) \to 1$ as $n \to \infty.$
Therefore, any cube $A\subset \mathbb{P} \mathrm{T}M$ can be approximated by a finite disjoint union of sets of the form $C_{n,k}$: For $n$ sufficiently large there are sets $A_1= \bigcup_{(i,k)\in \mathcal{C}_n^1}C_{n,k}^i$ and $A_2= \bigcup_{(i,k)\in \mathcal{C}_n^2}C_{n,k}^i$ with finite sets $\mathcal{C}_n^1$ and $\mathcal{C}_n^2$ of indices such that $A_1\subseteq A \subseteq A_2$ and $|\bar{\mu}(A)- \bar{\mu}(A_i)|\leq \frac{\epsilon}{3}\bar{\mu}(A)$ for $i =1,2.$  Additionally, we choose $n$ large enough such that  $ \frac{142}{n}< \frac{\epsilon}{3}$ holds. It follows, using estimates (\ref{eqn:5:5:28}), (\ref{eqn:5:5:29}), (\ref{eqn:Distri_Gamma}), and for $n>4, \left(1-\frac{36}{k_n^4}\right)^{-1} < 2$ that:
\begin{align*}
&\bar{\mu}(\Gamma_{n,j}\cap (f_n,\mathrm{d}f_n)^{-m_n}(A))- \bar{\mu}(\Gamma_{n,j})\bar{\mu}(A) \nonumber \\
\leq & \bar{\mu}(\Gamma_{n,j}\cap (f_n,\mathrm{d}f_n)^{-m_n}(A_2))- \bar{\mu}(\Gamma_{n,j})\bar{\mu}(A_2) + \bar{\mu}(\Gamma_{n,j})(\bar{\mu}(A_2)- \bar{\mu}(A)) \nonumber \\
\leq & \bar{\mu}(\cup_{l=0}^{k_n-1}\Gamma_{n,l,j}\cap (f_n,\mathrm{d}f_n)^{-m_n}(\cup_{(i,k)\in \mathcal{C}_n^2}C_{n,k}^i))- \bar{\mu}(\Gamma_{n,j})\bar{\mu}(\cup_{(i,k)\in \mathcal{C}_n^2} C_{n,k}^i) + \frac{\epsilon}{3}\bar{\mu}(\Gamma_{n,j})\bar{\mu}(A)  \nonumber \\
 \leq & \sum_{(i,k)\in \mathcal{C}_n^2}\left( \bar{\mu}(\Gamma_{n,l,j}\cap (f_n,\mathrm{d}f_n)^{-m_n}(C_{n,k}^i))- \mu(\tilde{I}_{n,l})\bar{\mu}(C_{n,k}^i) \right) +  \sum_{(i,k)\in \mathcal{C}_n^2}\bar{\mu}(C_{n,k}^i)(\mu(\tilde{I}_{n,l}) - \bar{\mu}(\Gamma_{n,j})) \\
& \qquad \qquad \qquad \qquad \qquad \qquad \qquad \qquad  \qquad \qquad \qquad \qquad \quad   + \frac{\epsilon}{3}\bar{\mu}(\Gamma_{n,j})\bar{\mu}(A)  \nonumber \\
 \leq & \sum_{(i,k)\in \mathcal{C}_n^2} \left(\frac{70}{n}+ \frac{36}{k_n^4}\right){\mu}(\tilde{I}_{n,l})\bar{\mu}(C_{n,k}^i) + \frac{\epsilon}{3}\bar{\mu}(\Gamma_{n,j})\bar{\mu}(A) \nonumber \\
 \leq & \sum_{(i,k)\in \mathcal{C}_n^2} \frac{71}{n}\left(1-\frac{36}{k_n^4}\right)^{-1}\bar{\mu}(\Gamma_{n,j})\bar{\mu}(C_{n,k}^i) + \frac{\epsilon}{3}\bar{\mu}(\Gamma_{n,j})\bar{\mu}(A) \nonumber \\
  \leq & \frac{142}{n} \bar{\mu}(\Gamma_{n,j})\bar{\mu}(A) + \frac{142}{n}\bar{\mu}(\Gamma_{n,j})(\bar{\mu}(A_2)- \bar{\mu}(A))+\frac{\epsilon}{3}\bar{\mu}(\Gamma_{n,j})\bar{\mu}(A) \nonumber \\
 \leq &  \frac{\epsilon}{3}\bar{\mu}(\Gamma_{n,j})\bar{\mu}(A) + \frac{\epsilon}{3}\bar{\mu}(\Gamma_{n,j})\bar{\mu}(A) +  \frac{\epsilon}{3}\bar{\mu}(\Gamma_{n,j})\bar{\mu}(A)\leq \epsilon\bar{\mu}(\Gamma_{n,j})\bar{\mu}(A). \label{eqn:weak_mixing_equation}
\end{align*}
Analogously, we estimate that $\bar{\mu}(\Gamma_{n,j}\cap (f_n,\mathrm{d}f_n)^{-m_n}(A))- \bar{\mu}(\Gamma_{n,j})\bar{\mu}(A) > -\epsilon\bar{\mu}(\Gamma_{n,j})\bar{\mu}(A).$ Both estimates together yield equation (\ref{eqn:6.2a}) required for Lemma \ref{thm:6.12}. This implies that $(f,\mathrm{d}f)$ is weakly mixing with respect to $\bar{\mu}$ on the $\mathbb{P} \mathrm{T}M.$
\end{proof}

\subsection{Criterion for analytic weakly mixing diffeomorphisms}
We build on our criterion for weak mixing deduced in the previous subsections and show that the weak mixing property still holds for real-analytic AbC diffeomorphisms obtained by utilizing sufficiently close analytic approximations of the conjugation maps from the smooth constructions. We start by showing that such approximations still give good distribution properties in the sense of Definition~\ref{def:6.2a}.
\begin{lemma} \label{lem:6.6.4c}
Consider $\phi_n \in \textup{Diff}^{\infty}(\T^2,\mu)$ and $\hat{\phi}_n\in \textup{Diff}_{\infty}^{\omega}(\T^2,\mu)$ satisfying
$d_{2}(\phi_n,\hat{\phi}_n\restriction_{\T^2}) < \epsilon_n$, where $\epsilon_n>0$ and $\mathfrak{d}_n>0$ are given by (\ref{eqn:4:4.2}). Denote $\Phi_n= \phi_n \circ R_{\alpha_{n+1}}^{m_n}\circ \phi_n^{-1}$ and $\hat{\Phi}_n= \hat{\phi}_n \circ R_{\alpha_{n+1}}^{m_n}\circ \hat{\phi}_n^{-1}.$ If the map $(\Phi_n,\mathrm{d}\Phi_n)$ is $(\gamma,\delta,\varepsilon_1,\varepsilon_2)$-distributing the elements $I\in \hat{\eta}_n$ in the sense of Definition~\ref{def:6.2a}, then $(\hat{\Phi}_n,\mathrm{d}\hat{\Phi}_n)$ is $(\gamma',\delta',\varepsilon_1',\varepsilon_2')$-distributing the elements $I\in \hat{\eta}_n$ with $\gamma' = \gamma + \frac{1}{2^n}, \ \delta' = \delta + \frac{1}{2^n}, \ \varepsilon_1' = 2\cdot \varepsilon_1 + \frac{3}{2^n},$ and $\varepsilon_2' = 2\cdot \varepsilon_2 + \frac{3}{2^n}.$ 
\end{lemma}
\begin{proof}
By using the triangle inequality and applying the mean value theorem, we can deduce the following proximity statement:
\begin{align}    
 &d_1(\Phi_n, \hat{\Phi}_n\restriction_{\T^2}) = d_1({\phi}_n \circ R_{\alpha_{n+1}}^{m_n}\circ {\phi}_n^{-1},  (\hat{\phi}_n \circ R_{\alpha_{n+1}}^{m_n}\circ \hat{\phi}_n^{-1})\restriction_{\T^2})\nonumber \\
 \leq &  d_1({\phi}_n \circ R_{\alpha_{n+1}}^{m_n}\circ {\phi}_n^{-1},  {\phi}_n \circ R_{\alpha_{n+1}}^{m_n}\circ (\hat{\phi}_n^{-1})\restriction_{\T^2}) +  d_1({\phi}_n \circ R_{\alpha_{n+1}}^{m_n}\circ (\hat{\phi}_n^{-1})\restriction_{\T^2},  (\hat{\phi}_n \circ R_{\alpha_{n+1}}^{m_n}\circ \hat{\phi}_n^{-1})\restriction_{\T^2})\nonumber \\
 \leq  & \vertiii{\phi_n}_{2}\cdot d_1(\phi_n, \hat{\phi}_n\restriction_{\T^2})+ \max(\vertiii{\phi_n}_1,\vertiii{\hat{\phi}_n}_1)\cdot d_1(\phi_n, \hat{\phi}_n\restriction_{\T^2}) \nonumber\\
 \leq & (2\vertiii{\phi_n}_{2}+\epsilon_n)\cdot d_1(\phi_n, \hat{\phi}_n\restriction_{\T^2})  \leq (2\vertiii{\phi_n}_{2}+1)\cdot \epsilon_n \leq \mathfrak{d}_n.
\end{align}
With the definition of $(\Phi_n, \mathrm{d}\Phi_n)$ being $(\gamma, \delta, \varepsilon_1, \varepsilon_2)$-distributing, it satisfies four conditions: The first two conditions show that ${\Phi}_n(\tilde{I}_{n,l})$ is contained within $[c_l, c_l + \gamma] \times \T^1$ for some constant $c_l\in \T^1$ and that
there exists a set $J_l$ such that $J_l \subseteq \pi_{r}(\Phi_n (\tilde{I}_{n,l}))$ with $1 - \delta \leq \lambda(J_l) \leq 1$ for all $l \in \{0, \ldots, k_n-1\}$.  We denote $J_l' = J_l \cap \pi_{r}(\hat{\Phi}_n(\tilde{I}_{n,l}))$, which satisfies $\lambda(J_l') > 1 - (\delta + 2\cdot \mathfrak{d}_n).$  With $\mathfrak{d}_n < \frac{1}{2^{n+1}},$ this allows us to choose  $\gamma' = \gamma + \frac{1}{2^n}$ and $\delta' = \delta + \frac{1}{2^n}.$

For any subset $\tilde{J}\subset J_l'$, we consider two subsets, $\tilde{J}_1$ and $\tilde{J}_2,$ inside $J_l$ such that $\tilde{J}_1\subset \tilde{J} \subset \tilde{J}_2$. Additionally, we require $\lambda(\tilde{J}_i \triangle \tilde{J})\leq \frac{1}{2^n}\cdot \lambda(\tilde{J})$ for $i=1,2$, and for sufficiently large $n$ that $\text{dist}(\partial \tilde{J}_i, \partial \tilde{J})> \mathfrak{d}_n$ for $i=1,2$.

Given the other conditions in the definition of $(\Phi_n, \mathrm{d}\Phi_n)$ as $(\gamma, \delta, \varepsilon_1, \varepsilon_2)$ distributing, we can conclude that for any $\tilde{J}_i \subset J_{l}$ 
for $i=1,2$, and for all $l \in \{0,1,\ldots, k_n-1\}$ 
\begin{align}
\lvert {\mu}(\tilde{I}_{n,l}\cap \Phi_n^{-1}(\T^1\times \tilde{J}_i)\lambda(J_{l})) - \lambda(\tilde{J}_i){\mu}(\tilde{I}_{n,l}) \rvert &\leq \varepsilon_1  \lambda(\tilde{J}_i)\mu(\tilde{I}_{n,l}) \label{eqn:8:8.2a}
\end{align}
and for any $j,k\in \{0,1,\ldots, k_n-1\}$, there exists a $l\in \{0,1,\ldots, k_n-1\}$ such that $k \equiv l+j \mod{k_n},$ for $ \tilde{J}_i \subset J_l'\subseteq J_l,$ and $T_k= \left[\frac{k}{k_n},\frac{k+1}{k_n}\right]$ we have
\begin{align}
\left\lvert \mu\left(\pi_M\left(\tilde{I}_{n,l,j}\cap (\Phi_n,\mathrm{d}\Phi_n)^{-1}(\T^1\times \tilde{J}_i \times T_k)\right)\right)\lambda(J_l) - \lambda(\tilde{J}_i) \mu(\tilde{I}_{n,l})\right\rvert  &\leq \varepsilon_2 \lambda(\tilde{J}_i) \mu(\tilde{I}_{n,l}).\label{eqn:8:8.2b} 
\end{align}
Let us denote the following subsets of $\T^2$ as $A_1 = \T^1 \times \tilde{J}_1$, $A = \T^1 \times \tilde{J}$, and $A_2 = \T^1 \times \tilde{J}_2$, which satisfy the following conditions: $A_1\subset A \subset A_2$, $\mu(A_i \triangle A)\leq \mu(\tilde{J}_i \triangle \tilde{J}) \leq\frac{1}{2^n} \cdot \mu(A),$ and $\text{dist}(\partial A_i, \partial A)> \text{dist}(\partial \tilde{J}_i, \partial \tilde{J})> \mathfrak{d}_n$ for $i=1,2$. 
With Lemma \ref{lem:8:8.2}, we can establish the following relations: if $I \in \Phi_n^{-1}(A_1)$, then $I \in  \hat{\Phi}_n^{-1}(A)$, and if $I\in \hat{\Phi}_n^{-1}(A)$, then $I\in \Phi_n^{-1}(A_2)$. This implies
\begin{align*}
\mu(\tilde{I}_{n,l}\cap \Phi_n^{-1}(A_1))\lambda({J}_l')  &\leq \mu(\tilde{I}_{n,l}\cap \hat{\Phi}_n^{-1}(A))\lambda({J}_l') \leq \mu(\tilde{I}_{n,l}\cap \Phi_n^{-1}(A_2))\lambda(J_l').
\end{align*}
Altogether with equation (\ref{eqn:8:8.2a}) and relation $|\lambda(J_l')-\lambda(J_l)|\leq 2 \mathfrak{d}_n\lambda(J_l)$ it holds:  
\begin{align}
    (1-2 \mathfrak{d}_n)(1-\varepsilon_1)\mu(\tilde{I}_{n,l})\lambda(\tilde{J}_1) &\leq \mu(\tilde{I}_{n,l}\cap \hat{\Phi}_n^{-1}(A)) \lambda(J_{l}') \leq (1+ 2\mathfrak{d}_n)(1+\varepsilon_1)\mu(\tilde{I}_{n,l})\lambda(\tilde{J}_2).
    \label{eqn:8:8.2c}
\end{align}
Furthermore, we can obtain the following estimates with $\mathfrak{d}_n \leq \frac{1}{2^{n+1}}$:
\begin{align}
&\mu(\tilde{I}_{n,l}\cap \hat{\Phi}_n^{-1}( A))\lambda(J_{l}')  -\mu(\tilde{I}_{n,l})\lambda(\tilde{J}) \nonumber\\
&\leq 
 (\varepsilon_1+ 2\mathfrak{d}_n(1+\varepsilon_1))\mu(\tilde{I}_{n,l})\lambda (\tilde{J}) + (1+2\mathfrak{d}_n)(1+\varepsilon_1)\mu(\tilde{I}_{n,l})(\lambda(\tilde{J}_2)- \lambda(\tilde{J})) \nonumber\\
   &\leq (\varepsilon_1+ 2\mathfrak{d}_n(1+\varepsilon_1))\mu(\tilde{I}_{n,l})\lambda (\tilde{J}) + (1+2\mathfrak{d}_n)(1+\varepsilon_1)\mu(\tilde{I}_{n,l})\lambda(\tilde{J}_2\triangle\tilde{J}) \nonumber\\
   &\leq \left(2\cdot\varepsilon_1+ \frac{3}{2^n}\right)\mu(\tilde{I}_{n,l})\lambda (\tilde{J}). \nonumber
\end{align}
Similarly, we deduce the other side of inequality using (\ref{eqn:8:8.2c}), and altogether, we obtain that  
$$\lvert \mu(\tilde{I}_{n,l}\cap \hat{\Phi}_n^{-1}(\T^1\times \tilde{J}))\lambda(J_{l}')  -\mu(\tilde{I}_{n,l})\lambda(\tilde{J}) \rvert \leq \left(2\cdot\varepsilon_1+ \frac{3}{2^n}\right)\mu(\tilde{I}_{n,l})\lambda (\tilde{J}).$$
Analogously, we deduce using the estimate (\ref{eqn:8:8.2b}) that for any pair $j,k\in \{0,1,\ldots, k_n-1\}$, there exists $l\in \{0,1,\ldots, k_n-1\}$ such that $k \equiv l+j \mod{k_n},$ and for $\tilde{J} \subset J_l',$ $T_k= \left[\frac{k}{k_n},\frac{k+1}{k_n}\right]$ we have
\begin{align}
\left\lvert\mu\left(\pi_M\left(\tilde{I}_{n,l,j}\cap (\hat{\Phi}_n,\mathrm{d} \hat{\Phi}_n)^{-1}(\T^1\times \tilde{J} \times T_k)\right)\right) \lambda(J_l') - \mu(\tilde{I}_{n,l})\lambda(\tilde{J})\right\rvert \leq \left(2\cdot\varepsilon_2+ \frac{3}{2^n}\right)\mu(\tilde{I}_{n,l})\lambda(\tilde{J}). \nonumber  
    \end{align}
In summary, this yields that $(\hat{\Phi}_n,\mathrm{d} \hat{\Phi}_n)$ is $(\gamma',\delta', \varepsilon_1', \varepsilon_2')$-distributing $\hat{I}_{n,j}\in \hat{\eta}_n.$
\end{proof}
\begin{lemma}\label{lem:6.4d} Let $g_n$ be as in Proposition {\ref{lemma:3a}} and satisfy the property of Lemma \ref{lem: 6.6} with the map  
$(\hat{\Phi}_n, \mathrm{d}\hat{\Phi}_n)$ which is supposed to $(\gamma',\delta',\varepsilon_1',\varepsilon_2')$ distribute $\hat{I}_{n,j} \in \hat{\eta}_n,$ where $\hat{\eta}_n$ is the partial partition of $\mathbb{P} \mathrm{T}M$ constructed in (\ref{eq:6:1a}). Furthermore, let $\hat{g}_n\in \textup{Diff}_{\infty}^{\omega}(\T^2,\mu)$ such that $d_k(g_n,\hat{g}_n\restriction_{\T^2}) < \epsilon_n$, where $\epsilon_n$ is given by (\ref{eqn:4:4.2}).  Let $S_n$ be a square of side length equal to $\frac{1}{k_n}$ lying in $\T^1 \times [\delta,1-\delta]$. Denote $\tilde{S}_n =S_n \cap \widetilde{G}_n$ with $\widetilde{G}_n$ defined in (\ref{eqn:future_domain}). For any $k\in \{0,\dots , k_n-1\},$ denote $\tilde{S}_{n,k}= \tilde{S}_n\times T_k$ with $T_k =\left[\frac{k}{k_n},\frac{k+1}{k_n}\right]$. 

Then for every pair $j,k\in \{0,1,\ldots,k_n-1\}$  there exist $l\in  \{0,1,\ldots,k_n-1\}$ such that $k\equiv (l+j)\mod k_n$ and it holds:
$${\left\lvert \mu\left(\pi_M\left(\tilde{I}_{n,l,j}\cap (\hat{\Phi}_n,\mathrm{d}\hat{\Phi}_n)^{-1}\circ (\hat{g}_n,\mathrm{d}\hat{g}_n)^{-1}(\tilde{S}_{n,k})\right)\right) - {{\mu}}(\tilde{I}_{n,l}){\mu(\tilde{S}_{n})}\right\rvert} \leq 2\left(\varepsilon_2'+\frac{69}{n}\right) {\mu}(\tilde{I}_{n,l})\mu(\tilde{S}_{n}).$$
\end{lemma} 
\begin{proof}
The proof follows a similar line of reasoning as the previous lemma. For notation, let $\tilde{S}_n = S_n \cap \widetilde{G}_n$ and $S_n$ be a square of side length $\frac{1}{k_n}$ lying in $\T^1 \times [\delta',  1-\delta']$. Now, consider two squares, $S_n^1$ and $S_n^2$, such that $S_n^1\subset {S}_n \subset S_n^2$, $\mu(S_n^i \triangle {S}_n) \leq \frac{1}{2^n}\cdot \mu(\tilde{S}_n)$, and $\text{dist}(\partial S_n^i, \partial \tilde{S}_n)> \mathfrak{d}_n$ for $i=1,2$.
Denote the cuboid $\tilde{S}_{n,k} = \tilde{S}_n \times T_k$ and $\tilde{S}_{n,k}^i = \tilde{S}_n^i \times T_k$ for $i=1,2$ where $\tilde{S}_{n}^i = S_n^i\cap \widetilde{G}_n.$ These cuboids satisfy the following conditions: $\tilde{S}_{n,k}^1\subset \tilde{S}_{n,k}\subset \tilde{S}_{n,k}^2.$

Using Lemma \ref{lem: 6.6} along with Remark ~\ref{lem:corolarly} for the maps $(g_n,\mathrm{d}g_n)$ and $(\hat{\Phi}_n, \mathrm{d}\hat{\Phi}_n)$, characterized by $(\gamma',\delta',\varepsilon_1',\varepsilon_2')$ distribute $\hat{I}_{n,j} = \cup_{l =0}^{k_n-1}\tilde{I}_{n,l,j} \in \hat{\eta}_n$  for the cuboid $\tilde{S}_{n,k}^i,$ $i=1,2$: for  any $j,k\in \{0,1,\ldots,k_n-1\}$  there exist $l\in  \{0,1,\ldots,k_n-1\}$ such that $k\equiv (l+j)\mod k_n,$ and the following estimates holds:
\begin{align} \label{eqn:5:5.7a}
    {\left\lvert {\mu}\left(\pi_M\left(\tilde{I}_{n,l,j}\cap (\hat{\Phi}_n,\mathrm{d}\hat{\Phi}_n)^{-1}\circ ({g}_n,\mathrm{d}{g}_n)^{-1}(\tilde{S}_{n,k}^i)\right)\right) - {\mu}(\tilde{I}_{n,l}){\mu(\tilde{S}_{n}^i)}\right\rvert} \leq \left(\varepsilon_2'+ \frac{68}{n}\right) {\mu}(\tilde{I}_{n,l})\mu(\tilde{S}_{n}^i).
\end{align}
Using the mean value theorem along with  $d_1(g_n,\hat{g}_n\restriction_{\T^2})< \epsilon_n,$ and equation (\ref{eqn:4:4.2}), we have $d_1(\hat{\Phi}_n^{-1}\circ g_n^{-1},\hat{\Phi}_n^{-1}\circ\hat{g}_n^{-1}\restriction_{\T^2})< \mathfrak{d}_n.$ Furthermore, considering the property of Lemma \ref{lem:8:8.2}, we can establish the following relations:
\begin{align*}
{\mu}(\pi_M(\tilde{I}_{n,l,j}\cap (\hat{\Phi}_n,\mathrm{d}\hat{\Phi}_n)^{-1}\circ ({g}_n,\mathrm{d}{g}_n)^{-1}(\tilde{S}_{n,k}^1))) &\leq {\mu}(\pi_M(\tilde{I}_{n,l,j}\cap (\hat{\Phi}_n,\mathrm{d}\hat{\Phi}_n)^{-1}\circ (\hat{g}_n,\mathrm{d}\hat{g}_n)^{-1}(\tilde{S}_{n,k})));\\ 
{\mu}(\pi_M(\tilde{I}_{n,l,j}\cap (\hat{\Phi}_n,\mathrm{d}\hat{\Phi}_n)^{-1}\circ (\hat{g}_n,\mathrm{d}\hat{g}_n)^{-1}(\tilde{S}_{n,k})))&\leq 
{\mu}(\pi_M(\tilde{I}_{n,l,j}\cap (\hat{\Phi}_n,\mathrm{d}\hat{\Phi}_n)^{-1}\circ ({g}_n,\mathrm{d}{g}_n)^{-1}(\tilde{S}_{n,k}^2))).
\end{align*}
Altogether with equation~(\ref{eqn:5:5.7a}), it holds:
\begin{align} 
&{\mu}\left(\pi_M\left(\tilde{I}_{n,l,j}\cap (\hat{\Phi}_n,\mathrm{d}\hat{\Phi}_n)^{-1}\circ (\hat{g}_n,\mathrm{d}\hat{g}_n)^{-1}(\tilde{S}_{n,k})\right)\right) - {\mu}(\tilde{I}_{n,l}){\mu(\tilde{S}_{n})} \nonumber \\
\leq & \bigg(\bigg(\varepsilon_2'+ \frac{68}{n}\bigg)\bigg(1+ \frac{1}{2^n}\bigg) + \frac{1}{2^n}\bigg){\mu}(\tilde{I}_{n,l})\mu(\tilde{S}_{n}) \leq 2\left(\varepsilon_2'+\frac{69}{n}\right) {\mu}(\tilde{I}_{n,l})\mu(\tilde{S}_{n}) .\label{eqn:5:5.7u}
\end{align}
In the subsequent equation, we apply $(1+ \frac{1}{2^n})<2$ and $\frac{1}{2^n} < \frac{2}{n}.$ Similarly, we derive the opposite side of the inequality and thereby validate the claim.
\end{proof} 

\begin{proposition}{(Criterion for analytic weak mixing)} \label{prop:6.6.3b}
  Let $\hat{f}_n=\hat{H}_n\circ R_{\alpha_{n+1}}\circ \hat{H}_n^{-1}$ be defined by (\ref{eqn:4:4.1a}),  such that $(\hat{f}_n,\mathrm{d}\hat{f}_n)$ converges to $(\hat{f},\mathrm{d}\hat{f})$ and satisfies { $d_1(\hat{f}_n, \hat{f})< \frac{1}{2^n} $} for all $n\in \N$. Additionally we assume that the sequence $(H_n)_{n\in \N}$ satisfies condition \ref{item:P1}, and the map $g_n$ is as in Lemma~\ref{lem:6.4d}. Consider a sequence of partial partitions $(\hat{\eta}_n)_{n\in \mathbb{N}}$, defined as in (\ref{eq:6:1a}) such that $(\hat{\eta}_n) \to \epsilon$. Additionally, we assume that the partial partition  
$$\hat{\hat{\nu}}_n= \Meng{\Gamma_{n,j}= (\hat{H}_{n-1},\mathrm{d} \hat{H}_{n-1})\circ (\hat{g}_n,\mathrm{d}\hat{g}_n)(\hat{I}_{n,j})}{ \hat{I}_{n,j}\in \hat{\eta}_n},$$
  satisfies $\hat{\hat{\nu}}_n \to \epsilon.$ Suppose for a sequence $(m_n)_{n\in \N}$ and the diffeomorphism $(\hat{\Phi}_n,\mathrm{d}\hat{\Phi}_n):\mathbb{P} \mathrm{T}M\to \mathbb{P} \mathrm{T}M$ with $\hat{\Phi}_n= \hat{\phi}_n\circ R_{\alpha_{n+1}}^{m_n}\circ \hat{\phi}_n^{-1}$ is $(\gamma,\delta,\varepsilon_1,\varepsilon_2)$-distributing $\hat{I}_{n,j}\in\hat{\eta}_n$ with 
$\gamma < \frac{1}{k_n^3q_n}, \delta < \frac{50}{k_n^2}, \varepsilon_1,\varepsilon_2<\frac{150}{n}$. Then  $(\hat{f},\mathrm{d}\hat{f})= \lim_{n\to\infty}(\hat{f}_n,\mathrm{d}\hat{f}_n)$ is weakly mixing on $\mathbb{P} \mathrm{T}M$.
\end{proposition}

\begin{proof}
The proof follows along the lines of Proposition \ref{prop:6.6.3a}, utilizing Lemma \ref{lem:6.6.4c} and Lemma \ref{lem:6.4d}.
    \end{proof}

\section{Explicit constructions} \label{sec:exp_setup}
\subsection{Sequence of partial partitions}
In this subsection, we define three sequences of partial partitions $\{\eta_n\}_{n\in \N}, \{\zeta_n\}_{n\in \N}$ and $\{\tilde{\eta}_n\}_{n\in \N}$ of the $2$-torus. For each $n\in \N,$  the partial partition $\zeta_n$ is a refined partition of the partial partition $\eta_n,$ and the partition elements of $\tilde{\eta}_n$ are unions of the partition elements of $\zeta_n.$ We define these sequences of partial partitions using the sequences $\{k_n\}_{n\in \N}$ and $\{q_n\}_{n \in \N}$ of rapidly growing natural numbers satisfying $\ref{item:P1}- \ref{item:P2}$.

\subsubsection{Partial partition \texorpdfstring{$\eta_n$}{Lg}} \label{subsubsec:eta}
Consider the following collection of multidimensional intervals in following form: 
\begin{align}\label{eqn:2.2b}
    \check{I}^{u_0,u_1,u_2}_{v_0}= &\left[ \frac{u_0}{2q_n}+ \frac{u_1}{2k_nq_n}+ \frac{u_2}{2k_n^6q_n} + \frac{1}{4n^5k_n^{16}q_n},\frac{u_0}{2q_n}+ \frac{u_1}{2k_nq_n}+ \frac{u_2+1}{2k_n^6q_n}  - \frac{1}{4n^5k_n^{16}q_n}\right]\nonumber \\
    & \times \left[ \frac{v_0}{k_n^5} + \frac{1}{2n^5k_n^{15}},\frac{v_0+1}{k_n^5}- \frac{1}{2n^5k_n^{15}}\right],
\end{align}
where $u_0\in\{0,1,\ldots,2q_n-1\},\ u_1\in\{0,1\ldots,k_n-1\}$ and $v_0,u_2 \in\{1,\ldots,k_n^5-2\}$.
Denote by $\eta_n$ the collection of such intervals $ \check{I}^{u_0,u_1,u_2}_{v_0}$ of $\mathbb{T}^2.$ 

\remark Note that $\eta_n$ is a partial partition of $\mathbb{T}^2$. For every $n\in \N$, $\eta_n$ consists of disjoint sets, covers a set of measure at least  $\left(1-\frac{2}{k_n^{5}}\right)^4\geq 1-\frac{8}{k_n^{5}}$ and for $I\in \eta_n,$ $\text{diam}(I)\leq \frac{\sqrt{2}}{k_n^5} \rightarrow 0$ as $n\rightarrow\infty$. Hence, $(\eta_n)_{n\in\N}$ converges to the decomposition into points. 

\subsubsection{Partial partition \texorpdfstring{$\zeta_n$}{Lg}} \label{subsubsec:zeta}
Consider the following collection of multidimensional intervals in the given form: 
\begin{align}\label{eqn:2.2a}
  I^{u_0,u_1,u_2,u_3,u_4}_{v_0,v_1,v_2}  = & \bigg[ \frac{u_0}{2q_n}+ \frac{u_1}{2k_nq_n}+ \frac{u_2}{2k_n^6q_n} + \frac{u_3}{2k_n^{11}q_n}+ \frac{u_4}{2k_n^{16}q_n} + \frac{1}{4n^5k_n^{26}q_n}, \nonumber \\
     & \ \frac{u_0}{2q_n}+ \frac{u_1}{2k_nq_n}+ \frac{u_2}{2k_n^6q_n} + \frac{u_3}{2k_n^{11}q_n}+ \frac{u_4+1}{2k_n^{16}q_n} - \frac{1}{4n^5k_n^{26}q_n}\bigg]  \nonumber  \\
      \times &\left[ \frac{v_0}{k_n^5}+ \frac{v_1}{2k_n^{11}q_n}+ \frac{v_2}{2k_n^{16}q_n}+ \frac{1}{4n^5k_n^{26}q_n},\frac{v_0}{k_n^5}+ \frac{v_1}{2k_n^{11}q_n}+ \frac{v_2+1}{2k_n^{16}q_n}- \frac{1}{4n^5k_n^{26}q_n}\right],
\end{align}
where $u_0\in\{0,1,\ldots,2q_n-1\},\ u_1\in\{0,1\ldots,k_n-1\},$ $v_0, v_2, u_2,u_3, u_4 \in \{1,\ldots,k_n^5-2\},$ 
and
 $v_1\in\{2q_n,\ldots,2k_n^6q_n-2q_n-1\}.$
Let $\zeta_n$ be the collection of such intervals $ I^{u_0,u_1,u_2,u_3,u_4}_{v_0,v_1,v_2}$.
\remark Note that $\zeta_n$ is a partial partition of $\mathbb{T}^2$. For every $n\in \N$, $\zeta_n$ consists of disjoint sets, covers a set of measure at least  $\left(1-\frac{2}{k_n^5}\right)^8\geq 1-\frac{16}{k_n^5}$ and for $I\in \zeta_n,$ $\text{diam}(I)\leq \frac{\sqrt{2}}{k_n^{15}q_n} \rightarrow 0$ as $n\rightarrow\infty$. Thus, $(\zeta_n)_{n\in\N}$ converges to the decomposition into points. We will use the sequence $(\zeta_n)_{n\in \N}$ in section~\ref{sec:metric} to verify that the assumptions of our criterion in Proposition~\ref{prop:metric} for the existence of a $f$-invariant measurable Riemannian metric are satisfied in our explicit constructions.

\remark\label{rem:2.7a} Note that $\zeta_n$ is a refined partition of $\eta_n$ where each element $\check{I}^{u_0,u_1,u_2}_{v_0}\in \eta_n$ is covered by some union of elements $I^{u_0,u_1,u_2,u_3,u_4}_{v_0,v_1,v_2}\in \zeta_n$ with a measure of at least 
$\left(1-\frac{8}{k_n^5}\right)\mu(\check{I}^{u_0,u_1,u_2}_{v_0}),$
\begin{align}\label{eqn:2.2c}
\bigcup_{v_1=2q_n}^{2k_n^{6}q_n-2q_n-1}\bigcup_{u_3, u_4, v_2=1}^{k_n^5-2}I^{u_0,u_1,u_2,u_3,u_4}_{v_0,v_1,v_2} \subseteq \check{I}^{u_0,u_1,u_2}_{v_0}.
\end{align}

\subsubsection{Partial partition \texorpdfstring{$\tilde{\eta}_n$}{Lg}} \label{subsubsec:Tildeeta}
This collection will be useful in defining a partial partition of the space $\mathbb{P} \mathrm{T}M$ and referring to it as the ``good domain" in our explicit construction. Denote $\tilde{\eta}_n$ be collection of the following elements for $u_0\in \{0,\ldots,2q_n-1\}$ and $ v_0 \in \{1,\ldots,k_n^5-2\}$ as
\begin{align}\label{eqn:7.1a}
\tilde{I}^{u_0}_{v_0} = \bigcup_{u_1 =0}^{k_n-1}\tilde{I}^{u_0,u_1}_{v_0}; \ \text{where} \ 
\tilde{I}^{u_0,u_1}_{v_0}= \bigcup_{u_2=1}^{k_n^5-2} \bigcup_{v_1=2q_n}^{2k_n^{6}q_n-2q_n-1}\bigcup_{u_3, u_4, v_2=1}^{k_n^5-2}I^{u_0,u_1,u_2,u_3,u_4}_{v_0,v_1,v_2}.
\end{align}
\remark\label{rem:assumption_intesection}Note that for each element $\tilde{I}^{u_0,u_1}_{v_0}$, we have 
$
\frac{1}{2k_n^6q_n}\left(1 - \frac{16}{k_n^5}\right) \leq \mu(\tilde{I}^{u_0,u_1}_{v_0}) \leq \frac{1}{2k_n^6q_n},$
and $\tilde{\eta}_n$ covers a set with a measure of at least $\left(1-\frac{25}{k_n^5}\right).$ This verify the partial partition $\tilde{\eta}_n$ satisfies the properties in  \ref{item:D2} and \ref{item:p1}. Furthermore, according to Remark \ref{rem:2.7a}, we have $
\tilde{I}_{v_0}^{u_0,u_1} \subseteq \bigcup_{u_2 = 1}^{k_n^3-1}\check{I}^{u_0,u_1,u_2}_{v_0},$ and  $\mu\left(\bigcup_{u_2 = 1}^{k_n^5-2}\check{I}^{u_0,u_1,u_2}_{v_0} \cap \tilde{I}_{v_0}^{u_0,u_1}\right) \geq \left(1-\frac{8}{k_n^5}\right)\mu\left(\bigcup_{u_2 = 1}^{k_n^5-2}\check{I}^{u_0,u_1,u_2}_{v_0}\right).$ 
\remark\label{item:verify:D5}Furthermore, with $q_{n+1}> {2k_n^{10}q_n^2}$ in \ref{item:P2}, and using the definition of the sequence of partial partitions $\{\tilde{\eta}_n\}_{n \in \mathbb{N}},$ where each $\tilde{I}_{n+1} \in \tilde{\eta}_{n+1}$ is a refined partition element of  $\tilde{I}_n \in \tilde{\eta}_n,$  we can verify property \ref{item:p4}.

\subsection{Smooth conjugation maps \texorpdfstring{$\phi_n$}{Lg} }\label{sec:phi}

\subsubsection{The conjugation map \texorpdfstring{$\Tilde{\phi}_{n}$}{Lg}}

In \cite[section 2.1]{Ku20} the smooth area-preserving diffeomorphism $\Tilde{\phi}_{\lambda,\varepsilon,\mu,\varepsilon_2}$ on $\T^1\times [0,1]$
satisfying the subsequent properties is constructed and referred to as the ``inner rotation of type A". With minor to no modifications, the proposition holds for the case of $[0,1]^2.$ 
\begin{proposition}\label{prop:1a}
 Let $\varepsilon,\varepsilon_2\in(0,\frac{1}{4})$ and $\lambda,\mu\in \N$. Then there is a smooth area-preserving diffeomorphism $\Tilde{\phi}_{\lambda,\varepsilon,\mu,\varepsilon_2}: [0,1]^2\rightarrow [0,1]^2$
such that
\begin{enumerate}
 \item $\Tilde{\phi}_{\lambda,\varepsilon,\mu,\varepsilon_2}$ coincides with the identity on $[0,1]^2\backslash [\varepsilon,1-\varepsilon]^2$;
 \item   $\Tilde{\phi}_{\lambda,\varepsilon,\mu,\varepsilon_2}\left(x+\frac{l}{q_n},y\right) = \Tilde{\phi}_{\lambda,\varepsilon,\mu,\varepsilon_2}(x,y)$ for all $(x,y)\in \left[\frac{l}{q_n},\frac{l+1}{q_n}\right)\times [0,1], l\in \mathbb{Z}.$ 
 \item Let $t_2\in \Z$, $\lceil 2\varepsilon\mu\rceil \leq t_2\leq \mu- \lceil 2\varepsilon\mu\rceil-1, |u_2|\leq \varepsilon_2,$ and $u_1\in (2\varepsilon,\frac{1}{2})$ be of the form $\frac{t_1}{\mu}$ with $t_1\in \Z$. Then we have 
    \begin{align}
\Tilde{\phi}_{\lambda,\varepsilon,\mu,\varepsilon_2}\left(\left[ \frac{u_1}{\lambda}, \frac{1-u_1}{\lambda}\right] \times \left[\frac{t_2+u_2}{\mu},\frac{t_2+1-u_2}{\mu}\right] \right) \nonumber\\
    = \left[\frac{1}{\lambda}-\frac{t_2+1-u_2}{\lambda\mu},\frac{1}{\lambda}-\frac{t_2+u_2}{\lambda\mu}\right] \times [u_1, 1-u_1]; \nonumber
    \end{align}
    \item $\Tilde{\phi}_{\lambda,\varepsilon,\mu,\varepsilon_2}$ acts as isometry on each cuboid 
    $$\left[\frac{t_1+2\varepsilon_2}{\lambda\mu}, \frac{t_1+1-2\varepsilon_2}{\lambda\mu}\right] \times \left[\frac{t_2+2\varepsilon_2}{\mu}, \frac{t_2+1-2\varepsilon_2}{\mu}\right],$$
    where $t_i\in \Z, \lceil 2\varepsilon \mu\rceil \leq t_i\leq \mu-\lceil 2\varepsilon\mu\rceil -1$ for $i=1,2;$
   \item $\vertiii{\Tilde{\phi}_{\lambda,\varepsilon,\mu,\varepsilon_2}}_r \leq c \cdot \lambda^r \mu^{r}$ where $c$ is a constant independent of $\lambda$ and $\mu$. 

\end{enumerate}
\end{proposition}
In our explicit construction, we choose the parameters as follows to define $\Tilde{\phi}_{n}: [0,1]^2\rightarrow [0,1]^2$ by
\begin{align}
\label{eqn:3.1.1b}
\Tilde{\phi}_{n} = \tilde{\phi}_{2k_nq_n,\frac{1}{2k_n^5},k_n^5,\frac{1}{4n^5k_n^{10}}}.
\end{align}
\remark\label{re:4a} The action of the map $\Tilde{\phi}_n$ on the elements of the partial partitions $\zeta_n$ and $\eta_n$ can be described by
\begin{align}
   \tilde{\phi}_n  (\check{I}^{u_0,u_1,u_2}_{v_0}) &= \check{I}^{u_0,u_1,k_n^5-v_0-1}_{u_2},\\
   \tilde{\phi}_n  (I^{u_0,u_1,u_2,u_3,u_4}_{v_0,v_1,v_2}) &= I^{u_0,u_1,k_n^5-v_0-1,u_3,u_4}_{ u_2,v_1,v_2}.
\end{align} 
\subsubsection{The conjugation map \texorpdfstring{$i_n$}{Lg}}
Within this subsection, we introduce what we refer to as the ``inner rotations of type B". These rotations will play a crucial role in establishing the weak mixing property for the  projectivized derivative extension. Specifically, we will use distinct rotation angles applied to specific regions, which allow us to analyze the distribution of elements in the tangent direction effectively.
\begin{figure}[t!]
    \centering
    \includegraphics[width=.7\textwidth]{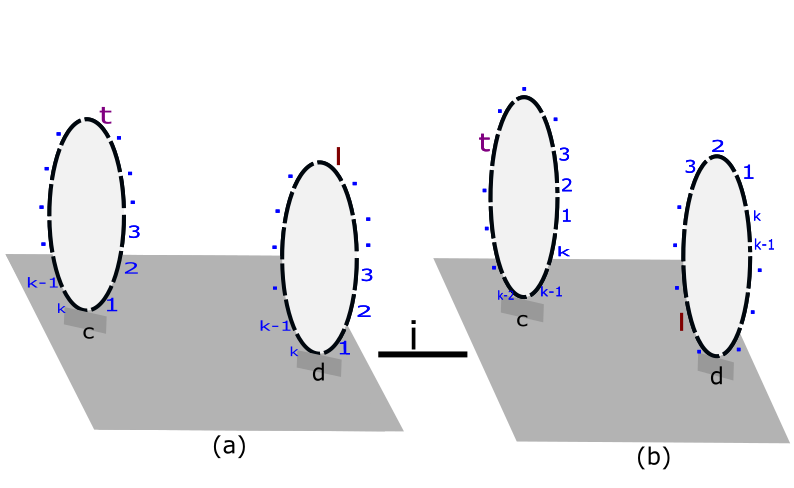}
    \caption{An example of the action $i_n$
  within the 
$\T^2$ transform on the fundamental domains is shown in (a) and (b).
    }
    \label{fig:inner_rotation_B}
\end{figure}
\begin{proposition}\label{prop:2a}
    Let $b_n= 2k_nq_n, a_n=2k_n^{11}q_n, c_n= k_n^5, \varepsilon_n= \frac{1}{2n^5k_n^{11}}.$ There is a smooth measure-preserving diffeomorphsim $i_n: [0,1]^2\rightarrow [0,1]^2$ such that 
    \begin{enumerate}
        \item  $i_n$ coincides with the identity on $[0,1]^2\backslash [\varepsilon_n,1-\varepsilon_n]^2$.
        \item Each square of form $\left[\frac{\mathrm{u}}{a_n},\frac{\mathrm{u}+1}{a_n}\right]\times \left[\frac{\mathrm{v}}{a_n},\frac{\mathrm{v}+1}{a_n}\right] $ with $\mathrm{u}, \mathrm{v} \in \mathbb{Z}$ is mapped onto itself by $i_n$ and $i_n$ coincides with the identity on an $\frac{\varepsilon_n}{a_n}$- neighbourhood of its boundary. 
        \item On every square $\left[\frac{l}{b_n}+\frac{\mathrm{u}}{a_n}+\frac{s_1+2\varepsilon_n}{a_nc_n},  \frac{l}{b_n}+\frac{\mathrm{u}}{a_n}+\frac{s_1+1-2\varepsilon_n}{a_nc_n}\right]\times \left[\frac{\mathrm{v}}{a_n}+\frac{s_2+2\varepsilon_n}{a_nc_n}, \frac{\mathrm{v}}{a_n}+\frac{s_2+1-2\varepsilon_n}{a_nc_n}\right],$ \\
        where $s_1,s_2 \in \mathbb{Z},$ $1\leq s_1, s_2 \leq c_n-2, \ \mathrm{u} \in \{0,\ldots,\frac{a_n}{b_n}-1\}, \  \mathrm{v}\in \{0,\ldots,a_n-1\},$  $i_n$ is a composition of a translation and a rotation by $\beta_l,$ where $\beta_l\equiv \frac{s\pi}{k_n}$ with $s\equiv l \mod k_n.$
        \item $i_n\left(x+\frac{p}{q_n},y\right) = i_n(x,y)$ for all $(x,y)\in \left[\frac{p}{q_n},\frac{p+1}{q_n}\right)\times [0,1],\ p\in \mathbb{Z}.$ 
        \item $\vertiii{i_n}_{r} \leq c_{n,k_n,r}'\cdot q_n^{r-1}$ where the constant $c_{n,k_n,r}'$ is independent of $q_n.$
    \end{enumerate}
\end{proposition}

\begin{proof}
Similar to \cite[Lemma 2.3]{Ku20}, there exist such a measure preserving diffeomorphism, $\psi_{c,\varepsilon,\beta}:[0,1]^2 \rightarrow [0,1]^2, c\in \N, \varepsilon \in (0,\frac{1}{5c}]$ and $\beta\in [0,\pi]$, is constructed with the aid of Moser's trick which satisfies following properties:
\begin{itemize}
    \item $\psi_{c,\varepsilon,\beta}$ coincides with the identity on $[0,1]^2\backslash [\varepsilon, 1-\varepsilon]^2.$
    \item On every square $\left[\frac{\mathrm{v}+2\varepsilon}{c}, \frac{\mathrm{v}+1- 2\varepsilon}{c}\right] \times \left[\frac{\mathrm{k}+2\varepsilon}{c}, \frac{\mathrm{k}+1- 2\varepsilon}{c}\right]$ with $1\leq \mathrm{v},\mathrm{k} \leq c-2$ the map $\psi_{c,\varepsilon,\beta}$ is equal to a composition of a translation and a rotation by arc $\beta$ around the centre.
\end{itemize}
Define the map $\psi_{a,c,\varepsilon, \beta}:\left[0,\frac{1}{a}\right]^2 \rightarrow \left[0,\frac{1}{a}\right]^2 $ by 
$\psi_{a,c,\varepsilon,\beta}= D_a^{-1}\circ \psi_{c,\varepsilon,\beta} \circ D_a,$ where the dilation map $D_a:\left[0,\frac{1}{a}\right]^2\rightarrow [0,1]^2$ defined by $D_a(\theta,r)= (a\cdot\theta, a\cdot r)$ for $a\in \mathbb{N}.$ Since it coincide with the identity in a neighbourhood of the boundary, we can extend it to a smooth diffeomorphism on $[0,1]^2$ equivariantly by the description 
$$\psi_{a,c,\varepsilon,\beta}\left(\theta + \frac{a_1}{a}, r+ \frac{a_2}{a}\right)= \left(\frac{a_1}{a},\frac{a_2}{a}\right)+ \psi_{a,c,\varepsilon,\beta}(\theta, r),$$ 
for $a_1,a_2 \in \mathbb{N}.$
In our concrete construction, we define $i_n$ by 
\begin{align}
 i_n = \psi_{a_n,c_n,\varepsilon_n,\beta_l},\label{eqn:3.2.3a}
\end{align}
where $a_n= 2k_n^{11}q_n, c_n = k_n^5, \varepsilon_n = \frac{1}{2n^5k_n^{10}}$ and the value of $\beta_l$ varies with the function domain, specifically within the context of the given domain: On domain $\left[\frac{l}{b_n}+\frac{\mathrm{u}}{a_n}, \frac{l}{b_n}+\frac{\mathrm{u}+1}{a_n}\right] \times [0,1],$ for all  $\mathrm{u} \in \{0,\ldots, \frac{a_n}{b_n} -1\},$
      choose the value of $\beta_l\equiv \frac{s \cdot \pi}{k_n}$ where $s\equiv l\mod k_n.$ 
Moreover, the map $\psi_{c_n,\varepsilon_n,\beta_l}$ defines the rotation by the arcs and it depends on the parameter $k_n$ but remains independent of $q_n$. Therefore, the norm of the map $i_n$ defined by $i_n = D_{a_n}^{-1} \circ \psi_{c_n,\varepsilon_n,\beta_l} \circ D_{a_n}$ can be estimated as
$$\vertiii{i_n}_r \leq a^{r-1}_n\cdot \vertiii{\psi_{c_n,\varepsilon_n,\beta_l}}_{r} \leq c_{n,k_n,r}' q_n^{r-1},$$
where $c_{n,k_n,r}'$ is a constant depending upon $n, k_n,$ and $r$, but independent of $q_n.$
\end{proof}
\remark Compared to the analogous construction in~\cite[Proposition 2.2]{Ku20}, our choice of the distinct rotation angles, denoted as $\beta_l$, has a different dependence on the specific domain. For illustration of the map $i_n$, we refer to Figure~\ref{fig:inner_rotation_B}.

\remark \label{re:5a} The action of the map $i_n$ on the elements of the partial partitions $\zeta_n$ and $\eta_n$ can be described as the composition of a rotation by an angle $\frac{u_1\cdot \pi}{k_n}$ and a translation, such that:
\begin{align}
      &i_n(\check{I}^{u_0,u_1,u_2}_{v_0}) = \check{I}^{u_0, u_1, u_2}_{v_0},\\
      &i_n(I^{u_0,u_1,u_2,u_3,u_4}_{v_0,v_1,v_2}) \subseteq \check{I}^{u_0, u_1, u_2}_{v_0},
\end{align} 
where $I^{u_0,u_1,u_2,u_3,u_4}_{v_0,v_1,v_2}\in \zeta_n$ and $\check{I}^{u_0,u_1,u_2}_{v_0}\in \eta_n$.

\subsubsection{The conjugation map \texorpdfstring{$\phi_n$}{Lg}}
We define the final conjugation map $\phi_n:[0,1]^2\rightarrow [0,1]^2$ for $u_0=0,1,\ldots, q_n-1$ as
\begin{align}\label{eqn:3.1.3a}
    \phi_n =\begin{cases} 
      i_n \circ \Tilde{\phi}_{n} & \text{on } \  \left[\frac{2u_0}{2q_n}, \frac{2u_0+1}{2q_n}\right]\times [0,1]\\
     \text{Id}  & \text{on} \  \left[\frac{2u_0+1}{2q_n}, \frac{2u_0+2}{2q_n}\right]\times [0,1],
     \end{cases}
\end{align}
where the map $\Tilde{\phi}_n$ and $i_n$ is defined by (\ref{eqn:3.1.1b}) and (\ref{eqn:3.2.3a}) respectively.
\remark Note that the map $\phi_n$ is a smooth area-preserving diffeomorphism on the 2-torus, given that both the maps $\Tilde{\phi}_n$ and $i_n$ act as identities in a neighborhood of the boundary of $[0,1]^2$. Consequently, $\phi_n$ is a well-defined Hamiltonian diffeomorphisms on the torus by Lemma \ref{lem: 2.2.3a}. For illustration of the map $\phi_n$, we refer to Figure~\ref{fig:inner_rotation_A,B}.

\begin{figure}[t!]
    \centering
    \includegraphics[width=.7\textwidth]{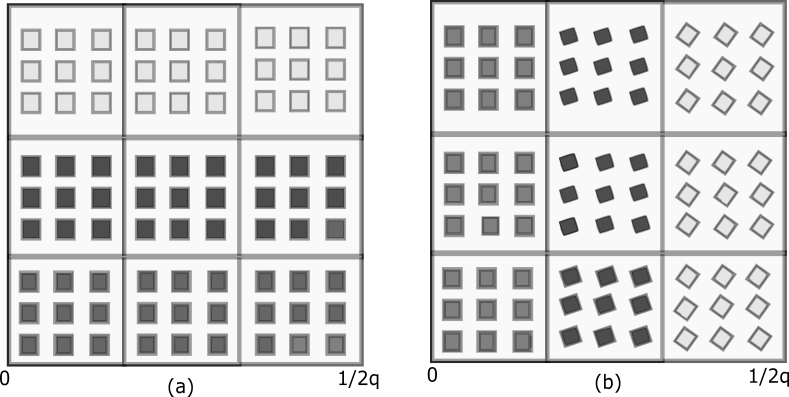}
    \caption{An example of the action of $\phi_n$ transforming from (a) to (b) on the fundamental domain  within $\T^2$.
    }
    \label{fig:inner_rotation_A,B}
\end{figure}

\section{Convergence of sequences \texorpdfstring{$<f_n>_{n \in \N}, <\hat{f_n}>_{n \in \N}$}{Lg}} \label{sec:conv}
 \subsection{Convergence of sequences \texorpdfstring{$<f_n>_{n \in \N}$}{Lg} in the smooth topology}
 There are some standard results on the closeness between the maps constructed as the conjugation of translations on the torus. The following two lemmas follow from Lemmas 3 and 4 in \cite{FSW} with minor to no modifications; hence, we skip the proofs for brevity.

\begin{lemma}[{\cite[Lemma 3]{FSW}}] \label{la:1a}
Let $r\in \mathbb{N}$. For all $\alpha,\beta \in \mathbb{R}$ and all $h\in \textup{Diff}^{\infty}(\mathbb{T}^2,\mu)$, we have the estimate
$$d_r(h S_{\alpha}h^{-1},h S_{\beta}h^{-1}) \leq C_r \max\{\vertiii{h}_{r+1},\vertiii{h^{-1}}_{r+1}\}|\alpha- \beta|,$$
where $C_r$ is a constant that depends only on $r$. 
\end{lemma}

\begin{lemma}[{\cite[Lemma 4]{FSW}}]  \label{la:1b}
For any $\epsilon>0,$ let $r_n$ be a sequence of natural numbers satisfying $\sum\limits_{n=1}^{\infty}\frac{1}{r_n}<\epsilon$. Suppose for any Liouville number $\alpha$ that there exists a sequence of rationals $\{\alpha_n\}$ satisfying
\begin{equation}\label{eqn:5.5}
    |\alpha-\alpha_n|<\frac{1}{2^{n+1}r_n\hat{C}_{r_n}q_n \vertiii{H_n}_{r_{n+1}}^{r_{n+1}}},
\end{equation}
where $\hat{C}_{r_n}$ is the same constant as in Lemma \ref{la:1a}. Then
the sequence of diffeomorphisms  $f_{n} = H_{n} \circ R_{\alpha_{n+1}}\circ H_{n}^{-1}$ converges to  $f\in \textup{Diff}^{\infty}(\mathbb{T}^2, \mu)$ in the $C^{\infty}$ topology. Moreover, for any $m\leq q_{n+1},$ we have 
\begin{equation}\label{eq:4a}
d_0(f^m,f_{n}^m) \leq \frac{1}{2^{n+1}}, \  \text{and} \  
\sup_{x\in \mathbb{P} \mathrm{T}M}d(({f}_n,\mathrm{d}{f}_n)^m(x), ({f},\mathrm{d}{f})^m(x))\leq \frac{1}{2^{n+1}}.
\end{equation} 
\end{lemma}

\begin{lemma} \label{lem:4.4.1a}
For any ${r_n}\in \mathbb{N},$ the conjugation diffeomorphisms constructed in Sections~\ref{sec:3.3.1} and~\ref{sec:phi} satisfy the following statements:
\begin{enumerate}
    \item $\vertiii{h_n}_{r_n}\leq c_{n,k_n,r_n}\cdot q_n^{2{r^3_n}},$ where $c_{n,k_n,r_n}$ is constant independent of $q_n.$ 
    \item $\vertiii{H_n}_{r_n} \leq \hat{c}_{n,k_n,r_n}\cdot q_n^{2{r^4_n}}$, where $\hat{c}_{n,k_n,r_n}$ is constant independent of $q_n.$
    \item For $\alpha$ Liouville, there exists a sequence of rational $\{\alpha_n\}_{n\in \N}$ satisfying  {\eqref{eqn:5.5}} and \ref{item:P1}-\ref{item:P2}. 
\end{enumerate}
\end{lemma}

\begin{proof}
Using Lemma \ref{le:2a} and the norm estimates of conjugation maps ${\phi_n}$ and $g_n,$ defined in Proposition \ref{prop:1a}, \ref{prop:2a} and \ref{lemma:3a}, we have estimate:
\begin{align}
    \vertiii{h_n}_{r_n} &\leq \vertiii{g_n}_{r_n}^{r_n}\cdot \vertiii{\phi_n}_{r_n}^{r_n}  \nonumber\\
    &\leq \tilde{c}_{n,k_n,r_n}\cdot q_n^{{r^2_n}}\cdot \vertiii{i_n}_{r_n}^{r_n} \cdot \vertiii{\tilde{\phi}_n}_{r_n}^{r_n} \leq c_{n,k_n,r_n}\cdot q_n^{2{r^3_n}}.\nonumber
\end{align}
Similarly, $\vertiii{H_n}_{r_n}= \vertiii{H_{n-1}\circ h_n}_{r_n} \leq \vertiii{H_{n-1}}_{r_n}^{r_n} \cdot \vertiii{h_n}_{r_n}^{r_n} $. Since all order derivatives of $H_{n-1}$ are independent of $q_n$, we can conclude that
$\vertiii{H_{n}}_{r_n} \leq  \hat{c}_{n,k_n,r_n}\cdot q_n^{2{r^4_n}}.$\\
For $\alpha$ being a Liouville, we can choose a sequence of rationals $\alpha_n=\frac{p_n}{q_n}$ (with $p_n, q_n$ relatively prime) that satisfy the following property, along with $\ref{item:P1}-\ref{item:P2}$: 
$$|\alpha-\alpha_n| \leq \frac{1}{2^{n+1} \cdot {r_n} \cdot \hat{c}_{n,k_n,r_n} \cdot q_n^{2(r_n+1)^5 +1}} \nonumber\\
\leq \frac{1}{2^{n+1}\cdot {r_n} \cdot \hat{C}_{r_n}\cdot q_n \cdot \vertiii{H_n}_{{r_n}+1}^{{r_n}+1}},$$
where $\hat{C}_{r_n}$ is constant depends only on $r_n.$
\end{proof}

\remark Finally, we have proven the estimate on the norms of the conjugation map $H_n$, as shown in \cite{FS}. Additionally, the existence of rationals satisfying (\ref{eqn:5.5}) guarantees the convergence of the sequence $f_n$ to $f\in \textup{Diff}^{\infty}(\T^2,\mu)$, as stated in Lemma {\ref{la:1b}}. 

\subsection{Convergence of sequence \texorpdfstring{$<\hat{f}_n>_{n \in \N}$}{Lg} in the analytic topology}

\begin{lemma}\label{lem:4:4.2b}
 There exists a sequence of rational numbers $\alpha_{n+1}$ such that the sequence of diffeomorphisms $\hat{f}_n \in \textup{Diff}_{\infty}^{\omega}(\T^2,\mu)$, defined by (\ref{eqn:4:4.1a}), converges to a diffeomorphism $\hat{f} \in \textup{Diff}_{\infty}^{\omega}(\T^2,\mu)$. Moreover, it satisfies the following condition for any natural number $m \leq q_{n+1}$,
\begin{enumerate}
    \item   $d_{0}((\hat{f}_n)^m,(\hat{f})^m)< \frac{1}{2^n},$
    \item $\sup_{x\in \mathbb{P} \mathrm{T}M}d((\hat{f}_n,\mathrm{d}\hat{f}_n)^m(x), (\hat{f},\mathrm{d}\hat{f})^m(x))< \frac{1}{2^n} < \frac{1}{n^2}.$
\end{enumerate}
\end{lemma}

\begin{proof}
Let us fix the sequence $\rho_n = n \in \N$ used for the complex domain $\Omega_{\rho_n}$ as defined by (\ref{eqn:2:2.3a}). Note that the sequence of diffeomorphisms $\hat{f}_n$, defined by (\ref{eqn:4:4.1a}), admits an entire complexification. Thus, each $n$th term of the sequence, $\hat{f}_n\in \textup{Diff}_{\rho_n}^{\omega}(\T^2,\mu)$, where $\rho_n = n$. With the relations $\alpha_{n+1}= \alpha_n + \frac{1}{k_nl_nq_n^2}$ and $\hat{h}_n\circ R_{\alpha_{n}} = R_{\alpha_{n}}\circ \hat{h}_n$, we can obtain for $\rho_n=n :$ 
\begin{align}
    d_{n}^{\omega}(\hat{f}_{n},\hat{f}_{n-1}) &= d_{n}^{\omega}(\hat{H_n}\circ R_{\alpha_{n+1}}\circ \hat{H}_n^{-1}, \hat{H}_{n-1}{\hat{h}_{n}}\circ{\hat{h}_{n}^{-1}} \circ R_{\alpha_{n}}\circ \hat{H}_{n-1}^{-1}) \nonumber\\
    &= d_{n}^{\omega}(\hat{H_n}\circ R_{\alpha_{n+1}}\circ \hat{H}_n^{-1}, \hat{H}_{n-1}{\hat{h}_{n}} \circ R_{\alpha_{n}}\circ {\hat{h}_{n}^{-1}}\circ \hat{H}_{n-1}^{-1})\nonumber\\
    &\leq |\mathrm{d}\hat{H}_n|_n \cdot |\alpha_{n+1} - \alpha_n|
     \leq |\mathrm{d}\hat{H}_n|_n\cdot \frac{1}{k_nl_nq_n^2} \leq \frac{1}{2^n}. \label{eqn:4:4.2a}
     \end{align}
    Since the construction of $\hat{h}_n = \hat{g}_n \circ \hat{\phi}_n$ does not depend on the parameter $l_n$, the last step can be achieved by choosing sufficiently large $l_n \in \mathbb{N}$ and $k_n \in \mathbb{N}$, while satisfying the additional conditions $\ref{item:P1}- \ref{item:P2}$ stated in Subsection \ref{sec:3.3.1}. These choices ensure that the following estimates hold: 
$$l_n > n^2\cdot k_n^{10} \cdot q_n^2 \cdot\|\mathrm{d}H_{n-1}\|_0, \quad \text{and} \quad \log(q_{n+1}) = \log(k_n\cdot l_n \cdot q_n^2)> \|\mathrm{d}\hat{H}_n\|,$$
     and for any natural number $m \leq q_{n+1},$  the 
 following inequality holds:
     \begin{align} \label{eqn:closeness_f_n}
    \sup_{x\in \mathbb{P} \mathrm{T}M}d((\hat{f}_n,\mathrm{d}\hat{f}_n)^m(x), (\hat{f}_{n-1},\mathrm{d}\hat{f}_{n-1})^m(x))< \frac{1}{2\cdot k_n} < \frac{1}{2^n}.     
\end{align}
Based on the estimates (\ref{eqn:4:4.2a}), we can conclude that for any $\varepsilon > 0$, there exists an $N \in \mathbb{N}$ such that $\sum_{i=N}^{\infty} \frac{1}{2^i} < \varepsilon,$ and for any natural numbers $m, n > N$, and $m \leq n$: 
\begin{align}\label{eqn:4:4:2a}
 d_{m}^{\omega}(\hat{f}_{n},\hat{f}_{m}) \leq \sum_{i=m+1}^{n} d_{m}^{\omega}(\hat{f}_{i},\hat{f}_{i-1}) \leq \sum_{i=m+1}^{n} d_{i}^{\omega}(\hat{f}_{i},\hat{f}_{i-1}) \leq \sum_{i=m+1}^{n} \frac{1}{2^i} < \varepsilon.
\end{align}
Analogously, we have $d_{k}^{\omega}(\hat{f}_{n},\hat{f}_{m})< d_{m}^{\omega}(\hat{f}_{n},\hat{f}_{m}) < \varepsilon,$  for any $k<m\leq n.$  
This shows that the sequence $(\hat{f}_n)_{n\in \N}$ is a Cauchy sequence in $\textup{Diff}_{k}^{\omega}(\T^2,\mu) $ for all $k \leq n$. As $\textup{Diff}_{k}^{\omega}(\T^2,\mu)$ is a complete space, there exists a limit diffeomorphism $\hat{f}\in \textup{Diff}_{k}^{\omega}(\T^2,\mu)$ for all $k\leq n$. Furthermore, the estimate (\ref{eqn:4:4:2a}) guarantees $\hat{f} \in \cap_{i=1}^{n}\textup{Diff}_{i}^{\omega}(\T^2,\mu)$.
 Additionally, the sequence $\{\rho_n\}_{n\in \N}$ expands the complex domain of analyticity, $\{\Omega_{\rho_n}\}_{n\in \N},$  and exhausts all $\C^2$. Consequently, we can conclude that the limit diffeomorphism, $\lim_{n\rightarrow \infty}\hat{f}_n = \hat{f},$ admits an entire complexification, i.e.
$\hat{f} \in \textup{Diff}_{\infty}^{\omega}(\T^2,\mu).$

Moreover, we can conclude that for any natural number $m\leq q_{n+1}:$
\begin{align}\label{eqn:5:5.2a}
d_0((\hat{f}_n)^m, (\hat{f})^m) \leq \sum_{i=n+1}^{\infty} d_0((\hat{f}_i)^m, (\hat{f}_{i-1})^m) \leq \sum_{i=n+1}^{\infty} \frac{1}{2^i} < \frac{1}{2^n},
\end{align}
and the sequence $(\hat{f}_n, \mathrm{d}\hat{f}_n)$ defined on $\mathbb{P} \mathrm{T}M,$ using equation (\ref{eqn:closeness_f_n}),  also satisfies :
\begin{align}\label{eqn:5:5.2b}
   \sup_{x\in \mathbb{P} \mathrm{T}M}d((\hat{f}_n,\mathrm{d}\hat{f}_n)^m(x), (\hat{f},\mathrm{d}\hat{f})^m(x)) &\leq \sum_{i=n+1}^{\infty} \sup_{x\in \mathbb{P} \mathrm{T}M}d((\hat{f}_i,\mathrm{d}\hat{f}_i)^m(x), (\hat{f}_{i-1},\mathrm{d}\hat{f}_{i-1})^m(x)) \nonumber\\
   & \leq  \sum_{i=n+1}^{\infty} \frac{1}{2\cdot k_i} <  \sum_{i=n+1}^{\infty}\frac{1}{2^i} < \frac{1}{2^n}.  
\end{align}
Altogether, this shows that there exists $\hat{f}\in \textup{Diff}_{\infty}^{\omega}(\T^2,\mu)$ such that $\lim_{n\rightarrow \infty} \hat{f}_n = \hat{f}$, and it satisfies both the estimates of the lemma given by (\ref{eqn:5:5.2a}) and (\ref{eqn:5:5.2b}).
\end{proof}

 \section{Existence of an invariant measurable Riemannian metric}\label{sec:metric}
Next, we show that our conjugation maps $h_n$ and $\hat{h}_n$ fulfill the requirements of Proposition~\ref{prop:metric}.

\begin{lemma} \label{lem:deviation}
    Let the partial partition $\zeta_n$ be defined as in section~\ref{subsubsec:zeta}. The maps $h_n$ and $\hat{h}_n$ given by (\ref{eq:3:3b}) and (\ref{eqn:4:4.1b}), respectively, satisfy the following for all $I_n \in {\zeta}_n$:
    \begin{enumerate} 
        \item $\text{dev}_{I_n}(h_n) =0$;
        \item $\text{dev}_{I_n}(\hat{h}_n) \leq \frac{2\mathfrak{d}_n}{\|DH_{n-1}\|_{0}^2}$ with $\mathfrak{d}_n$ as in equation (\ref{eqn:4:4.2}).
        \end{enumerate}
\end{lemma}
\begin{proof}
For the first statement, it is sufficient to show that $h_n$ acts as an isometry with respect to $\omega_0$ for all $I^{u_0,u_1,u_2,u_3,u_4}_{v_0,v_1,v_2} \in {\zeta}_n$.
In our explicit construction, the map $\phi_n$ is defined by $(\ref{eqn:3.1.3a})$. If $u_0$ is odd, it acts as the identity and, hence, as an isometry. If $u_0$ is even, $\phi_n$ acts as $i_n\circ \tilde{\phi}_n$ on $I^{u_0,u_1,u_2,u_3,u_4}_{v_0,v_1,v_2}$. In Proposition \ref{prop:1a}, equation (\ref{eqn:3.1.1b}), and Remark \ref{re:4a}, the mapping behavior of $\tilde{\phi}_n$ is explicitly discussed, and it acts as an isometry on $I^{u_0,u_1,u_2,u_3,u_4}_{v_0,v_1,v_2}$ as follows:
$$\tilde{\phi}_n(I^{u_0,u_1,u_2,u_3,u_4}_{v_0,v_1,v_2}) = {I}^{u_0,u_1,k_n^5-v_0-1,u_3,u_4}_{u_2,v_1,v_2}.$$
In Proposition \ref{prop:2a}, the second property defines the map $i_n$ as a composition of a translation and some rotation on such sets ${I}^{u_0,u_1,k_n^5-v_0-1,u_3,u_4}_{u_2,v_1,v_2}$. In particular, $i_n({I}^{u_0,u_1,k_n^5-v_0-1,u_3,u_4}_{u_2,v_1,v_2}) \subset \check{I}^{u_0,u_1,k_n^5-v_0-1}_{u_2}$. The map $g_n$, as defined in Proposition \ref{lemma:3a}, acts as a translation on the image set $\phi_n(I^{u_0,u_1,u_2,u_3,u_4}_{v_0,v_1,v_2}) \subset \check{I}^{u_0,u_1,k_n^5-v_0-1}_{u_2}$. Altogether, we conclude the first statement.

Note that with our choice of parameters and the fact that $d_1(\phi_n,\hat{\phi}_n)< \epsilon_n$ and $d_1(g_n,\hat{g}_n)< \epsilon_n$ in section \ref{sec:4:4.a}, we obtain 
$$d_1(h_n,\hat{h}_n)\leq d_1(g_n\circ \phi_n, \hat{g}_n\circ \hat{\phi}_n) \leq (1+\vertiii{g_n}_2)\cdot \epsilon_n \leq \frac{\mathfrak{d}_n}{\|DH_{n-1}\|_{0}^2}.$$ 
Additionally, for any unit vector $v = (v_1, v_2)$, we obtain :
$$\left\lvert \| \mathrm{d}\hat{h}_n(v)\| - \|\mathrm{d}{h}_n(v)\| \right \rvert \leq \| \mathrm{d}\hat{h}_n(v) - \mathrm{d}{h}_n(v)\| \leq d_1(h_n,\hat{h}_n) \leq \frac{\mathfrak{d}_n}{\|DH_{n-1}\|_{0}^2}.$$ 

Referring to Definition \ref{def:8:8.1}, we can further calculate: 

$$
\left| \log \|\mathrm{d}\hat{h}_n(v)\| - \log \|\mathrm{d}h_n(v)\| \right| = \left| \log \frac{\|\mathrm{d}\hat{h}_n(v)\|}{\|\mathrm{d}h_n(v)\|} \right| \leq \frac{\frac{\mathfrak{d}_n}{\|DH_{n-1}\|_{0}^2}}{\min(\|\mathrm{d}\hat{h}_n(v)\|, \|\mathrm{d}h_n(v)\|)}.
$$
Here, we used the general relation for positive numbers $a,b$ that if $|a - b| \leq \epsilon $ then $|\log(a/b)| \leq \frac{\epsilon}{\min(a, b)}$. Since $\hat{h}_n$ and $h_n$ are close to isometries, we can assume that $\min(\|\mathrm{d}\hat{h}_n(v)\|, \|\mathrm{d}h_n(v)\|) \geq \frac{1}{2}$ for all $v \in \mathrm{T}I_n$ with $\|v\| = 1$. Consequently, the deviation for $\hat{h}_n$ can be estimated by:
\begin{align*}
\text{dev}_{I_n}(\hat{h}_n) & = \max_{v \in \mathrm{T}I_n, \|v\|  = 1} \left| \log \|\mathrm{d}\hat{h}_n(v)\| \right| \\
&\leq \max_{v \in \mathrm{T}I_n, \|v\| = 1} \left( \left| \log \|\mathrm{d}h_n(v)\| \right| + \left| \log \frac{\|\mathrm{d}\hat{h}_n(v)\|}{\|\mathrm{d}h_n(v)\|} \right| \right)\\
&\leq \text{dev}_{I_n}(h_n) + \frac{\mathfrak{d}_n}{\|DH_{n-1}\|_{0}^2} \max_{v \in \mathrm{T}I_n, \|v\| = 1} \frac{1}{\min(\|\mathrm{d}\hat{h}_n(v)\|, \|\mathrm{d}h_n(v)\|)} \\
 &\leq \text{dev}_{I_n}(h_n) + \frac{2\mathfrak{d}_n}{\|DH_{n-1}\|_{0}^2} \leq \frac{2\mathfrak{d}_n}{\|DH_{n-1}\|_{0}^2}.\qedhere
\end{align*}
\end{proof}
\remark\label{rem:8:3:1} Note that $h_n$ acts as an isometry on the sets $I^{u_0,u_1,u_2,u_3,u_4}_{v_0,v_1,v_2},$ and the union of these elements defines $\tilde{\eta}_n$ as described in subsection \ref{subsubsec:Tildeeta}.  This construction verifies the property \ref{item:D1}.
\remark\label{rem:5.5.4} By Lemma~\ref{lem:deviation} we can choose $\mathfrak{d}_n=0$ for all $n\in \N$ when applying Proposition~\ref{prop:metric} for our sequence $({f}_n)_{n\in \N}$ of $C^{\infty}$ diffeomorphisms defined as in \eqref{eq:3a}. Thus, the limit diffeomorphism $f=\lim_{n\rightarrow \infty} {f}_n$ admits an invariant measurable Riemannian metric.  
\remark \label{rem:5.5.5}For any function $\kappa:\textup{Diff}^{\infty}(M,\mu)\to \left(0,1\right)$ we can choose the sequence ${\mathfrak{d}_n}$ in equation (\ref{eqn:4:4.2}) to decay sufficiently fast so that the condition $\sum_{k\geq n}\delta_k < \kappa(H_{n-1})$ holds for every $n \in \N$. This is ensured by choosing a sufficiently large $l_n$ dependent on the function $\kappa(H_{n-1})$. Along with Lemma ~\ref{lem:deviation} and Proposition ~\ref{prop:metric}, we can then conclude that the limit diffeomorphism $\hat{f}$ obtained from the sequence $\hat{f}_n$ constructed as in \eqref{eqn:4:4.1a} admits an invariant measurable Riemannian metric.

\section{Proof of weak mixing of the derivative extension}\label{sec:app_crit}

\subsection{Choice of partial partition of \texorpdfstring{$\mathbb{P} \mathrm{T}M$}{Lg}} \label{sec:7.7.1a}
 Consider the 3-dimensional elements in $\mathbb{P} \mathrm{T}M$ for $\tilde{I}^{u_0}_{v_0} = \cup_{u_1 =0}^{k_n-1}\tilde{I}^{u_0,u_1}_{v_0}\in \tilde{\eta}_n$ by (\ref{eqn:7.1a}), and $j\in \{0, \ldots, k_n-1\}$ as
    $$\hat{I}_{u_0,v_0,j}= \tilde{I}^{u_0}_{v_0} \times T_j =  \bigcup_{u_1=0}^{k_n-1}(\tilde{I}^{u_0,u_1}_{v_0} \times T_j), \ \ \text{where} \ \ 
T_j= \left[\frac{j}{k_n},\frac{j+1}{k_n}\right].$$ Denote the following collection of elements to form a partial partition of $\mathbb{P} \mathrm{T}M$ by 
\begin{align}\label{eqn:partial partition}
    \hat{\eta}_n := \left\{\hat{I}_{u_0,v_0,j} \ : \ u_0\in \{0,\ldots,2q_n-1\}, v_0 \in \{1,\ldots,k_n^5-2\}, j\in \{0,\ldots, k_n-1\}    \right\}.
\end{align}
\remark\label{rem:7.7.1} Note that $\hat{\eta}_n$ is a partial partition of $\mathbb{P} \mathrm{T}M$. For every $n\in \N$, $\hat{\eta}_n$ consists of disjoint sets, covers a set of measure at least $\left(1-\frac{25}{k_n^5}\right)^2\geq 1-\frac{50}{k_n^5}$. For $I\in \hat{\eta}_n,$ it is a finite union of elements of the form $I^{u_0,u_1,u_2,u_3,u_4}_{v_0,v_1,v_2} \times T_j,$ and every element has $\text{diam}(I^{u_0,u_1,u_2,u_3,u_4}_{v_0,v_1,v_2} \times T_j) \leq \frac{\sqrt{2}}{k_n} \rightarrow 0$ as $n\rightarrow\infty$. Thus, the decomposition $(\hat{\eta}_n)_{n\in\N}$ converges to the decomposition into points.
\subsection{Choice of mixing sequence \texorpdfstring{$m_n$}{Lg}}\label{sec:5.1a}
We consider  $m_n \coloneqq \min\left\{m\leq q_{n+1} \ | \ \inf_{k\in \Z} \left \lvert m\frac{q_np_{n+1}}{q_{n+1}}-\frac{1}{2}+k \right\rvert \leq \frac{q_{n}}{q_{n+1}}\right\}$
and $$\mathfrak{a}_n \coloneqq \left(m_n\alpha_{n+1}- \frac{1}{2q_n}\mod \frac{1}{q_n}\right),$$
as defined in \cite{FS} for the torus case. Under the growth assumption $q_{n+1} > 2k_n^{12}q_n^{2}$,  we obtain $|\mathfrak{a}_n| \leq \frac{1}{q_{n+1}} <  \frac{1}{2k_n^{12}q_n^{2}}.$
\remark \label{re:7.2a} Since the action of $R_{\alpha_{n+1}}^{m_{n+1}}$ acts as a translation in the $\theta$ direction by the factor $m_{n+1}\alpha_{n+1}$, its action is trivial in the tangent direction and by the condition $|\mathfrak{a}_n|  < \frac{1}{2k_n^{12}q_n^2}$ we have
\begin{align*}
   {\mu}( R_{\alpha_{n+1}}^{m_n}(\check{I}^{u_0,u_1,u_2}_{v_0}) \triangle \check{I}^{u_0+1,u_1,u_2}_{v_0}) &< |\mathfrak{a}_n| < \frac{1}{2k_n^{12}q_n^2}, \\
    \mu(R_{\alpha_{n+1}}^{m_n} (\tilde{I}^{u_0,u_1}_{v_0} )\triangle \tilde{I}^{u_0+1,u_1}_{v_0}) &< |\mathfrak{a}_n| < \frac{1}{2k_n^{12}q_n^2}.
\end{align*}
\subsection{Action of \texorpdfstring{$(\Phi_n,\mathrm{d}\Phi_n)$}{Lg}}
Note that $\tilde{I}^{u_0,u_1}_{v_0}\subseteq \bigcup_{u_2=1}^{k_n^5-2}\check{I}^{u_0,u_1,u_2}_{v_0}$ with $\check{I}^{u_0,u_1,u_2}_{v_0} \in {\eta}_n$ belongs to the ``good domain" of the conjugation map $\phi_n$ by (\ref{eqn:2.2c}) and (\ref{eqn:7.1a}). Therefore, we can describe the mapping behaviour of the projectivized derivative extension $(\phi_n,\mathrm{d}\phi_n)$ on the elements $\hat{I}_{u_0,v_0,j}\in \hat{\eta}_n$ explicitly.
\begin{lemma}\label{lem:7.3a} For any $\tilde{I}^{u_0,u_1}_{v_0} \times T_j\subseteq \mathbb{P} \mathrm{T}M,$ where $u_0$ is even and $\tilde{I}^{u_0,u_1}_{v_0}$ is defined by (\ref{eqn:7.1a}), the following holds: 
    \begin{enumerate}
   \item $\phi_n(\tilde{I}^{u_0,u_1}_{v_0}) \subseteq \bigcup_{u_2=1}^{k_n^5-2} \check{I}^{u_0,u_1,k_n^5-v_0-1}_{u_2},$ and $\phi_n^{-1} (\tilde{I}^{u_0,u_1}_{v_0})\subseteq \bigcup_{u_2=1}^{k_n^5-2} \check{I}^{u_0,u_1,v_0}_{k_n^5-u_2-1}.$ 
    \item $(\phi_n,\mathrm{d}\phi_n) (\tilde{I}^{u_0,u_1}_{v_0}\times T_j)\subseteq \bigcup_{u_2=1}^{k_n^5-2} \check{I}^{u_0,u_1,k_n^5-v_0-1}_{u_2}\times T_k,$ where $k \equiv u_1+ j \mod k_n.$ Moreover,  
  \begin{equation}\label{eq:Forward}
 {\mu}\bigg(\pi_M\bigg((\phi_n,\mathrm{d}\phi_n)\big(\tilde{I}^{u_0,u_1}_{v_0}\times T_j\big) \triangle
    \bigg(\bigcup_{u_2=1}^{k_n^5-2}\check{I}^{u_0,u_1,k_n^5-v_0-1}_{u_2} \times T_k \bigg)\bigg)\bigg) \leq \frac{8}{k_n^{11}q_n}.
    \end{equation} 
    \item  $(\phi_n,\mathrm{d}\phi_n)^{-1} (\tilde{I}^{u_0,u_1}_{v_0}\times T_j)\subseteq \bigcup_{u_2=1}^{k_n^5-2} \check{I}^{u_0,u_1,v_0}_{k_n^5-u_2-1}\times T_{k'},$ where $k' \equiv -u_1+ j \mod k_n.$ Moreover, 
    \begin{equation}\label{eq:Inverse}
{\mu}\bigg(\pi_M\bigg((\phi_n,\mathrm{d}\phi_n)^{-1}\big(\tilde{I}^{u_0,u_1}_{v_0}\times T_j\big) \triangle
    \bigg(\bigcup_{u_2=1}^{k_n^5-2}\check{I}^{u_0,u_1,v_0}_{k_n^5-u_2-1} \times T_{k'} \bigg)\bigg)\bigg) \leq \frac{8}{k_n^{11}q_n}.
    \end{equation} 
  \end{enumerate}    
\end{lemma}
In particular, an element $\tilde{I}^{u_0,u_1}_{v_0} $  is mapped to a region $\bigcup_{u_2=1}^{k_n^5-2}\check{I}^{u_0,u_1,k_n^5-v_0-1}_{u_2} \subset \T^2$, which has almost full vertical length, and the $j$th tangent interval $T_j$ is translated to the $k$th tangent interval $T_k$, where $k\equiv u_1+ j \mod k_n$, under the map $(\phi_n, \mathrm{d}\phi_n)$. 

\begin{proof} 
In our explicit construction $\phi_n$ defined by $(\ref{eqn:3.1.3a})$, if $u_0$ is even, $\phi_n$ acts as $i_n\circ \tilde{\phi}_n$ on $\tilde{I}^{u_0,u_1}_{v_0}.$ In Proposition \ref{prop:1a} and Remark \ref{re:4a}, the mapping behaviour of $\tilde{\phi}_n$ is discussed explicitly. We also note that $\mathrm{d}_p\tilde{\phi}_n = \text{id}$ for any base point $p\in I^{u_0,u_1,u_2,u_3,u_4}_{v_0,v_1,v_2}$. Therefore we conclude
$$(\tilde{\phi}_n, \mathrm{d}\tilde{\phi}_n)(\tilde{I}^{u_0,u_1}_{v_0}\times T_j)= \bigcup_{u_2=1}^{k_n^5-2}\bigcup_{v_1=2q_n}^{2k_n^{6}q_n-2q_n-1}\bigcup_{u_3, u_4, v_2=1}^{k_n^5-2}{I}^{u_0,u_1,k_n^5-v_0-1,u_3,u_4}_{u_2,v_1,v_2} \times T_j.$$

In Proposition \ref{prop:2a}, the third property defines the map $i_n$ as a composition of a translation and a rotation by $\frac{u_1\cdot \pi}{k_n}$ on ${I}^{u_0,u_1,k_n^5-v_0-1,u_3,u_4}_{ u_2,v_1,v_2}$. Additionally, the second statement in Proposition~\ref{prop:2a} establishes that the image of ${I}^{u_0,u_1,k_n^5-v_0-1,u_3,u_4}_{u_2,v_1,v_2},$ for all $u_3, v_2, v_1\in \Z,$ under $i_n$ is contained in $\check{I}^{u_0,u_1,k_n^5-v_0-1}_{u_2}.$ Consequently, we obtain the following containment:
$$\big(i_n,\mathrm{d} i_n\big) \bigg(\bigcup_{u_2=1}^{k_n^5-2} \bigcup_{v_1=2q_n}^{2k_n^{6}q_n-2q_n-1}\bigcup_{u_3, u_4, v_2=1}^{k_n^5-2}{I}^{u_0,u_1,k_n^5-v_0-1,u_3,u_4}_{u_2,v_1,v_2} \times T_j\bigg) \subseteq \bigcup_{u_2=1}^{k_n^5-2} \check{I}^{u_0,u_1,k_n^5-v_0-1}_{u_2} \times T_k,$$
where $k\equiv j+ u_1 \mod k_n.$ This implies the first statement and the first part of the second statement of the lemma.

Together with (\ref{eqn:7.1a}), the first statement of the lemma, Remark \ref{rem:assumption_intesection}, and the fact that the map $\phi_n$ is measure-preserving, we get
$${\mu}\bigg(\bigg(\bigcup_{u_2=1}^{k_n^5-2}\check{I}^{u_0,u_1,k_n^5-v_0-1}_{u_2}\bigg)\bigg\backslash \phi_n \big(\tilde{I}^{u_0,u_1}_{v_0}\big)\bigg)\leq {\mu}\bigg(\bigcup_{u_2=1}^{k_n^5-2}\check{I}^{u_0,u_1,k_n^5-v_0-1}_{u_2}\bigg) -  \mu\bigg(\phi_n \big(\tilde{I}^{u_0,u_1}_{v_0}\big)\bigg) \leq \frac{8}{k_n^{11}q_n}. $$
Moreover, using the first part of second statement of the lemma, we can conclude that
\begin{equation*}
    \begin{split}
\mu\bigg(\pi_M\bigg((\phi_n,\mathrm{d}\phi_n)\big(\tilde{I}^{u_0,u_1}_{v_0}\times T_j\big) \triangle  \bigg(\bigcup_{u_2=1}^{k_n^5-2}\check{I}^{u_0,u_1,k_n^5-v_0-1}_{u_2} \times T_k \bigg)\bigg)\bigg) \\
\leq  {\mu}\bigg(\bigg(\bigcup_{u_2=1}^{k_n^5-2}\check{I}^{u_0,u_1,k_n^5-v_0-1}_{u_2}\bigg)\bigg\backslash \phi_n \big(\tilde{I}^{u_0,u_1}_{v_0}\big)\bigg) \leq \frac{8}{k_n^{11}q_n}. 
 \end{split}
\end{equation*}

Analogously, we examine the effect of $\phi_n^{-1}$, which is defined by $\tilde{\phi}_n^{-1} \circ i_n^{-1}$, on $\tilde{I}^{u_0,u_1}_{v_0}$ for an even value of $u_0$. With Proposition \ref{prop:2a}, we conclude 
$$\bigcup_{u_2=1}^{k_n^5-2} \big(i_n,\mathrm{d}i_n\big)^{-1} \bigg( \bigcup_{v_1=2q_n}^{2k_n^{6}q_n-2q_n-1}\bigcup_{u_3, u_4, v_2=1}^{k_n^5-2}I^{u_0,u_1,u_2,u_3,u_4}_{v_0,v_1,v_2} \times T_j\bigg) \subseteq \bigcup_{u_2=1}^{k_n^5-2} \check{I}^{u_0,u_1,u_2}_{v_0} \times T_{k'}, $$
where $k' \equiv - u_1+ j \mod k_n.$
Considering Remark \ref{re:4a}, which discusses the action of $\tilde{\phi}_n$ on the elements $\check{I}^{u_0,u_1,u_2}_{v_0} \in \eta_n$, and in addition to the fact that $\mathrm{d}_p\tilde{\phi}_n^{-1} = \text{id}$ for any base point $p \in \check{I}^{u_0,u_1,u_2}_{v_0},$ we conclude 
$$(\tilde{\phi}_n,\mathrm{d}\tilde{\phi}_n)^{-1}\bigg(\bigcup_{u_2=1}^{k_n^5-2} \check{I}^{u_0,u_1,u_2}_{v_0} \times T_{k'}\bigg) \subseteq \bigcup_{u_2=1}^{k_n^5-2} \check{I}^{u_0,u_1,v_0}_{k_n^5-u_2-1}\times T_{k'}.$$
Altogether, this yields the third statement as
$$
(\tilde{\phi}_n,\mathrm{d}\tilde{\phi}_n)^{-1}\circ \big(i_n,\mathrm{d}i_n\big)^{-1}(\tilde{I}^{u_0,u_1}_{v_0}\times T_j)\subseteq \bigcup_{u_2=1}^{k_n^5-2} \check{I}^{u_0,u_1,v_0}_{k_n^5-u_2-1}\times T_{k'}.
$$
Similar to the proof of estimate \eqref{eq:Forward}, we obtain the estimate \eqref{eq:Inverse}.
\end{proof}

\remark\label{rem:assumption_intersection2} Note that the partial partition $\hat{\eta}_n$ defined in Section~\ref{sec:7.7.1a} satisfies the required property \ref{item:p2} for the partial partition $\hat{\eta}_n$ of $\mathbb{P}\mathrm{T}M$ in Section~\ref{sec:crit_weak_mixing}. Specifically, with our explicit partition of $\hat{\eta}_n$, we have $\tilde{I}_{n,l} = \tilde{I}_{v_0}^{u_0,u_1}$ and $\hat{I}_{n,j} = \bigcup_{u_1=0}^{k_n-1} \tilde{I}_{v_0}^{u_0,u_1} \times T_j,$ for $u_0\in\{0,\ldots, q_n-1\}, v_0\in \{1,\ldots,k_n^5-2\},$ and $u_1,j\in \{0,\ldots,k_n-1\}.$
Since the map $g_n^{-1},$ defined in Proposition~\ref{lemma:3a}, acts as a translation in the horizontal direction on each $\tilde{I}_{v_0}^{u_0,u_1}$ by $\frac{[nq_n^\sigma]v_0}{k_n^5}$, we have $g_n^{-1}(\tilde{I}_{v_0}^{u_0,u_1}) = \tilde{I}_{v_0}^{u_0',u_1'}$ for some $u_0'$ and $u_1'$.\\
Furthermore, applying $\phi_n^{-1}$ to these elements and using the first statement of Lemma~\ref{lem:7.3a}, we obtain $h_n^{-1}(\tilde{I}_{v_0}^{u_0,u_1}) = \phi_n^{-1}\circ g_n^{-1}(\tilde{I}_{v_0}^{u_0,u_1})\subseteq \cup_{u_2=1}^{k_n^5-2}\check{I}_{k_n^5-u_2-2}^{u_0',u_1',v_0}$ with the estimate 
\begin{align}\label{eq:D4_property1}
\mu(h_n^{-1}(\tilde{I}_{v_0}^{u_0,u_1})\triangle [\cup_{u_2=1}^{k_n^5-2}\check{I}_{k_n^5-u_2-2}^{u_0',u_1',v_0}])\leq \frac{16}{k_n^5}\mu(\tilde{I}_{v_0}^{u_0,u_1}).
\end{align}
Along with Remark~\ref{rem:2.7a}, which states that each $\bigcup_{u_2=1}^{k_n^5-2} \check{I}_{k_n^5-u_2-2}^{u_0',u_1',v_0}$ is almost covered by the elements $\tilde{I}_{v_0}^{u_0} \in \tilde{\eta}_n$, let $\bar{\Lambda}_{\tilde{I}_{v_0}^{u_0',u_1'}} = \{ \tilde{I}_{b_0}^{a_0,a_1} \subseteq \bigcup_{u_2=1}^{k_n^5-2} \check{I}_{k_n^5-u_2-2}^{u_0',u_1',v_0} : \tilde{I}_{b_0}^{a_0} \in \tilde{\eta}_n, \tilde{I}_{b_0}^{a_0} = \bigcup_{a_1 =0}^{k_n-1} \tilde{I}_{b_0}^{a_0,a_1} \}$ be a collection of such sets. We have the following estimate:
$$\mu(\cup_{u_2=1}^{k_n^5-2}\check{I}_{k_n^5-u_2-2}^{u_0',u_1',v_0} \triangle [\cup_{\tilde{I}_{b_0}^{a_0,a_1}\in \bar{\Lambda}_{\tilde{I}_{v_0}^{u_0',u_1'}}}\tilde{I}_{b_0}^{a_0,a_1}])\leq \frac{16}{k_n^5} \mu(\tilde{I}_{v_0}^{u_0',u_1'}).$$
Furthermore, by applying the triangle inequality with equation \eqref{eq:D4_property1}, we can derive the following estimation:
$$\mu(h_n^{-1}(\tilde{I}_{v_0}^{u_0,u_1})\triangle [\cup_{\tilde{I}_{b_0}^{a_0,a_1}\in \bar{\Lambda}_{\tilde{I}_{v_0}^{u_0',u_1'}}}\tilde{I}_{b_0}^{a_0,a_1}])\leq \frac{32}{k_n^5}\mu(\tilde{I}_{v_0}^{u_0,u_1}),$$
which verifies the property \ref{item:p2} for the elements of $\tilde{\eta}_n$.

\begin{lemma}\label{lem:7:7.3} Consider the map $\Phi_n = \phi_n\circ R_{\alpha_{n+1}}^{m_n}\circ \phi_n^{-1}$. The map $(\Phi_n,\mathrm{d}\Phi_n)$ is $(\gamma,\delta,\varepsilon_1,\varepsilon_2)$-distributing the elements  $\hat{I}_{u_0,v_0,j}\in \hat{\eta}_n$ in the sense of 
 Definition \ref{def:6.2a} with $\gamma = \frac{1}{2k_n^5q_n}, \delta = \frac{1}{k_n^4}, \varepsilon_1=\varepsilon_2=\frac{1}{n^5}$. 
\end{lemma}

\begin{proof} For any $\hat{I}_{u_0,v_0,j} = \cup_{u_1=0}^{k_n-1} \tilde{I}^{u_0,u_1}_{v_0}\times T_j \in \hat{\eta}_n$, we consider two cases for the values of $u_0$ to understand the action of $(\Phi_n,\mathrm{d}\Phi_n)$. If $u_0$ is even, we apply Lemma~\ref{lem:7.3a}. If $u_0$ is odd, we use the fact that the maps $(\phi_n,\mathrm{d}\phi_n)$ act as the identity inside the ``good domains" and Remark~\ref{re:7.2a}.  \\
\textbf{Case 1:} If $u_0$ is even, we have for each fixed $u_1 \in \{0,\ldots,k_n-1\}$,
\begin{align}
    (\Phi_n,\mathrm{d}\Phi_n)(\tilde{I}^{u_0,u_1}_{v_0}\times T_j) &= (\phi_n,\mathrm{d}\phi_n)\circ (R_{\alpha_{n+1}},\mathrm{d}R_{\alpha_{n+1}})^{m_n}\circ (\phi_n,\mathrm{d}\phi_n)^{-1}\Bigg(\tilde{I}^{u_0,u_1}_{v_0}\times T_j\Bigg)\nonumber\\
    &\subseteq (\phi_n,\mathrm{d}\phi_n)\circ (R_{\alpha_{n+1}},\mathrm{d}R_{\alpha_{n+1}})^{m_n}\Bigg(\bigcup_{u_2=1}^{k_n^5-2}\check{I}^{u_0,u_1,v_0}_{k_n^5-u_2-1} \times T_{k'} \Bigg),  \nonumber 
    \end{align} 
  where $k' \equiv -u_1+ j \mod k_n$. Using the Remark \ref{re:7.2a} and Lemma \ref{lem:7.3a}, and  the fact that $(\phi_n,\mathrm{d}\phi_n)$ acts as the identity in the domain $(\bigcup_{u_2=1}^{k_n^5-2}\check{I}^{u_0+1,u_1,v_0}_{k_n^5-u_2-1} \times T_{k'} ),$  we can conclude $$(\Phi_n,\mathrm{d}\Phi_n)(\tilde{I}^{u_0,u_1}_{v_0}\times T_j)\big)\cap \bigcup_{u_2=1}^{k_n^5-2}\check{I}^{u_0+1,u_1,v_0}_{k_n^5-u_2-1} \times T_{k'} \neq \emptyset, $$   
  and the following estimate
   \begin{equation*}\label{eqn:lem:8.4.1a}
   {\mu}\bigg(\pi_M\bigg(\big((\Phi_n,\mathrm{d}\Phi_n)(\tilde{I}^{u_0,u_1}_{v_0}\times T_j)\big)\triangle \bigcup_{u_2=1}^{k_n^5-2}\check{I}^{u_0+1,u_1,v_0}_{k_n^5-u_2-1} \times T_{k'} \big)\bigg)\bigg) \leq \frac{9}{k_n^{11}q_n}.
   \end{equation*}
   
\noindent\textbf{Case 2:} If $u_0$ is odd, we have for each fixed $u_1 \in \{0,\ldots,k_n-1\}$,
\begin{align}
    (\Phi_n,\mathrm{d}\Phi_n)(\tilde{I}^{u_0,u_1}_{v_0}\times T_j) &=(\phi_n,\mathrm{d}\phi_n)\circ (R_{\alpha_{n+1}},\mathrm{d}R_{\alpha_{n+1}})^{m_n}\circ (\phi_n,\mathrm{d}\phi_n)^{-1}\big(\tilde{I}^{u_0,u_1}_{v_0}\times T_j \big) \nonumber \\
    & = (\phi_n,\mathrm{d}\phi_n)\circ (R_{\alpha_{n+1}},\mathrm{d}R_{\alpha_{n+1}})^{m_n}\big(\tilde{I}^{u_0,u_1}_{v_0}\times T_j\big).   \nonumber 
\end{align}
Similar to Case 1, using Remark \ref{re:7.2a} and Lemma {\ref{lem:7.3a}}, we have $$(\Phi_n,\mathrm{d}\Phi_n)(\tilde{I}^{u_0,u_1}_{v_0}\times T_j)\big) \bigcap \bigcup_{u_2=1}^{k_n^5-2}\check{I}^{u_0+1,u_1,k_n^5-v_0-1}_{u_2} \times T_k \neq \emptyset,$$ and the following estimate $$ \mu\bigg(\pi_M\bigg(\big((\Phi_n,\mathrm{d}\Phi_n)(\tilde{I}^{u_0,u_1}_{v_0}\times T_j)\big)\triangle \big(\bigcup_{u_2=1}^{k_n^5-2}\check{I}^{u_0+1,u_1,k_n^5-v_0-1}_{u_2} \times T_k\big)\bigg)\bigg) \leq \frac{9}{k_n^{11}q_n},$$
where $k \equiv u_1+ j \mod k_n$. 

Note that the $\theta$ widths of the sets $(\bigcup_{u_2=1}^{k_n^5-2}\check{I}^{u_0+1,u_1,k_n^5-v_0-1}_{u_2})$ and $(\bigcup_{u_2=1}^{k_n^5-2}\check{I}^{u_0+1,u_1,v_0}_{k_n^5-u_2-1})$ are at most $\frac{1}{2k_n^6q_n}(1-\frac{1}{k_n^5}).$ Consequently, both cases imply that for all $u_1\in \{0,\ldots,k_n-1\}$ the set $\Phi_n(\tilde{I}^{u_0,u_1}_{v_0})$ is contained within a set of $\theta$ width at most $\frac{1}{2k_n^6q_n} -\frac{1}{2k_n^{11}q_n}+ \frac{9}{k_n^{11}q_n}$. Hence, we can choose $\gamma = \frac{1}{2k_n^5q_n}.$

Furthermore, note that $\mu(\tilde{I}^{u_0,u_1}_{v_0}) \geq \frac{1}{2k_n^6q_n}\left(1 - \frac{16}{k_n^5}\right)$. The transformation $\Phi_n$ is measure preserving, and $\Phi_n(\tilde{I}^{u_0,u_1}_{v_0})$ has $\theta$ width of at most $\frac{1}{2k_n^6q_n}\left(1+\frac{4}{k_n^5}\right)$, i.e.
$\lambda\left(\pi_{\theta}(\Phi_n(\tilde{I}^{u_0,u_1}_{v_0}))\right) \leq \frac{1}{2k_n^6q_n}(1+\frac{4}{k_n^5}).$  This implies that $\lambda\left(\pi_{r}\big(\Phi_n(\tilde{I}^{u_0,u_1}_{v_0})\big)\right) \geq \left(1 - \frac{20}{k_n^5}\right)$, allowing us to choose $\delta = \frac{1}{k_n^4}$.

Denote $J_{u_1} = \pi_{r}\big(\Phi_n(\tilde{I}^{u_0,u_1}_{v_0})\big),$   $A_{\theta}^{u_1}= \pi_{\theta}(\bigcup_{u_2=1}^{k_n^5-2}\check{I}^{u_0+1,u_1,k_n^5-v_0-1}_{u_2}) = \pi_{\theta}(\check{I}^{u_0+1,u_1,k_n^5-v_0-1}_{u_2}),$ and  $J_{u_1}' = \pi_{r}(\bigcup_{u_2=1}^{k_n^5-2}\check{I}^{u_0+1,u_1,k_n^5-v_0-1}_{u_2}).$
Thus we can express $\bigcup_{u_2=1}^{k_n^5-2}\check{I}^{u_0+1,u_1,k_n^5-v_0-1}_{u_2} =  A_{\theta}^{u_1} \times J_{u_1}'$. 

Analogous to Cases 1 and 2, using Remark~\ref{rem:assumption_intersection2} and Lemma~\ref{lem:7.3a}, we obtain the following estimate:
\begin{align}\label{eqn:explicit_distri_phi_n}
\mu\big(\Phi_n(\tilde{I}^{u_0,u_1}_{v_0}) \triangle \bigcup_{u_2=1}^{k_n^5-2}\check{I}^{u_0+1,u_1,k_n^5-v_0-1}_{u_2} \big) \leq \frac{9}{k_n^{11}q_n} \leq \frac{36}{k_n^5}\mu\big(\bigcup_{u_2=1}^{k_n^5-2}\check{I}^{u_0+1,u_1,k_n^5-v_0-1}_{u_2}\big),
\end{align}
for  any $\tilde{J} \subseteq J_{u_1}'\cap J_{u_1}$, we can conclude that 
\begin{align}\label{eq:Distri1}
   \left(1-\frac{36}{k_n^5}\right)\lambda(A_{\theta}^{u_1})\lambda(\tilde{J})\leq  \mu(\Phi_n(\tilde{I}^{u_0,u_1}_{v_0})\cap (\mathbb{T}\times \tilde{J}))\leq \left(1+\frac{36}{k_n^5}\right)\lambda(A_{\theta}^{u_1})\lambda(\tilde{J}).
    \end{align}
Equation  (\ref{eqn:explicit_distri_phi_n}) also yields 
\begin{equation}\label{eq:ApproxDistri}
\begin{split}
|\lambda(A_{\theta}^{u_1})\lambda(J_{u_1}) - \mu(\Phi_n(\tilde{I}^{u_0,u_1}_{v_0})) | = |\mu(\bigcup_{u_2=1}^{k_n^5-2}\check{I}^{u_0+1,u_1,k_n^5-v_0-1}_{u_2})- \mu(\Phi_n(\tilde{I}^{u_0,u_1}_{v_0})) | 
\\ 
\leq \mu\big(\Phi_n(\tilde{I}^{u_0,u_1}_{v_0}) \triangle \bigcup_{u_2=1}^{k_n^5-2}\check{I}^{u_0+1,u_1,k_n^5-v_0-1}_{u_2} \big) \leq \frac{9}{k_n^{11}q_n} \leq \frac{36}{k_n^5}\mu(\tilde{I}^{u_0,u_1}_{v_0}).
\end{split}
\end{equation}
For the last inequality, we used  $\frac{1}{2k_n^6q_n}(1-\frac{16}{k_n^5})< \mu(\tilde{I}_{v_0}^{u_0,u_1})$ that yields $\frac{1}{\mu(\tilde{I}_{v_0}^{u_0,u_1})} \leq {4k_n^6q_n} $.

Using the triangle inequality with the estimates in (\ref{eqn:explicit_distri_phi_n}), (\ref{eq:Distri1}), and (\ref{eq:ApproxDistri}), we obtain 
\begin{align}
    &|\mu(\Phi_n(\tilde{I}^{u_0,u_1}_{v_0})\cap (\mathbb{T}\times \tilde{J}))\lambda(J_{u_1}) - \lambda(\tilde{J})\mu(\tilde{I}^{u_0,u_1}_{v_0})|\nonumber \\
    \leq &  \lambda(J_{u_1}) |\mu(\Phi_n(\tilde{I}^{u_0,u_1}_{v_0})\cap (\mathbb{T}\times \tilde{J})) - \lambda(\tilde{J})\lambda(A_{\theta}^{u_1})| 
    + \lambda(\tilde{J})|\lambda(A_{\theta}^{u_1})\lambda(J_{u_1}) - \mu(\Phi_n(\tilde{I}^{u_0,u_1}_{v_0})) |\nonumber \\
    \leq & \frac{36}{k_n^5} \lambda(\tilde{J})\lambda(A_{\theta}^{u_1})\lambda(J_{u_1}) + \frac{36}{k_n^5}\lambda(\tilde{J})\mu(\tilde{I}^{u_0,u_1}_{v_0}) < \frac{1}{n^5}\lambda(\tilde{J})\mu(\tilde{I}^{u_0,u_1}_{v_0}).
\end{align}
Thus, we can choose $\varepsilon_1= \frac{1}{n^5}.$ 

Next, we show that $(\Phi_n,\mathrm{d}\Phi_n)$ distributes $\hat{I}_{u_0,v_0,j} = \cup_{u_1=0}^{k_n-1}\tilde{I}^{u_0,u_1}_{v_0}\times T_j \in \hat{\eta}_n$ in the tangent direction in the sense of Definition \ref{def:6.2a}:  For any $j,k \in \{0,\ldots,k_n-1\},$ there exist a unique $u_1\in \{0,\ldots,k_n-1\}$ satisfying $k\equiv u_1+j \mod k_n,$ and $\tilde{J}\subseteq J_{u_1} = \pi_{r}(\Phi_n(\tilde{I}^{u_0,u_1}_{v_0}))$ such that $$(\Phi_n,\mathrm{d}\Phi_n)(\tilde{I}^{u_0,u_1}_{v_0}\times T_j) \cap (\mathbb{T}\times \tilde{J}\times T_k)\neq \emptyset,$$ and building on \eqref{eq:Distri1} we get the following  estimate 
   \begin{align*}
      \big|\mu\left(\pi_M\left(\tilde{I}^{u_0,u_1}_{v_0}\times T_j \cap (\Phi_n,\mathrm{d}\Phi_n)^{-1}(\mathbb{T}\times \tilde{J}\times T_k)\right)\right)\lambda(J_{u_1}) - \lambda(\tilde{J})\mu(\tilde{I}^{u_0,u_1}_{v_0})\big| \leq \frac{1}{n^5}\lambda(\tilde{J})\mu(\tilde{I}^{u_0,u_1}_{v_0}),
   \end{align*}
which allows us to choose $\varepsilon_2 = \frac{1}{n^5}.$ 

Altogether, this shows that $(\Phi_n,\mathrm{d}\Phi_n)$ is $(\frac{1}{2k_n^5q_n},\frac{1}{k_n^4},\frac{1}{n^5},\frac{1}{n^5})$-distributing every element $\hat{I}_{u_0,v_0,j}\in \hat{\eta}_n.$
\end{proof}

\begin{lemma}\label{lemm:6.16}
Consider a sequence of partial partitions $(\hat{\eta}_n)_{n\in\N}$ and $(\tilde{\eta}_n)_{n\in \N}$ of the form as described in (\ref{eqn:partial partition}) and (\ref{eqn:7.1a}), respectively, and the diffeomorphism $g_n$ from Proposition~\ref{lemma:3a}. Furthermore, let $(H_n)_{n\in \N}$ be a sequence of smooth measure-preserving diffeomorphism satisfying the condition \ref{item:P1}. Consider the partial partitions
\begin{align} \label{eqn:partial partition_2}  \hat{\nu}_n= \{\Gamma_{u_0,v_0,j}= (H_{n-1},\mathrm{d} H_{n-1})\circ (g_n,\mathrm{d}g_n)(\hat{I}_{u_0,v_0,j}) \ \ : \ \ \hat{I}_{u_0,v_0,j}\in \hat{\eta}_n\}.
\end{align}
Then $\hat{\nu}_n \to \epsilon.$ 
\end{lemma}
\begin{proof}
By construction $\hat{\eta}_n$ is a partial partition of $\mathbb{P} \mathrm{T}M$ and $\hat{\eta}_n \to \epsilon$, i.e. for every $n\in \N$, $\hat{\eta}_n$ consists of countable disjoint sets and covers a set of almost full measure, $\lim\limits_{n\rightarrow \infty} \bar{\mu} \left(\bigcup \limits_{(u_0,v_0,j)\in \Lambda_n} \hat{I}_{u_0,v_0,j}\right) = 1$, where $\Lambda_n$ denotes the set of permissible indices in (\ref{eqn:partial partition}). By definition, an element $\hat{I}_{u_0,v_0,j} \in \hat{\eta}_n$ is of the form $\hat{I}_{u_0,v_0,j} = \cup_{u_1 =0}^{k_n-1}\tilde{I}_{v_0}^{u_0,u_1}\times T_j$. 
Denote $\Gamma_{u_0,v_0,j} = \cup_{u_1 =0}^{k_n-1} \Gamma_{v_0,j}^{u_0,u_1}$ and $\Gamma_{v_0,j}^{u_0,u_1} =(H_{n-1},\mathrm{d} H_{n-1})\circ (g_n,\mathrm{d}g_n)(\tilde{I}_{v_0}^{u_0,u_1} \times T_j).$ 
Using the estimates in Lemma~\ref{lemm~measurepreservation}, we can build on Remark \ref{rem:7.7.1} to conclude that the collection of elements $\Gamma_{u_0,v_0,j}\in \hat{\nu}_n$ covers a set of almost full measure:
\begin{align*}
\lim_{n\rightarrow \infty} \bar{\mu} \left(\bigcup_{u_0,v_0,j\in \Lambda_n} \Gamma_{u_0,v_0,j}\right) &= \lim_{n\rightarrow \infty} \sum_{u_0,v_0,j \in \Lambda_n} \sum_{u_1 =0}^{k_n-1}\bar{\mu} \left((H_{n-1},\mathrm{d} H_{n-1})\circ (g_n,\mathrm{d}g_n)(\tilde{I}_{v_0}^{u_0,u_1}\times T_j )\right) \nonumber \\ 
& \geq \lim_{n\rightarrow \infty} \sum_{u_0,v_0,j \in \Lambda_n} \sum_{u_1 =0}^{k_n-1} \left(1-\frac{36}{k_n^4}\right)\cdot\frac{1}{k_n}\mu(\tilde{I}_{v_0}^{u_0,u_1})\\
& \geq \lim_{n\rightarrow \infty} \left(1-\frac{36}{k_n^4}\right) \mu\left(\bigcup_{u_0,v_0,j\in \Lambda_n} \bigcup_{u_1 =0}^{k_n-1}\tilde{I}_{v_0}^{u_0,u_1}\right)\\
& \geq \lim_{n\rightarrow \infty} \left(1-\frac{36}{k_n^4}\right) \left(1-\frac{25}{k_n^3}\right) \geq \lim_{n\rightarrow \infty} \left(1-\frac{80}{k_n^3}\right).
\end{align*}
Further note in Proposition \ref{lemma:3a} that the map $g_n$ acts as the translation in the $\theta$-direction on $\T^1 \times [\frac{j+2\varepsilon_n}{a_n}, \frac{j+1- 2\varepsilon_n}{a_n}]$ by a factor $[nq_n^{\sigma}]\cdot\frac{j}{k_n^5}$ and its action is trivial in the tangent direction. 
Since $\tilde{I}_{v_0}^{u_0,u_1}\subset \mathcal{G}_n$, it follows that each $g_n(\tilde{I}_{v_0}^{u_0,u_1})$ is contained in a rectangle of $\theta$ width $\frac{1}{2k_n^6q_n} + \frac{[nq_n^{\sigma}]}{k_n^5}$ and height $\frac{1}{k_n^5}.$ Hence, for $\sigma < 0.5,$  and  using the fact that $\mathrm{d}g_n$ acts as the identity function on the spherical coordinates with base points $x\in \mathcal{G}_n$ (see (\ref{eqn:gooddomain_g_n})), we have $\text{diam}((g_n, \mathrm{d}{g_n})(\tilde{I}_{v_0}^{u_0,u_1} \times T_j))\leq \frac{3}{k_n}.$
Combining this with condition \ref{item:P1}, we can conclude that
$$\text{diam}\left((H_{n-1},\mathrm{d} H_{n-1})\circ (g_n,\mathrm{d}g_n)(\tilde{I}_{v_0}^{u_0,u_1} \times T_j)\right) < \frac{1}{n^2}.$$ 
Therefore, $\text{diam}{(\Gamma_{v_0,j}^{u_0,u_1})} \to 0$ for all $\Gamma_{v_0,j}^{u_0,u_1}\in \hat{\nu}_n$ as $n \to \infty$ and consequently $\hat{\nu}_n\to \epsilon$.
\end{proof}
\begin{lemma}\label{lem:partial_partition2}
Consider the sequence of maps $H_{n}, g_n, \hat{H}_{n},$ and $\hat{g}_n$ as defined in Sections \ref{sec:3.3.1} and \ref{sec:4:4.a} respectively.
Consider the partial partitions $\hat{\nu}_n$ and $\hat{\hat{\nu}}_n,$ defined by equation (\ref{eqn:partial partition_2}), and
\begin{align}\label{eqn:partial partition_3}
\hat{\hat{\nu}}_n= \left\{(\hat{H}_{n-1},\mathrm{d} \hat{H}_{n-1})\circ (\hat{g}_n,d\hat{g}_n)(\hat{I}_{n,j}) \ \ : \ \ \hat{I}_{n,j}\in \hat{\eta}_n\right\}.
\end{align}
If $\hat{\nu}_n \to \epsilon$, then $\hat{\hat{\nu}}_n \to \epsilon.$   
\end{lemma}
\begin{proof}
The proof proceeds along the same lines as Lemma~\ref{lemm:6.16}, employing the estimate $d_k(g_n,\hat{g}_n)\leq \epsilon_n$, and for every $\tilde{I}_{v_0}^{u_0,u_1}\in \tilde{\eta}_n$:
\begin{align}
    | \text{diam}(\hat{H}_{n-1}\circ \hat{g}_n(\tilde{I}_{v_0}^{u_0,u_1}) )- \text{diam}({H}_{n-1}\circ {g}_n(\tilde{I}_{v_0}^{u_0,u_1}))| &\leq d_1(\hat{H}_{n-1}\circ \hat{g}_n \restriction_{\T^2}, {H}_{n-1}\circ {g}_n) \nonumber\\
    & \leq (1+ \vertiii{H_{n-1}}_0)\cdot \epsilon_n  \leq \frac{1}{2^n}.\nonumber
\end{align}
Therefore, for every $\tilde{I}_{v_0}^{u_0,u_1}\in \hat{\eta}_n,$ we have $\text{diam}({H}_{n-1}\circ {g}_n(\tilde{I}_{v_0}^{u_0,u_1}) ) \rightarrow 0$, which in turn implies that $\text{diam}(\hat{H}_{n-1}\circ \hat{g}_n(\tilde{I}_{v_0}^{u_0,u_1}) ) \rightarrow 0$ as $n\rightarrow \infty$, leading to $\hat{\hat{\nu}}_n \to \epsilon.$ 
\end{proof}

\subsection{Application of the criterion for weak mixing: Proof of Theorems \ref{thm:1.1} and \ref{thm:1.2}}

\begin{proof}[Proof of Theorem \ref{thm:1.1}]
For any Liouville number $\alpha \in \mathbb{R}$ we can consider a sequence of diffeomorphisms $f_n$ constructed by equations (\ref{eq:3a}) and (\ref{eq:3:3b}) with explicit conjugation maps as in Section \ref{sec:conj_g_n} and \ref{sec:phi}. Lemma \ref{la:1b} and \ref{lem:4.4.1a} guarantee the convergence of the sequence $(f_n)_{n\in \N}$ to a limit diffeomorphism $f\in \mathcal{A}_{\alpha}(\mathbb{T}^2)$ and the proximity condition { $d_1(f_n, f) < \frac{1}{2^n}$} for all $n \in \N$. 

To establish the weak mixing property of the limiting map $(f, \mathrm{d}f)$ with respect to the invariant measure $\bar{\mu}$ we will apply Proposition \ref{prop:6.6.3a}. We consider the sequence of the decomposition $\hat{\eta}_n$ and $\hat{\nu}_n$, and the associated natural numbers $(m_n)$ as described in equations (\ref{eqn:partial partition}), (\ref{eqn:partial partition_2}), and Section (\ref{sec:5.1a}), respectively. Note that $\hat{\nu}_n \rightarrow \epsilon$ by Remark (\ref{rem:7.7.1}), $\hat{\eta}_n \rightarrow \epsilon$ by Lemma \ref{lemm:6.16}, and the diffeomorphism $(\Phi_n,\mathrm{d} \Phi_n)$ is $(\frac{1}{2k_n^5q_n}, \frac{1}{k_n^4}, \frac{1}{n^5}, \frac{1}{n^5})$-distributing every $\hat{I}_n \in \hat{\eta}_n$, as verified Lemma \ref{lem:7:7.3}. 

Hence, all the conditions of Proposition \ref{prop:6.6.3a} are satisfied and, consequentially, the map $(f,\mathrm{d}f)$ is weakly mixing with respect to the measures $\bar{\mu}$. Similarly, employing the same combinatorics and the construction of conjugation maps with minor modifications, we can replicate the same construction for any other compact two-dimensional connected manifold $M$ admitting a nontrivial circle action. For more details, we refer to \cite[Section 2.4]{FS}.
\end{proof}

\begin{proof}[Proof of Theorem \ref{thm:1.2}]

The proof follows along the same lines as in the smooth case. Lemma \ref{lem:4:4.2b} ensures the convergence of the sequence of analytic measure-preserving diffeomorphisms $ (\hat{f}_n)_{n\in \N} $ constructed via equations (\ref{eqn:4:4.1a}) and (\ref{eqn:4:4.1b}). It also provides the proximity condition $ d_1(\hat{f}_n, \hat{f}) < \frac{1}{2^n} $.

The existence of an invariant measure $ \bar{\mu} $ with respect to the limiting map $ (\hat{f}, \mathrm{d}\hat{f}) $ is guaranteed by Subsection \ref{sec:2.2.1} and Remark \ref{rem:5.5.5}, similar to the smooth case.

We then consider the sequence of decompositions $ \hat{\eta}_n, \hat{\nu}_n, $ and $ \hat{\hat{\nu}}_n $, along with the associated natural numbers $ (m_n) $, as described in equations (\ref{eqn:partial partition}), (\ref{eqn:partial partition_2}), (\ref{eqn:partial partition_3}), and Section \ref{sec:5.1a}, respectively. It is noted that $ \hat{\eta}_n \rightarrow \epsilon $, $ \hat{\nu}_n \rightarrow \epsilon $, and $ \hat{\hat{\nu}}_n \rightarrow \epsilon $, as outlined in Remark \ref{rem:7.7.1} and Lemmas \ref{lemm:6.16} and \ref{lem:partial_partition2}. Referring to Lemma \ref{lem:6.6.4c} and Lemma \ref{lem:7:7.3}, we deduce that for any $ \hat{I}_n \in \hat{\eta}_n $, the diffeomorphism $ (\hat{\Phi}_n, \mathrm{d} \hat{\Phi}_n) $ is $ (\gamma', \delta', \varepsilon_1', \varepsilon_2') $-distributing $\hat{I}_n$, where
$\gamma' = \gamma + \frac{1}{2^n} \leq \frac{1}{k_n^5q_n}$, $\delta' = \delta + \frac{1}{2^n} \leq \frac{2}{k_n^4}$,  and $\varepsilon_1' = \varepsilon_2' = \frac{2}{n^5} + \frac{3}{2^n} \leq \frac{1}{n}$.

Thus, all the requirements of Proposition \ref{prop:6.6.3b} are fulfilled. Its application yields that the map $ (\hat{f}, \mathrm{d}\hat{f}) $ is weakly mixing with respect to the measure $ \bar{\mu} $.
\end{proof}

\bibliographystyle{abbrv}
\bibliography{Ref}
\end{document}